\documentclass{siamart190516}
\usepackage{amsfonts}
\usepackage{amsfonts,amsmath,amssymb}
\usepackage{mathrsfs,mathtools,stmaryrd,wasysym}
\usepackage{enumerate}
\usepackage{enumitem}
\usepackage{esint}
\usepackage{graphicx,psfrag}
\usepackage{hyperref}
\usepackage{amsmath,stackengine}
\usepackage{yfonts}
\usepackage{booktabs}
\DeclareMathAlphabet{\mathpzc}{OT1}{pzc}{m}{it}
\renewcommand{\theequation}{\thesection.\arabic{equation}}

\newcommand{\bH}{\boldsymbol{H}}
\newcommand{\bE}{\boldsymbol{E}}

\newcommand{\bp}{\boldsymbol{p}}
\newcommand{\bv}{\boldsymbol{v}}

\newcommand{\bn}{\boldsymbol{n}}

\newcommand{\bw}{\boldsymbol{w}}
\newcommand{\by}{\boldsymbol{y}}
\newcommand{\bz}{\boldsymbol{z}}

\newcommand{\bbf}{\boldsymbol{f}}

\newcommand{\R}{\mathbb{R}}

\newcommand{\rW}{\mathrm W}
\newcommand{\rL}{\mathrm L}
\newcommand{\brL}{\boldsymbol{\rL}}
\newcommand{\T}{\mathscr{T}}

\newcommand{\bsiep}{\varepsilon_{\sigma}}

\newcommand{\calJ}{\mathcal{J}}

\newcommand{\jump}[1]{\llbracket #1 \rrbracket}

\newcommand{\cu}{\mathop{\mathbf{curl}}\nolimits}
\newcommand{\di}{\mathop{\mathrm{div}}\nolimits}

\newsiamremark{remark}{Remark}
\newcommand{\TheTitle}{Error estimates for a bilinear optimal control problem of Maxwell's equations}
\newcommand{\ShortTitle}{Estimates for a control problem of Maxwell's equations}
\newcommand{\TheAuthors}{F. Fuica, F. Lepe, P. Venegas}

\headers{\ShortTitle}{\TheAuthors}

\title{{\TheTitle}\thanks{FF is supported by ANID through FONDECYT postdoctoral project 3230126. 
FL was partially supported by ANID through FONDECYT Project 11200529. 
PV is partially supported by ANID through FONDECYT Project 1211030 and Centro de Modelamiento
Matemático (CMM), grant FB210005, BASAL funds for centers of excellence.}}

\author{Francisco Fuica\thanks{Facultad de Matem\'aticas, Pontificia Universidad Cat\'olica de Chile, Avenida Vicu\~{n}a Mackenna 4860, Santiago, Chile.
(\email{francisco.fuica@uc.cl}).}
\and
Felipe Lepe\thanks{GIMNAP, Universidad del B\'io B\'io, Concepci\'on, Chile. (\email{flepe@ubiobio.cl})}
\and
Pablo Venegas\thanks{GIMNAP, Universidad del B\'io B\'io, Concepci\'on, Chile. (\email{pvenegas@ubiobio.cl})}
}

\ifpdf
\hypersetup{
  pdftitle={\TheTitle},
  pdfauthor={\TheAuthors}
}
\fi

\date{Draft version of \today.}

\begin{document}

\maketitle

\begin{abstract}
We consider a control-constrained optimal control problem subject to time-harmonic Maxwell's equations; the control variable belongs to a finite-dimensional set and enters the state equation as a coefficient.  
We derive existence of optimal solutions, and analyze first- and second-order optimality conditions. 
We devise an approximation scheme based on the lowest order N\'ed\'elec finite elements to approximate optimal solutions. 
We analyze convergence properties of the proposed scheme and prove a priori error estimates. 
We also design an a posteriori error estimator that can be decomposed as the sum two contributions related to the discretization of the state and adjoint equations, and prove that the devised error estimator is reliable and locally efficient.
We perform numerical tests in order to assess the performance of the devised discretization strategy and the a posteriori error estimator.
\end{abstract}

\begin{keywords}
optimal control, time-harmonic Maxwell's equations, first- and second-order optimality conditions, finite elements, convergence, error estimates.
\end{keywords}

% REQUIRED
\begin{AMS}
35Q60,         % PDEs in connection with optics and electromagnetic theory
49J20,   	   % Existence theories for optimal control problems involving partial differential equations,
49K20,         % Optimality conditions for problems involving partial differential equations
49M25,		   % Discrete approximations in optimal control
65N15,         % Error bounds for boundary value problems involving PDEs
65N30.         % Finite element, Rayleigh-Ritz and Galerkin methods for boundary value problems involving PDEs
\end{AMS}

\section{Introduction}
In this work we focus our study on existence of solutions, optimality conditions, and a priori and a posteriori error estimates for an optimal control problem that involves time-harmonic Maxwell's equations  as state equation and a finite dimensional control space. More precisely, let $\Omega \subset \R^3$ be an open, bounded, and simply connected polyhedral domain with Lipschitz boundary $\Gamma$. Given a control cost $\alpha > 0$, desired states $\by_\Omega \in  \brL^2(\Omega;\mathbb{C})$ and $\bE_{\Omega}\in \brL^2(\Omega;\mathbb{C})$, and $\ell\in \mathbb{N}$, we define the cost functional
\begin{align}\label{eq:cost_function}
\calJ(\by,\mathbf{u}):=\dfrac{1}{2}\|\by-\by_\Omega\|^2_{\brL^2(\Omega;\mathbb{C})} + \dfrac{1}{2}\|\cu \by -\bE_\Omega\|^2_{\brL^2(\Omega;\mathbb{C})}
+\dfrac{\alpha}{2}\|\mathbf{u}\|_{\mathbb{R}^{\ell}}^2.
\end{align}
Let $\bbf \in \brL^2(\Omega;\mathbb{C})$ be an externally imposed source term, let $\mu\in \rL^\infty(\Omega)$ be a function satisfying $\mu \geq \mu_0 > 0$ with $\mu_0\in \mathbb{R}^{+}$, and let $\omega>0$ be a constant representing the angular frequency. 
Given a function $\bsiep\in \rL^{\infty}(\Omega;\mathbb{C})$, we will be concerned with the following optimal control problem: Find $\min \calJ(\by,\mathbf{u})$ subject to 
\begin{align}\label{eq:strong_state_eq}
\cu \mu^{-1}\cu \by  - \omega^{2}(\bsiep\cdot \mathbf{u})\by= \bbf \quad \mbox{in } \Omega, \qquad
\by\times \bn= \mathbf{0} \quad \mbox{on } \Gamma,
\end{align}
and the control constraints 
\begin{align}\label{def:box_constraints}
\mathbf{u}=(\mathbf{u}_{1},\ldots,\mathbf{u}_{\ell})\in U_{ad}, \qquad U_{ad}:=\left\{\mathbf{v} \in \mathbb{R}^{\ell}: \mathbf{a}\leq \mathbf{v}\leq \mathbf{b}\right\}.
\end{align}
Here, the control bounds $\mathbf{a},\mathbf{b}\in \mathbb{R}^{\ell}$ are such that $\mathbf{0} < \mathbf{a} < \mathbf{b}$. 
We immediately point out that, throughout this work, vector inequalities must be understood componentwise. In \eqref{eq:strong_state_eq}, $\boldsymbol{n}$ denotes the outward unit normal. 
In an abuse of notation, we use $\bsiep\cdot \mathbf{u}$ to denote $\sum_{k=1}^{\ell}\bsiep|_{\Omega_{k}}^{}\mathbf{u}_{k}$, where $\{\Omega_{k}\}_{k=1}^{\ell}$ is a given partition of $\Omega$ (see section \ref{sec:partition_fields}). Further details on $\bsiep$ will be deferred until section \ref{sec:model_problem}.

%%%%

Time-harmonic Maxwell's equations are given by the system of first-order partial differential equations:
\begin{align}\label{eq:strong_Maxwell}
\cu \by - i\omega\mu \boldsymbol{h} \! = \!\boldsymbol{0}, \quad \cu\boldsymbol{h}+i\omega \varepsilon \by \! = \! \boldsymbol{j}, \quad
\text{div}(\varepsilon\by) \! = \! \rho, ~\text{and}~~
\text{div}(\mu \boldsymbol{h}) \! = \! 0, ~  \text{in }\Omega, \hspace{-0.3cm}
\end{align}
where $\by$ is the electric field, $\boldsymbol{h}$ is the magnetic field, $\varepsilon$ is the real-valued electrical permittivity of the material, $\mu$ is the real-valued magnetic permeability, and the source terms $\boldsymbol{j}$ and $\rho$ are the current density and the charge
density, respectively, which are related by the charge conservation equation $-i\omega\rho + \text{div}\,\boldsymbol{j} = 0$. We assume that $\boldsymbol{j} = \hat{\boldsymbol{j}} + \sigma\by$, where $\hat{\boldsymbol{j}}$ is an externally imposed current and the real-valued coefficient $\sigma$ is the conductivity. In addition, we assume that the medium $\Omega$ is surrounded by a perfect conductor, so that we have the boundary condition $\by \times \boldsymbol{n} = 0$ on $\partial\Omega$. In particular, for a detailed derivation of problem \eqref{eq:strong_state_eq} from \eqref{eq:strong_Maxwell}, we refer the reader to \cite[section 2]{CiarletJr2020}; see also \cite[section 8.3.2]{MR3793186}. We notice that, for simplicity, we have considered $\bbf = i\omega\hat{\boldsymbol{j}}$.

%%%%

Optimal control problems subject to Maxwell's and eddy current equations have been widely studied over the last decades, due to their strong relationship with physics and engineering.
We refer the interested reader to the following non-comprehensive list of references concering numerical methods for their approximation, namely, a priori and a posteriori error estimates: \cite{MR2661664, MR2891467, MR2870896, MR3105785, MR3259029, MR3460108, MR3515109, MR3745168, MR3679913, MR3723330, MR4090837, MR4518560, 2022arXiv220915129A}. 
In all these references, the control enters the state equation as a source term. 
When the control enters the state equation as coefficient, as in \eqref{eq:strong_state_eq}, the analysis becomes more challenging due to the \emph{nonlinear} coupling between the state and control variables; this coupling has led to this type of problems being referred to as \emph{bilinear optimal control problems}. 
The aforementioned coupling complicates both the analysis and discretization, since the state variable depends nonlinearly on the control and, consequently, the uniqueness of solutions of \eqref{eq:cost_function}--\eqref{def:box_constraints} cannot be guaranteed. 
Hence, a proper optimization study requires the analysis of second-order optimality conditions. 

%%%%

Regarding bilinear optimal control problems subject to Maxwell's and eddy current equations, we mention \cite{MR2957021,MR3377426,MR4054226}. 
In \cite{MR2957021}, the author studied an optimal control problem governed by the time-harmonic eddy current equations, where the controls (scalar functions) entered as a coefficient in the state equation. 
After analyzing regularity results, existence of optimal controls, and first-order optimality conditions, the author proposed a discretization strategy and prove, assuming that the optimal controls belongs to $\mathrm{W}^{1,\infty}(\Omega)$, convergence results of such finite element discretization without a rate; second-order optimality conditions were not provided.
 Similarly, in \cite{MR3377426}, the author introduced  an optimal control approach based on grad-div regularization and divergence penalization for the problem previously studied in \cite{MR2957021}.
 However, due to the lack of regularity of controls, no discretization analysis was given.
  In \cite{MR4054226}, the authors studied an optimal control problem with controls as coefficients of time-harmonic Maxwell’s equations, with applications to invisibility cloak design. 
  The controls represented the permittivity and permeability of the metamaterial. 
  After presenting first-order optimality conditions using the Lagrange multiplier methodology, the authors solve the state equation with the discontinuous Galerkin method and presented numerical tests to demonstrate the effectiveness of the proposed method.

%%%%

In contrast to \cite{MR2957021,MR3377426}, besides considering Maxwell's equations instead of eddy current equations, in our work the control corresponds to a vector acting on both the electrical permittivity and conductivity of the material $\Omega$, in a given partition. 
This implies that conductivity may change in different regions of $\Omega$. 
This is a plausible consideration on the conductivity in applications, since some devises that conduct electricity are designed with different materials and hence, with different conductivity properties.  
In this manuscript, we provide existence of optimal solutions and necessary and sufficient optimality conditions. 
Then, we propose an approximation scheme based on N\'ed\'elec finite elements and present a priori error estimates for the state equations which, in turn, allow us to prove that continuous strict local solutions of the control problem can be approximated by local minima of suitable discrete problems. 
Moreover, under appropriate assumptions on the adjoint equation (see assumptions \eqref{eq:assumption_p-ph} and \eqref{eq:assumption_p-ph_curl}), we provide a priori error estimates and convergence rates between continuous and discrete optimal solutions. 
The aforementioned assumptions, which follow from the reduced regularity properties of the adjoint variable, motivate the development and analysis of adaptive finite element methods \cite{MR1885308,MR3059294} for the proposed control problem. With this in mind, we propose a residual-type a posteriori error estimator for the control problem and prove its reliability and local efficiency; the error estimator is built as the sum two contributions related to the discretization of the state and adjoint equations.
 Moreover, it can be used to drive adaptive procedures and is capable to attain optimal order of convergence for the approximation error by refining in the regions where singularities may appear. 
Finally, we mention that our problem also can be seen as an identification parameter problem for Maxwell's equations. On this matter, we refer the reader to \cite{MR4109590} and the recent article \cite{MR4724154}.

%%%%

We organize our manuscript as follows. Section \ref{sec:preliminaries} is devoted to set notation and basic definitions that we will use throughout our work.
In section \ref{sec:review_bilinear}, basic results for the state equation as well as a priori and posteriori error estimates are reviewed. 
The core of our paper begins in section \ref{sec:the_ocp}, where the analysis of the optimal control problem is performed. 
To make matters precise, in this section we prove existence of optimal solutions for the considered problem and study first- and second-order optimality conditions. 
In section \ref{sec:fem_for_ocp} a suitable finite element discretization of the optimal control problem is proposed and its corresponding convergence properties are proved.
Moreover, we propose an a posteriori error estimator for the designed finite element scheme and show reliability and local efficiency properties. 
We end our exposition with a series of numerical tests reported in section \ref{sec:num_ex}.

%%%%%%%%%%%%%%%%%%%%%%%%%%%%%%%%%%%%%%%%%%%%%%%%%%%
%%%%%%%%%%%%%%%%%%%%%%%%%%%%%%%%%%%%%%%%%%%%%%%%%%%
%%%%%%%%%%%%%%%%%%%%%%%%%%%%%%%%%%%%%%%%%%%%%%%%%%%
%%%%%%%%%%%%%%%%%%%%%%%%%%%%%%%%%%%%%%%%%%%%%%%%%%%
%%%%%%%%%%%%%%%%%%%%%%%%%%%%%%%%%%%%%%%%%%%%%%%%%%%
%%%%%%%%%%%%%%%%%%%%%%%%%%%%%%%%%%%%%%%%%%%%%%%%%%%
%%%%%%%%%%%%%%%%%%%%%%%%%%%%%%%%%%%%%%%%%%%%%%%%%%%
%%%%%%%%%%%%%%%%%%%%%%%%%%%%%%%%%%%%%%%%%%%%%%%%%%%

\section{Notation and preliminaries}
\label{sec:preliminaries}

%%%%%%%%%%%%%%%%%%%%%%%%%%%%%%%%%%%%%%%%%%%%%%%%%%%
%%%%%%%%%%%%%%%%%%%%%%%%%%%%%%%%%%%%%%%%%%%%%%%%%%%
%%%%%%%%%%%%%%%%%%%%%%%%%%%%%%%%%%%%%%%%%%%%%%%%%%%
%%%%%%%%%%%%%%%%%%%%%%%%%%%%%%%%%%%%%%%%%%%%%%%%%%%

\subsection{Notation}

Throughout the present manuscript, we use standard notation for Lebesgue and Sobolev spaces and their norms.
We use uppercase bold letters to denote the vector-valued counterparts of the aforementioned spaces whereas lowercase bold letters are used to denote vector-valued functions. In particular, we define 
\begin{align*}
\mathbf{H}(\textnormal{div},\Omega) &:= \left\{ \bw \in \brL^2(\Omega;\mathbb{C}): \text{div } \bw \in \rL^2(\Omega;\mathbb{C})\right\}, \\
\mathbf{H}(\cu,\Omega)&:= \left\{ \bw \in \brL^2(\Omega;\mathbb{C}): \cu  \bw \in \brL^2(\Omega;\mathbb{C})\right\}, 
\end{align*}
and $\mathbf{H}_{0}(\cu,\Omega):= \left\{ \bw \in \mathbf{H}(\cu,\Omega): \bw \times \bn = \mathbf{0} \right\}$. In addition, given $s \geq 0$, we introduce the space $\mathbf{H}^{s}(\cu,\Omega):= \left\{ \bw \in \mathbf{H}^{s}(\Omega;\mathbb{C}): \cu  \bw \in \mathbf{H}^{s}(\Omega;\mathbb{C})\right\}$. 

If $\mathcal{X}$ is a normed vector space, we denote by $\mathcal{X}'$ and $\|\cdot\|_{\mathcal{X}}$ the dual and the norm of $\mathcal{X}$, respectively. 
We denote by $\langle \cdot,\cdot \rangle_{\mathcal{X}',\mathcal{X}}$ the duality pairing between $\mathcal{X}'$ and $\mathcal{X}$. 
When the spaces $\mathcal{X}'$ and $\mathcal{X}$ are clear from the context, we simply denote the duality pairing $\langle \cdot,\cdot \rangle_{\mathcal{X}',\mathcal{X}}$ by $\langle \cdot,\cdot \rangle$. 
For the particular case $\mathcal{X}=\brL^2(G;\mathbb{C})$, with $G\subset \mathbb{R}^{3}$ a bounded domain, we shall denote its inner product and norm by $(\cdot,\cdot)_{G}$ and $\|\cdot\|_{G}$, respectively. Given a complex function $\bw$, we denote by $\overline{\bw}$ its complex conjugate.

The relation $\mathfrak{a} \lesssim \mathfrak{b}$ indicates that $\mathfrak{a} \leq C \mathfrak{b}$, with a constant $C > 0$ that does not depend on either $\mathfrak{a}$, $\mathfrak{b}$, or discretization parameters. The value of the constant $C$ might change at each occurrence.

%%%%%%%%%%%%%%%%%%%%%%%%%%%%%%%%%%%%%%%%%%%%%%%%%%%
%%%%%%%%%%%%%%%%%%%%%%%%%%%%%%%%%%%%%%%%%%%%%%%%%%%
%%%%%%%%%%%%%%%%%%%%%%%%%%%%%%%%%%%%%%%%%%%%%%%%%%%
%%%%%%%%%%%%%%%%%%%%%%%%%%%%%%%%%%%%%%%%%%%%%%%%%%%

\subsection{Piecewise smooth fields}\label{sec:partition_fields}

Let $\ell\in\mathbb{N}$. The set $\mathcal{P}:=\{\Omega_{k}\}_{k=1}^{\ell}$ is called a \emph{partition} of $\Omega$ if any two elements do not intersect and $\overline{\Omega} = \cup_{k=1}^{\ell}\overline{\Omega}_{k}$.
The corresponding interface is defined by $\Sigma:=\cup_{1\leq k \neq k'\leq \ell} (\Gamma_{k}\cap \Gamma_{k'})$, where $\Gamma_{k}$ and $\Gamma_{k'}$ denote the boundaries of $\Omega_{k}$ and $\Omega_{k'}$, respectively. With this partition at hand, we define 
\begin{align*}\label{def:piecewise_W100}
P\rW^{1,\infty}(\Omega):= \{\zeta\in \rL^{\infty}(\Omega;\mathbb{C}) : \zeta|^{}_{\Omega_{k}} \in \rW^{1,\infty}(\Omega_{k};\mathbb{C}), ~ 1\leq k \leq \ell\}.
\end{align*}

%%%%%%%%%%%%%%%%%%%%%%%%%%%%%%%%%%%%%%%%%%%%%%%%%%%
%%%%%%%%%%%%%%%%%%%%%%%%%%%%%%%%%%%%%%%%%%%%%%%%%%%
%%%%%%%%%%%%%%%%%%%%%%%%%%%%%%%%%%%%%%%%%%%%%%%%%%%
%%%%%%%%%%%%%%%%%%%%%%%%%%%%%%%%%%%%%%%%%%%%%%%%%%%
%%%%%%%%%%%%%%%%%%%%%%%%%%%%%%%%%%%%%%%%%%%%%%%%%%%
%%%%%%%%%%%%%%%%%%%%%%%%%%%%%%%%%%%%%%%%%%%%%%%%%%%
%%%%%%%%%%%%%%%%%%%%%%%%%%%%%%%%%%%%%%%%%%%%%%%%%%%
%%%%%%%%%%%%%%%%%%%%%%%%%%%%%%%%%%%%%%%%%%%%%%%%%%%

\section{The state equation}
\label{sec:review_bilinear}

In this section, we review well-posedness results for \eqref{eq:strong_state_eq} and further regularity properties for its solution. 
Additionally, we present a priori and a posteriori error estimates for a specific finite element setting. 

%%%%%%%%%%%%%%%%%%%%%%%%%%%%%%%%%%%%%%%%%%%%%%%%%%%
%%%%%%%%%%%%%%%%%%%%%%%%%%%%%%%%%%%%%%%%%%%%%%%%%%%
%%%%%%%%%%%%%%%%%%%%%%%%%%%%%%%%%%%%%%%%%%%%%%%%%%%
%%%%%%%%%%%%%%%%%%%%%%%%%%%%%%%%%%%%%%%%%%%%%%%%%%%

\subsection{The model problem}\label{sec:model_problem}

Let $\mathbf{f}\in \mathbf{H}_{0}(\cu,\Omega)'$ be a given forcing term, let $\mu\in \rL^\infty(\Omega)$ be such that $\mu \geq \mu_0 > 0$ with $\mu_0\in \mathbb{R}^{+}$, let $\mathfrak{u}\in U_{ad}$, and let $\omega\in \mathbb{R}^{+}$. 
We introduce the electric permittivity $\varepsilon\in \rL^{\infty}(\Omega)$ and the conductivity $\sigma\in \rL^{\infty}(\Omega)$ of the material $\Omega$, and assume that there exist $\varepsilon_{+},\varepsilon^{+}\in\mathbb{R}^{+}$ and $\sigma_{+},\sigma^{+}\in\mathbb{R}^{+}$ such that
\begin{align*}
\varepsilon_{+} \leq \varepsilon \leq \varepsilon^{+} \quad 
\text{and} \quad
\sigma_{+} \leq \sigma \leq \sigma^{+}.
\end{align*}
We define $\bsiep:= \varepsilon + i\sigma\omega^{-1}$ and consider the problem: Find $\mathbf{y} \in \mathbf{H}_{0}(\cu,\Omega)$ such that
\begin{equation}\label{eq:weak_eq}
(\mu^{-1} \cu \mathbf{y}, \cu \bw)_{\Omega} - \omega^{2}((\bsiep\cdot \mathfrak{u})\mathbf{y},\bw)_{\Omega} = \langle \mathbf{f},\bw \rangle
\quad \forall \bw \in \mathbf{H}_{0}(\cu,\Omega).
\end{equation}
We recall that $\bsiep\cdot \mathbf{u}$ denotes $\sum_{k=1}^{\ell}\bsiep|_{\Omega_{k}}^{}\mathbf{u}_{k}$, where $\mathcal{P}=\{\Omega_{k}\}_{k=1}^{\ell}$ is a given partition of $\Omega$; see section \ref{sec:partition_fields}.
This problem is well posed \cite[Theorem 8.3.5]{MR3793186}. In particular, we have the stability bound $\|\mathbf{y}\|_{\mathbf{H}(\cu,\Omega)}\lesssim \|\mathbf{f}\|_{\mathbf{H}_{0}(\cu,\Omega)'}$. 
%\PV{For the sake of mathematical generality, the term to be controlled will be related to the complex variable $\varepsilon_\sigma$, however the analysis proposed below can be adapted to take into consideration the real variables $\varepsilon$ or $\sigma$, as long as the associated bilinear form is coercive.}

The next result states further regularity properties for the solution of \eqref{eq:weak_eq}.

\begin{theorem}[extra regularity]\label{thm:extra_reg_Maxwell}
Let $\mathbf{y} \in \mathbf{H}_{0}(\cu,\Omega)$ be the unique solution to problem \eqref{eq:weak_eq}. Then,

$\mathrm{(i)}$ if $\mathbf{f}\in \mathbf{H}(\textnormal{div},\Omega)$ and $\bsiep,\mu\in P\rW^{1,\infty}(\Omega)$, there exists $\mathfrak{t}\in(0,\tfrac{1}{2})$ such that $\mathbf{y}\in \mathbf{H}^{s}(\cu,\Omega)$ for all $s\in[0, \mathfrak{t})$,

$\mathrm{(ii)}$ if $\mathbf{f}\in \mathbf{H}(\textnormal{div},\Omega)$ and $\bsiep,\mu\in \rW^{1,\infty}(\Omega)$, there exists $\epsilon > 0$ such that $\mathbf{y} \in \mathbf{H}_{0}(\cu,\Omega)\cap \mathbf{H}^{\frac{1}{2}+\epsilon}(\Omega;\mathbb{C})$. If, in addition, $\Omega$ is convex, we have that $\epsilon=\frac{1}{2}$.
\end{theorem}
\begin{proof}
The first statement stems from \cite[Section 6.4]{CiarletJr2020}, whereas that $\mathrm{(ii)}$ follows from the fact that $\mathbf{y} \in \mathbf{H}_{0}(\cu,\Omega)\cap\mathbf{H}(\textnormal{div},\Omega)$ in combination with the regularity of the potential provided in \cite[Proposition 3.7 and Theorem 2.17]{MR1626990}.
\end{proof}

%%%%%%%%%%%%%%%%%%%%%%%%%%%%%%%%%%%%%%%%%%%%%%%%%%%
%%%%%%%%%%%%%%%%%%%%%%%%%%%%%%%%%%%%%%%%%%%%%%%%%%%
%%%%%%%%%%%%%%%%%%%%%%%%%%%%%%%%%%%%%%%%%%%%%%%%%%%
%%%%%%%%%%%%%%%%%%%%%%%%%%%%%%%%%%%%%%%%%%%%%%%%%%%

\subsection{Finite element approximation}
\label{sec:fem}
In this section, we present a finite element approximation for problem \eqref{eq:weak_eq} and review basic error estimates. 

We begin by introducing some terminology and further basic ingredients. 
We denote by $\mathscr{T}_{h}=\{T\}$ a conforming partition of $\overline{\Omega}$ into simplices $T$ with size $h_T=\text{diam}(T)$. 
Let us  define $h:=\max_{T\in\mathscr{T}_h}h_T$ and  $\#\T_{h}$ the total number of elements in $\T_{h}$. 
We denote by $\mathbb{T}:=\{\T_{h}\}_{h>0}$ a collection of conforming and shape regular meshes that are refinements of an initial mesh $\mathscr{T}_{\textrm{in}}$. A further requisite for each mesh $\mathscr{T}_{h}\in\mathbb{T}$ is being conforming with the physical partition $\mathcal{P}$ (see section \ref{sec:partition_fields}) \cite[Section 2.4]{MR4456704}: Given $\mathscr{T}_{h}\in \mathbb{T}$, we assume that, for all $T \in  \mathscr{T}_{h}$ there exists $\Omega_{T}\in \mathcal{P}$ such that $T \subset \Omega_{T}$. 
This implies that the interfaces of the partition $\mathcal{P}$ are covered by mesh faces.

Given a mesh $\mathscr{T}_{h}$, we introduce the lowest-order N\'ed\'elec finite element space \cite{Monk}
\begin{align}\label{def:discrete_space}
\mathbf{V}(\T_{h}) :=\{\boldsymbol{v}_h\in\mathbf{H}_0(\cu;\Omega):\boldsymbol{v}_h|_{T} \in \boldsymbol{\mathcal{N}}_0(T)
~ \forall T \in \mathscr{T}_h\},
\end{align}
with $\boldsymbol{\mathcal{N}}_0(T):=[\mathbb{P}_{0}(T)]^3
\oplus \boldsymbol{x}\times[\widetilde{\mathbb{P}}_0(T)]^3$, where $\widetilde{\mathbb{P}}_0(T)$ is the subset of homogeneous polynomials of degree $0$ defined in $T$.

With these ingredients at hand, we introduce the following  Galerkin approximation to problem \eqref{eq:weak_eq}: Find $\mathbf{y}_{h} \in \mathbf{V}(\T_{h})$ such that
\begin{align}\label{eq:discrete_eq}
(\mu^{-1}\cu \mathbf{y}_{h}, \cu \bw_{h})_{\Omega} - \omega^{2}((\bsiep\cdot\mathfrak{u})\mathbf{y}_{h},\bw_{h})_{\Omega} = \langle \mathbf{f},\bw_{h} \rangle \quad \forall \bw_{h} \in \mathbf{V}(\T_{h}).
\end{align}
The existence and uniqueness of a solution $\mathbf{y}_{h} \in \mathbf{V}(\mathscr{T}_{h})$ for problem \eqref{eq:discrete_eq} follows as in the continuous case. We also have that $\|\mathbf{y}_{h}\|_{\mathbf{H}(\cu,\Omega)}\lesssim \|\mathbf{f}\|_{\mathbf{H}_{0}(\cu,\Omega)'}$.

%%%%%%%%%%%%%%%%%%%%%%%%%%%%%%%%%%%%%%%%%%%%%%%%%%%
%%%%%%%%%%%%%%%%%%%%%%%%%%%%%%%%%%%%%%%%%%%%%%%%%%%
%%%%%%%%%%%%%%%%%%%%%%%%%%%%%%%%%%%%%%%%%%%%%%%%%%%
%%%%%%%%%%%%%%%%%%%%%%%%%%%%%%%%%%%%%%%%%%%%%%%%%%%

\subsubsection{A priori error estimates for the model problem}

The following result follows directly from \cite[Theorem 6.15]{CiarletJr2020}.

\begin{theorem}[error estimates]\label{thm:error_estimate}
Let $\mathbf{y}\in \mathbf{H}_{0}(\cu,\Omega)$ and $\mathbf{y}_{h}\in\mathbf{V}(\T_{h})$ be the solutions to \eqref{eq:weak_eq} and \eqref{eq:discrete_eq}, respectively. 
If condition \textnormal{(i)} from Theorem \ref{thm:extra_reg_Maxwell} holds, then we have the a priori error estimate
\begin{align*}
\| \mathbf{y} - \mathbf{y}_{h}\|_{\mathbf{H}(\cu,\Omega)} \lesssim h^{s}\|\mathbf{f}\|_{\mathbf{H}(\textnormal{div},\Omega)},
\end{align*}
where $s\in[0, \mathfrak{t})$ with $\mathfrak{t}$ given as in Theorem \ref{thm:extra_reg_Maxwell}.
\end{theorem}

%%%%%%%%%%%%%%%%%%%%%%%%%%%%%%%%%%%%%%%%%%%%%%%%%%%
%%%%%%%%%%%%%%%%%%%%%%%%%%%%%%%%%%%%%%%%%%%%%%%%%%%
%%%%%%%%%%%%%%%%%%%%%%%%%%%%%%%%%%%%%%%%%%%%%%%%%%%
%%%%%%%%%%%%%%%%%%%%%%%%%%%%%%%%%%%%%%%%%%%%%%%%%%%

\subsubsection{A posteriori error estimate for the model problem}
\label{sec:apost_st_eq}
The aim of this section is to introduce a suitable residual-based a posteriori error estimator for \eqref{eq:weak_eq}. 
We note that, since we will not be dealing with uniform refinement 
 within our a posteriori error analysis setting, the parameter $h$ does not bear the meaning of a \emph{mesh size}. 
It can be thus interpreted as $h = 1/n$, where $n\in\mathbb{N}$ is an index set in a sequence of refinements of an initial mesh $\T_{\textrm{in}}$.

Given $T\in \mathscr{T}_{h}$, $\mathscr{S}_T$ denotes the set of faces of $T$, $\mathscr{S}_T^I$ denotes the set of inner faces of $T$. We also define the set 
\[
\mathscr{S}:=\bigcup_{T\in\mathscr{T}_{h}}\mathscr{S}_T.
\]
We decompose $\mathscr{S}=\mathscr{S}_{\Omega}\cup\mathscr{S}_{\Gamma}$,
where  $\mathscr{S}_{\Gamma}:=\{ S\in \mathscr{S}: S\subset \Gamma\}$
and $\mathscr{S}_{\Omega}:=\mathscr{S}\backslash\mathscr{S}_{\Gamma}$. 
For $T \in \mathscr{T}_{h}$, we define the \emph{star} associated with the element $T$ as
\begin{equation}\label{def:patch}
\mathcal{N}_T:= \left \{ T^{\prime}\in\mathscr{T}_{h}: \mathscr{S}_{T}\cap \mathscr{S}_{T^\prime}\neq\emptyset \right \}.
\end{equation}
In an abuse of notation, below we denote by $\mathcal{N}_T$ either the set itself or the union of its elements. We also introduce, given a vertex $\mathsf{v}$ of an element $T$, the sets $\mathcal{N}_{\mathsf{v}}:=\cup_{T'\in\T : \mathsf{v}\in T'} T'$, $\widetilde{\mathcal{N}}_{\mathsf{v}}:=\cup_{\mathsf{v}'\in \mathcal{N}_{\mathsf{v}}} \mathcal{N}_{\mathsf{v}'}$, and 
\begin{align}\label{def:M_T}
\mathcal{M}_{T}:=\bigcup_{\mathsf{v}\in T}\widetilde{\mathcal{N}}_{\mathsf{v}};
\end{align}
see \cite[Section 2]{MR2373173}. Given $S\in \mathscr{S}_{\Omega}$, we denote by $\mathcal{N}_S \subset \mathscr{T}_{h}$ the subset that contains the two elements that have $S$ as a side, namely, $\mathcal{N}_S:=\{T^+,T^-\}$, where $T^+, T^- \in \mathscr{T}_{h}$ are such that $S = T^+ \cap T^-$. 
Moreover, for any  sufficiently smooth  function
$\boldsymbol{v}$, we define the jump through $S$ by
$$
\llbracket \boldsymbol{v} \rrbracket_{S}(\boldsymbol{x}) 
=
\llbracket \boldsymbol{v} \rrbracket(\boldsymbol{x}):
=
\lim_{t\to 0^{+}}\boldsymbol{v}(\boldsymbol{x} - t\boldsymbol{n}_{T}) - \lim_{t\to 0^{+}}\boldsymbol{v}(\boldsymbol{x} + t\boldsymbol{n}_{T}) \quad	\text{for all } \boldsymbol{x}\in S,
$$
where $\boldsymbol{n}_{T}$ denotes the outer unit normal vector. 

Let $T\in \mathscr{T}_{h}$. We assume that $\mathbf{f}|^{}_{T}\in 
\mathbf{H}^{1}(T;\mathbb{C})$. 
We introduce the local error indicator $\mathcal{E}_T^2:=\mathcal{E}_{T,1}^2+\mathcal{E}_{T,2}^2$, where the local contributions $\mathcal{E}_{T,1}$ and $\mathcal{E}_{T,2}$ are defined by
\begin{align*}
\mathcal{E}_{T,1}^2  
\! :=  &
 h_{T}^2\|\di(\mathbf{f} + \omega^{2}(\bsiep\cdot\mathfrak{u})\mathbf{y}_{h})\|_{T}^2
+\dfrac{h_{T}}{2}\sum_{S\in\mathscr{S}_T^I}\left\|\jump{(\mathbf{f} + \omega^{2}(\bsiep\cdot\mathfrak{u})\mathbf{y}_{h})\cdot\boldsymbol{n}}\right\|_{S}^{2}, \\
\mathcal{E}_{T,2}^2 
\! :=  &
  h_{T}^2 \! \left\|\mathbf{f}-\cu(\mu^{-1}\cu\by_h) + \omega^{2}(\bsiep\cdot\mathfrak{u})\mathbf{y}_{h}\right\|_{T}^2 + \dfrac{h_{T}}{2}\sum_{S\in\mathscr{S}_T^I}\left\|\jump{\mu^{-1}\cu\by_h\times\boldsymbol{n}}\right\|_{S}^{2}\!.
\end{align*}
We thus propose the following global a posteriori error estimator associated to the discretization \eqref{eq:discrete_eq} of problem \eqref{eq:weak_eq}: $\mathcal{E}_{\mathscr{T}_h}^{2} := \sum_{T\in\mathscr{T}_h}\mathcal{E}_T^2$.

We introduce the Sch\"oberl quasi-interpolation operator $\Pi_{h}: \mathbf{H}_{0}(\cu,\Omega) \to \mathbf{V}(\T)$, which satisfies \cite[Theorem 1]{MR2373173}: 
For all $\bw \in \mathbf{H}_{0}(\cu,\Omega)$ there exists $\varphi\in \textrm{H}_0^1(\Omega)$ and $\boldsymbol\Psi\in \mathbf{H}_{0}^{1}(\Omega)$ such that $\bw - \Pi_{h}\bw = \nabla \varphi + \boldsymbol\Psi$, and also satisfy
\begin{equation}\label{eq:property_op_decom}
h_{T}^{-1}\|\varphi\|_{T} + \|\nabla \varphi\|_{T} \lesssim \|\bw\|_{\mathcal{M}_{T}}, \qquad
h_{T}^{-1}\|\boldsymbol\Psi\|_{T} + \|\nabla \boldsymbol\Psi\|_{T} \lesssim \|\textbf{curl}\,\bw\|_{\mathcal{M}_{T}},
\end{equation}
where $\mathcal{M}_{T}$ is defined in \eqref{def:M_T}.

We present the following reliability result and, for the sake of readability, a proof.

\begin{theorem}[global reliability of $\mathcal{E}$]\label{thm:global_reli_weak}
Let $\mathbf{y}\in \mathbf{H}_{0}(\cu,\Omega)$ and $\mathbf{y}_{h}\in\mathbf{V}(\T_{h})$ be the solutions to \eqref{eq:weak_eq} and \eqref{eq:discrete_eq}, respectively. 
If condition \textnormal{(i)} from Theorem \ref{thm:extra_reg_Maxwell} holds, then we have the a posteriori error estimate
\begin{align*}
\| \mathbf{y} - \mathbf{y}_{h}\|_{\mathbf{H}(\cu,\Omega)} \lesssim \mathcal{E}_{\mathscr{T}_h}.
\end{align*}
The hidden constant is independent of $\mathbf{y}$, $\mathbf{y}_{h}$, the size of the elements in $\T_{h}$, and $\#\T_{h}$.
\end{theorem}
\begin{proof}
To simplify the presentation of the material, we define $\mathbf{e}_{\mathbf{y}}:=\mathbf{y} - \mathbf{y}_{h}$. Let $\bw \in \mathbf{H}_{0}(\cu,\Omega)$ be arbitrary.
The use of Galerkin orthogonality in conjunction with the decomposition $\bw - \Pi_{h}\bw = \nabla \varphi + \boldsymbol\Psi$, with $\varphi\in \textrm{H}_0^1(\Omega)$ and $\boldsymbol\Psi\in \mathbf{H}_{0}^{1}(\Omega)$, yield
\begin{align*}
(\mu^{-1} &\cu \mathbf{e}_{\mathbf{y}}, \cu \bw)_{\Omega} - \omega^{2}((\bsiep\cdot\mathfrak{u})\mathbf{e}_{\mathbf{y}},\bw)_{\Omega}
\\
& = (\mathbf{f} + \omega^{2}(\bsiep\cdot\mathfrak{u})\mathbf{y}_{h}, (\bw - \Pi_{h}\bw))_{\Omega} - (\mu^{-1}\cu \mathbf{y}_{h}, \cu (\bw - \Pi_{h}\bw))_{\Omega} 
\\
& = (\mathbf{f}+ \omega^{2}(\bsiep\cdot\mathfrak{u})\mathbf{y}_{h}, \nabla \varphi )_{\Omega} + (\mathbf{f} + \omega^{2}(\bsiep\cdot\mathfrak{u})\mathbf{y}_{h},\boldsymbol\Psi)_{\Omega} - (\mu^{-1}\cu \mathbf{y}_{h}, \cu \boldsymbol\Psi)_{\Omega}.
\end{align*}
Then, applying an elementwise integration by parts formula we obtain
\begin{align}\label{eq:identity_apost}
&(\mu^{-1} \cu \mathbf{e}_{\mathbf{y}}, \cu \bw)_{\Omega}  - \omega^{2}((\bsiep\cdot\mathfrak{u})\mathbf{e}_{\mathbf{y}},\bw)_{\Omega} \\ \nonumber
= 
\!&\sum_{T\in\T_{h}}\!(\mathbf{f}  + \omega^{2}(\bsiep\cdot\mathfrak{u})\mathbf{y}_{h} \!-\! \cu(\mu^{-1}\cu \mathbf{y}_{h}),\boldsymbol\Psi)_{T} -  \sum_{S\in\mathcal{S}}(\llbracket \mu^{-1}\cu \mathbf{y}_{h}\times \boldsymbol{n}\rrbracket,\boldsymbol\Psi)_{S} \\ \nonumber
&-\sum_{T\in \T_{h}} (\text{div}(\mathbf{f} + \omega^{2}(\bsiep\cdot\mathfrak{u})\mathbf{y}_{h}),\varphi)_{T}
+ \sum_{S\in\mathcal{S}}(\llbracket (\mathbf{f} + \omega^{2}(\bsiep\cdot\mathfrak{u})\mathbf{y}_{h})\cdot\boldsymbol{n}\rrbracket, \varphi)_{S}.
\end{align}
On the other hand, from the coercivity property \cite[Proposition 4.1]{CiarletJr2020} we observe that
\begin{equation}\label{eq:coercivity_a}
\|\mathbf{e}_{\mathbf{y}}\|_{\mathbf{H}(\cu,\Omega)}^{2} 
\lesssim
|(\mu^{-1} \cu \mathbf{e}_{\mathbf{y}}, \cu \mathbf{e}_{\mathbf{y}})_{\Omega} - \omega^{2}((\bsiep\cdot\mathfrak{u})\mathbf{e}_{\mathbf{y}},\mathbf{e}_{\mathbf{y}})_{\Omega}|.
\end{equation}
Therefore, using $\bw = \mathbf{e}_{\mathbf{y}}$ in \eqref{eq:identity_apost}, inequality \eqref{eq:coercivity_a}, basic inequalities, the estimates in \eqref{eq:property_op_decom}, and the finite number of overlapping patches, we arrive at $\|\mathbf{e}_{\mathbf{y}}\|_{\mathbf{H}(\cu,\Omega)}^{2}
\lesssim \mathcal{E}_{\T_{h}}\|\mathbf{e}_{\mathbf{y}}\|_{\mathbf{H}(\cu,\Omega)}$, which concludes the proof.
\end{proof}

%%%%%%%%%%%%%%%%%%%%%%%%%%%%%%%%%%%%%%%%%%%%%%%%%%%
%%%%%%%%%%%%%%%%%%%%%%%%%%%%%%%%%%%%%%%%%%%%%%%%%%%
%%%%%%%%%%%%%%%%%%%%%%%%%%%%%%%%%%%%%%%%%%%%%%%%%%%
%%%%%%%%%%%%%%%%%%%%%%%%%%%%%%%%%%%%%%%%%%%%%%%%%%%
%%%%%%%%%%%%%%%%%%%%%%%%%%%%%%%%%%%%%%%%%%%%%%%%%%%
%%%%%%%%%%%%%%%%%%%%%%%%%%%%%%%%%%%%%%%%%%%%%%%%%%%
%%%%%%%%%%%%%%%%%%%%%%%%%%%%%%%%%%%%%%%%%%%%%%%%%%%
%%%%%%%%%%%%%%%%%%%%%%%%%%%%%%%%%%%%%%%%%%%%%%%%%%%

\section{The optimal control problem}
\label{sec:the_ocp}

In this section, we analyze the following weak formulation of the optimal control problem \eqref{eq:cost_function}--\eqref{def:box_constraints}: Find
\begin{equation}\label{eq:weak_min_problem}
\min\{ \mathcal{J}(\by,\mathbf{u}):~ (\by,\mathbf{u}) \in  \mathbf{H}_{0}(\cu,\Omega) \times U_{ad}\}, 
\end{equation}
subject to
\begin{equation}\label{eq:weak_st_eq}
(\mu^{-1}\cu \by, \cu \bw)_{\Omega} - \omega^{2}((\bsiep\cdot\mathbf{u})\by,\bw)_{\Omega} = (\bbf,\bw)_{\Omega} \quad \forall \bw \in \mathbf{H}_{0}(\cu,\Omega).
\end{equation}
We recall that $\bbf \in \brL^2(\Omega;\mathbb{C})$, $U_{ad}$ is defined in \eqref{def:box_constraints}, and that $\omega\in \mathbb{R}^{+}$, $\mu\in \rL^\infty(\Omega)$, and $\bsiep$ are given as in section \ref{sec:model_problem}. Note that in \eqref{eq:weak_st_eq} the control corresponds to a vector acting on both the
electrical permittivity and conductivity of the material $\Omega$, in a given partition. We have considered this scenario only for the sake of mathematical generality. In particular, the analysis developed below can be adapted to take into consideration the real-valued coefficients $\varepsilon$ or $\sigma$.

%\begin{remark}[extensions]\label{rmk:real}
%  \PV{For the sake of mathematical generality, the term to be controlled is related to the complex variable $\varepsilon_\sigma$. Nevertheless, the subsequent analysis can be adjusted  to take into consideration the real variables $\varepsilon$ or $\sigma$, provided that the sesquilinear form is coercive. } Additionally, the analysis can also be extended to other bilinear optimal control problems of relevant variables within the Maxwell's equations framework.
% For instance, given \PV{real functions $\kappa,\chi \in P\rW^{1,\infty}(\Omega)$ satisfying $\kappa \geq \kappa_0 >0 $ 
%and $\chi \geq \chi_0 >0 $ with $\kappa_0, \mu_0\in \mathbb{R}^{+}$, 
%the state equation \eqref{eq:strong_state_eq} can be modified as follows:}
%
%\begin{align*}
%\cu \chi\cu \by  +(\kappa\cdot \mathbf{u})\by= \bbf \quad \mbox{in } \Omega, \qquad
%\by\times \bn= \mathbf{0} \quad \mbox{on } \Gamma.
%\end{align*}
%This problem arises, for example, when discretizing time-dependent Maxwell's equations 
%(see, e.g., \cite{MR2373173, MR1735971, MR2560827, MR2319088} for a posteriori error analysis of such formulation).
%\end{remark}

\begin{remark}[extensions]\label{rmk:real2}
 The analysis that we present in what follows extends to other bilinear optimal control problems of relevant variables within the Maxwell's equations framework.
 For instance, given real-valued coefficients $\kappa,\chi \in P\rW^{1,\infty}(\Omega)$ satisfying $\kappa \geq \kappa_0 >0 $ 
and $\chi \geq \chi_0 >0 $ with $\kappa_0, \mu_0\in \mathbb{R}^{+}$, 
the state equation \eqref{eq:strong_state_eq} can be modified as follows:
\begin{align*}
\cu \chi\cu \by  +(\kappa\cdot \mathbf{u})\by= \bbf \quad \mbox{in } \Omega, \qquad
\by\times \bn= \mathbf{0} \quad \mbox{on } \Gamma.
\end{align*}
This problem arises, for example, when discretizing time-dependent Maxwell's equations 
(see, e.g., \cite{MR2373173, MR1735971, MR2560827, MR2319088} for a posteriori error analysis of such formulation).
\end{remark}

%%%%%%%%%%%%%%%%%%%%%%%%%%%%%%%%%%%%%%%%%%%%%%%%%%%
%%%%%%%%%%%%%%%%%%%%%%%%%%%%%%%%%%%%%%%%%%%%%%%%%%%
%%%%%%%%%%%%%%%%%%%%%%%%%%%%%%%%%%%%%%%%%%%%%%%%%%%
%%%%%%%%%%%%%%%%%%%%%%%%%%%%%%%%%%%%%%%%%%%%%%%%%%%

\subsection{Existence of solutions}\label{sec:existence_of_sol}

Let us introduce the set $\mathbf{U}:=\{\mathbf{v}\in \mathbb{R}^{\ell}: \exists \mathbf{c}\in\mathbb{R}^{\ell}, \mathbf{c} > \mathbf{0} \text{ such that } \mathbf{v} > \mathbf{c} > \mathbf{0}\}$. We note that $U_{ad}\subset \mathbf{U}$. With $\mathbf{U}$ at hand, we introduce the control-to-state operator $\mathcal{S} : \mathbf{U} \to \mathbf{H}_{0}(\cu,\Omega)$ as follows: for any $\mathbf{u}\in\mathbf{U}$, $\mathcal{S}$ associates to it the unique solution $\by \in \mathbf{H}_{0}(\cu,\Omega)$ of problem \eqref{eq:weak_st_eq}.

The next result states differentiability properties of $\mathcal{S}$.

\begin{theorem}[differentiability properties of $\mathcal{S}$]\label{thm:diff_prop_S}
The control-to-state operator $\mathcal{S}$ is of class $C^{\infty}$. Moreover, for $\mathbf{h}\in \mathbb{R}^{\ell}$, $\bz:=\mathcal{S}'(\mathbf{u})\mathbf{h}\in \mathbf{H}_{0}(\cu,\Omega)$ corresponds to the unique solution to 
\begin{equation}\label{eq:first_der_S}
(\mu^{-1}\cu \bz, \cu \bw)_{\Omega} - \omega^{2}((\bsiep\cdot\mathbf{u})\bz,\bw)_{\Omega} = \omega^{2}((\bsiep\cdot\mathbf{h}) \by,\bw)_{\Omega}
\end{equation}
for all  $\bw \in \mathbf{H}_{0}(\cu,\Omega)$, where $\by = \mathcal{S}\mathbf{u}$. Moreover, if $\mathbf{h}_1,\mathbf{h}_2\in \mathbb{R}^{\ell}$, then $\boldsymbol\zeta=\mathcal{S}''(\mathbf{u})(\mathbf{h}_1,\mathbf{h}_2)\in \mathbf{H}_{0}(\cu,\Omega)$ is the unique solution to 
\begin{equation}\label{eq:second_der_S}
(\mu^{-1}\cu \boldsymbol\zeta, \cu \bw)_{\Omega} - \omega^{2}((\bsiep\cdot\mathbf{u})\boldsymbol\zeta,\bw)_{\Omega}
 = \omega^{2}((\bsiep\cdot\mathbf{h}_{1}) \bz_{\mathbf{h}_{2}} + (\bsiep\cdot\mathbf{h}_{2}) \bz_{\mathbf{h}_{1}},\bw)_{\Omega} \hspace{-0.3cm}
\end{equation}
for all  $\bw \in \mathbf{H}_{0}(\cu,\Omega)$, with $\bz_{\mathbf{h}_{i}}=\mathcal{S}'(\mathbf{u})\mathbf{h}_i$ and $i \in \{ 1,2 \}$.
\end{theorem}
\begin{proof}
The proof is based on the implicit function theorem. With this in mind, we
define the operator $\mathcal{F}: \mathbf{H}_0(\cu,\Omega)\times \mathbf{U}\to \mathbf{H}_0(\cu,\Omega)'$ by 
\begin{align*}
\mathcal{F}(\by,\mathbf{u}):=\cu \mu^{-1}\cu \by  - \omega^{2} (\bsiep\cdot \mathbf{u})\by - \bbf.
\end{align*}
A direct computation reveals that $\mathcal{F}$ is of class $C^{\infty}$ and satisfies $\mathcal{F}(\mathcal{S}\mathbf{u},\mathbf{u})=0$ for all $\mathbf{u}\in\mathbf{U}$. Moreover, Lax--Milgram lemma yields that
\[
\partial_{\by}\mathcal{F}(\by,\mathbf{u})(\bz) = \cu \mu^{-1}\cu \bz  -\omega^{2} (\bsiep\cdot \mathbf{u})\bz,
\]
is an isomorphism from $\mathbf{H}_0(\cu,\Omega)$ to $\mathbf{H}_0(\cu,\Omega)'$. Therefore, the implicit function theorem implies that the control-to-state operator $\mathcal{S}$ is infinitely Fr\'echet differentiable. Finally, \eqref{eq:first_der_S} and \eqref{eq:second_der_S} follow by simple calculations.
\end{proof}

Let us define the reduced cost functional $j: \mathbf{U}\to \mathbb{R}_{0}^{+}$ by $j(\mathbf{u})=\mathcal{J}(\mathcal{S}\mathbf{u},\mathbf{u})$. A direct consequence of Theorem \ref{thm:diff_prop_S} is the Fréchet differentiability $j$. 

\begin{corollary}[differentiability properties of $j$]
The reduced cost functional $j: \mathbf{U}\to \mathbb{R}_{0}^{+}$ is of class $C^{\infty}$. 
\end{corollary}

Since $j$ is continuous and $U_{ad}$ is compact, Weierstra{\ss} theorem immediately yields the existence of at least one globally optimal control $\mathbf{u}^{*}\in U_{ad}$, with a corresponding optimal state $\by^{*}:=\mathcal{S}\mathbf{u}^{*}\in \mathbf{H}_0(\cu,\Omega)$. This is summarized in the next result.

\begin{theorem}[existence of optimal solutions]\label{thm:existence_opt_sol}
The optimal control problem \eqref{eq:weak_min_problem}--\eqref{eq:weak_st_eq} admits at least one global solution $(\by^{*},\mathbf{u}^{*}) \in  \mathbf{H}_{0}(\cu,\Omega) \times U_{ad}$. 
\end{theorem}

Since our optimal control problem \eqref{eq:weak_min_problem}--\eqref{eq:weak_st_eq} is not convex, we discuss optimality conditions under the framework of local solutions in $\mathbb{R}^{\ell}$ with $\ell\in\mathbb{N}$. To be precise, a control $\mathbf{u}^{*} \in {U}_{ad}$ is said to be locally optimal in $\mathbb{R}^{\ell}$ for \eqref{eq:weak_min_problem}--\eqref{eq:weak_st_eq} if there exists a constant $\delta>0$ such that $\mathcal{J}(\by^{*},\mathbf{u}^{*})\leq \mathcal{J}(\by,\mathbf{u})$ for all $\mathbf{u}\in U_{ad}$ such that $\|\mathbf{u}-\mathbf{u}^{*}\|_{\mathbb{R}^{\ell}}\leq \delta$. Here, $\by^{*}$ and $\by$ denote the states associated to $\mathbf{u}^{*}$ and $\mathbf{u}$, respectively.

%%%%%%%%%%%%%%%%%%%%%%%%%%%%%%%%%%%%%%%%%%%%%%%%%%%
%%%%%%%%%%%%%%%%%%%%%%%%%%%%%%%%%%%%%%%%%%%%%%%%%%%
%%%%%%%%%%%%%%%%%%%%%%%%%%%%%%%%%%%%%%%%%%%%%%%%%%%
%%%%%%%%%%%%%%%%%%%%%%%%%%%%%%%%%%%%%%%%%%%%%%%%%%%
\subsection{Optimality conditions}\label{sec:opt_cond}

%%%%%%%%%%%%%%%%%%%%%%%%%%%%%%%%%%%%%%%%%%%%%%%%%%%
%%%%%%%%%%%%%%%%%%%%%%%%%%%%%%%%%%%%%%%%%%%%%%%%%%%
%%%%%%%%%%%%%%%%%%%%%%%%%%%%%%%%%%%%%%%%%%%%%%%%%%%
%%%%%%%%%%%%%%%%%%%%%%%%%%%%%%%%%%%%%%%%%%%%%%%%%%%

\subsubsection{First-order optimality condition}\label{sec:1st_opt_condition}
We begin with a standard result: if $\mathbf{u}^{*}\in U_{ad}$ denotes a locally optimal control for \eqref{eq:weak_min_problem}--\eqref{eq:weak_st_eq}, then \cite[Theorem 3.7]{MR3311948}
\begin{align}
\label{eq:variational_inequality}
j'(\mathbf{u}^{*})(\mathbf{u} - \mathbf{u}^{*}) \geq 0 \qquad \forall \mathbf{u} \in U_{ad}.
\end{align}
In \eqref{eq:variational_inequality}, $j'(\mathbf{u}^{*})$ denotes the Gate\^aux derivative of $j$ at $\mathbf{u}^{*}$. To explore \eqref{eq:variational_inequality} we introduce, given $\mathbf{u}\in U_{ad}$ and $\by = \mathcal{S}\mathbf{u}$, the \emph{adjoint variable} $\bp \in \mathbf{H}_{0}(\cu,\Omega)$ as the unique solution to the \emph{adjoint equation}
\begin{align}\label{eq:adj_eq}
& (\mu^{-1}\cu \bp, \cu \bw)_{\Omega} - \omega^{2}((\bsiep\cdot \mathbf{u})\bp,\bw)_{\Omega}  \\
& \qquad \qquad = (\overline{\by - \by_{\Omega}},\bw)_{\Omega} + (\overline{\cu \by - \bE_{\Omega}}, \cu \bw)_{\Omega} \nonumber
\end{align}
for all $\bw \in \mathbf{H}_{0}(\cu,\Omega)$. The well-posedness of \eqref{eq:adj_eq} follows from the Lax-Milgram lemma. Moreover, the following stability estimate holds:
\begin{align}\label{eq:stab_adj_eq}
\|\bp\|_{\mathbf{H}(\cu,\Omega)} 
\lesssim 
\|\by\|_{\mathbf{H}(\cu,\Omega)} + \|\by_{\Omega}\|_{\Omega} + \|\mathbf{E}_{\Omega}\|_{\Omega}
\lesssim
\|\bbf\|_{\Omega} + \|\by_{\Omega}\|_{\Omega} + \|\mathbf{E}_{\Omega}\|_{\Omega}.
\end{align}

We have all the ingredients at hand to give a characterization for \eqref{eq:variational_inequality}.

\begin{theorem}[first-order necessary optimality condition]\label{thm:first_ord_opt} 
Every locally optimal control $\mathbf{u}^{*}\in U_{ad}$ for problem \eqref{eq:weak_min_problem}--\eqref{eq:weak_st_eq} satisfies the variational inequality
\begin{align}\label{eq:var_ineq_with_adj_state}
\sum_{k=1}^{\ell}\left(\alpha\mathbf{u}^{*}_{k} + \omega^{2} \mathfrak{Re}\left\{\int_{\Omega_{k}}\bsiep\by^{*}\cdot\bp^{*}\right\}\right)\left(\mathbf{u}_{k} - \mathbf{u}^{*}_{k}\right) \geq 0 \qquad \forall \mathbf{u}\in U_{ad},
\end{align}
where $\bp^{*}\in \mathbf{H}_{0}(\cu,\Omega)$ solves \eqref{eq:adj_eq} with $\mathbf{u}$ and $\by$ replaced by $\mathbf{u}^{*}$ and $\by^{*} = \mathcal{S}\mathbf{u}^{*}$, respectively. We recall that $\mathcal{P}=\{\Omega_{k}\}_{k=1}^{\ell}$ is the given partition from section \textnormal{\ref{sec:partition_fields}}.
\end{theorem}
\begin{proof}
A direct calculation reveals that \eqref{eq:variational_inequality} can be rewritten as follows:
\begin{equation}\label{eq:var_ineq_separate}
\mathfrak{Re}\{(\bz_{\mathbf{u} - \mathbf{u}^{*}}, \by^{*} - \by_{\Omega})_{\Omega} + (\cu(\bz_{\mathbf{u} - \mathbf{u}^{*}}), \cu\by^{*} - \mathbf{E}_{\Omega})_{\Omega}\} + \alpha(\mathbf{u}^{*}, \mathbf{u} - \mathbf{u}^{*})_{\mathbb{R}^{\ell}} \geq 0
\end{equation}
for all $\mathbf{u}\in U_{ad}$, where, to simplify the notation, we have defined $\bz_{\mathbf{u} - \mathbf{u}^{*}} := \mathcal{S}'(\mathbf{u}^{*})(\mathbf{u} - \mathbf{u}^{*})$.
We immediately notice that $\bz_{\mathbf{u} - \mathbf{u}^{*}}\in \mathbf{H}_0(\cu,\Omega)$ corresponds to the unique solution to \eqref{eq:first_der_S} with $\mathbf{u}=\mathbf{u}^{*}$, $\by=\by^{*}$, and $\mathbf{h}=\mathbf{u} - \mathbf{u}^{*}$. 
Since $\alpha(\mathbf{u}^{*}, \mathbf{u} - \mathbf{u}^{*})_{\mathbb{R}^{\ell}}$ is already present in \eqref{eq:var_ineq_separate}, we concentrate on the remaining terms. 
Let us use $\bw = \overline{\bz}_{\mathbf{u} - \mathbf{u}^{*}}$ in problem  \eqref{eq:adj_eq} and $\bw=\overline{\bp^{*}}$ in the problem that $\bz_{\mathbf{u} - \mathbf{u}^{*}}$ solves to obtain 
\begin{align}\label{eq:identity_first_order}
&\mathfrak{Re}\{(\bz_{\mathbf{u} - \mathbf{u}^{*}}, \by^{*} - \by_{\Omega})_{\Omega} + (\cu(\bz_{\mathbf{u} - \mathbf{u}^{*}}), \cu\by^{*} - \mathbf{E}_{\Omega})_{\Omega}\} \\
& \qquad \qquad = \omega^{2} \mathfrak{Re}\{(\bsiep\cdot(\mathbf{u} - \mathbf{u}^{*}))\by^{*},\overline{\bp^{*}})_{\Omega}\}. \nonumber
\end{align}
Therefore, using identity \eqref{eq:identity_first_order} in \eqref{eq:var_ineq_separate}, we conclude the desired inequality \eqref{eq:var_ineq_with_adj_state}.
\end{proof}

%%%%%%%%%%%%%%%%%%%%%%%%%%%%%%%%%%%%%%%%%%%%%%%%%%%
%%%%%%%%%%%%%%%%%%%%%%%%%%%%%%%%%%%%%%%%%%%%%%%%%%%
%%%%%%%%%%%%%%%%%%%%%%%%%%%%%%%%%%%%%%%%%%%%%%%%%%%
%%%%%%%%%%%%%%%%%%%%%%%%%%%%%%%%%%%%%%%%%%%%%%%%%%%

\subsubsection{Second-order optimality conditions}\label{sec:2nd_opt_condition}

For each $k\in \{1,\ldots,\ell\}$, we define 
$\bar{\mathfrak{d}}_{k}:=\alpha\mathbf{u}^{*}_{k} 
+ \omega^{2} \mathfrak{Re}\{\int_{\Omega_{k}}\bsiep\by^{*}\cdot\bp^{*}\}$. Here, $\mathbf{u}^{*},\by^{*},\bp^{*}$ and $\Omega_{k}$ are given as in the statement of Theorem \ref{thm:first_ord_opt}. We introduce the cone of critical directions at $\mathbf{u}^{*} \in U_{ad}$:
\begin{align}\label{def:critical_cone}
\mathbf{C}_{\mathbf{u}^{*}}:=\{\mathbf{v}\in \mathbb{R}^{\ell}\, \text{ that satisfies } \eqref{eq:sign_cond} \text{ and }  \mathbf{v}_{k} = 0 \text{ if } |\bar{\mathfrak{d}}_{k}| > 0 \},
\end{align}
where condition \eqref{eq:sign_cond} reads, for all $k\in \{1,\ldots,\ell\}$, as follows:
\begin{align}\label{eq:sign_cond}
\mathbf{v}_{k} \geq 0 ~ \text{ if } ~ \mathbf{u}^{*}_{k} = \mathbf{a}_{k}
\quad \text{ and } \quad
\mathbf{v}_{k} \leq 0  ~ \text{ if } ~ \mathbf{u}^{*}_{k} = \mathbf{b}_{k}.
\end{align}
With this set at hand, we present the next result which follows from the standard Karush--Kuhn--Tucker theory of mathematical optimization in finite-dimensional spaces; see, e.g., \cite[Theorem 3.8]{MR3311948} and \cite[Section 6.3]{MR2012832}.

\begin{theorem}[second-order necessary and sufficient optimality conditions]\label{thm:second_necess_suff_opt_cond} 
If $\mathbf{u}^{*} \in U_{ad}$ is a local minimum for problem \eqref{eq:weak_min_problem}--\eqref{eq:weak_st_eq}, then $j''(\mathbf{u}^{*})\mathbf{v}^2 \geq 0$ for all $\mathbf{v}\in \mathbf{C}_{\mathbf{u}^{*}}$. Conversely, if $\mathbf{u}^{*} \in U_{ad}$ satisfies the variational inequality \eqref{eq:var_ineq_with_adj_state} \textnormal{(}equivalently \eqref{eq:variational_inequality}\textnormal{)} and the second-order sufficient condition
\begin{align}
\label{eq:second_order_sufficient}
j''(\mathbf{u}^{*})\mathbf{v}^2 > 0 \quad \forall \mathbf{v}\in \mathbf{C}_{\mathbf{u}^{*}} \setminus \{ \mathbf{0} \},
\end{align}
then there exist $\eta>0$ and $\delta >0$ such that
\begin{align*}\label{eq:quad_growth_new}
\displaystyle j(\mathbf{u})\geq j(\mathbf{u}^{*})+\frac{\eta}{4}\|\mathbf{u}-\mathbf{u}^{*}\|_{\mathbb{R}^\ell}^2
\qquad \forall \mathbf{u}\in U_{ad}: \|\mathbf{u}-\mathbf{u}^{*}\|_{\mathbb{R}^\ell}\leq \delta.
\end{align*}
In particular, $\mathbf{u}^{*} $ is a strict local solution of \eqref{eq:weak_min_problem}--\eqref{eq:weak_st_eq}.
\end{theorem} 

In order to provide error estimates for solutions of problem \eqref{eq:weak_min_problem}--\eqref{eq:weak_st_eq}, we shall use an equivalent condition to \eqref{eq:second_order_sufficient} which follows directly of our finite dimensional setting for the control variable. To present it, we introduce, for $\tau > 0$, the cone 
\begin{equation}\label{def:cone_kappa}
\mathbf{C}_{\mathbf{u}^{*}}^{\tau}:=\{\mathbf{v}\in \mathbb{R}^{\ell} \text{ that satisfies } \eqref{eq:sign_cond} \text{ and } \eqref{eq:v_i_tau} \},
\end{equation}
where, for $k \in\{1,\ldots,\ell\}$, condition \eqref{eq:v_i_tau} reads as follows:
\begin{equation}\label{eq:v_i_tau}
 |\bar{\mathfrak{d}}_k| > \tau \implies \mathbf{v}_{k} = 0.
\end{equation}

\begin{theorem}[equivalent condition]\label{thm:equivalent_opt_cond}
Let $\mathbf{u}^{*} \in U_{ad}$ be such that it satisfies the variational inequality \eqref{eq:var_ineq_with_adj_state} \textnormal{(}equivalently \eqref{eq:variational_inequality}\textnormal{)}. Then, \eqref{eq:second_order_sufficient} is equivalent to
\begin{equation}\label{eq:second_order_equivalent}
\exists\tau,\nu >0: \quad j''(\mathbf{u}^{*})\mathbf{v}^2  \geq \nu \|\mathbf{v}\|_{\mathbb{R}^{\ell}}^2 \quad \forall \mathbf{v} \in \mathbf{C}_{\mathbf{u}^{*}}^\tau.
\end{equation}
\end{theorem}

We end this section with a result that will be useful for proving error estimates. 

\begin{proposition}[$j''$ is locally Lipschitz]
Let $\mathbf{u}_{1},\mathbf{u}_{2}\in U_{ad}$ and $\mathbf{h}\in\mathbb{R}^{\ell}$. Then, we have the following estimate:
\begin{align}\label{eq:Lipschitz_property}
|j''(\mathbf{u}_{1})\mathbf{h}^{2} - j''(\mathbf{u}_{2})\mathbf{h}^{2}| \leq C_{L} \| \mathbf{u}_{1} - \mathbf{u}_{2}\|_{\mathbb{R}^{\ell}}\|\mathbf{h}\|_{\mathbb{R}^{\ell}}^{2},
\end{align}
where $C_{L}>0$ denotes a constant depending only on the problem data.
\end{proposition}
\begin{proof}
We proceed on the basis of two steps. 

\underline{Step 1.} (characterization of $j''$) Let  $\mathbf{u}\in U_{ad}$ and $\mathbf{h}\in \mathbb{R}^{\ell}$. We start with a simple calculation and obtain that
\begin{align}\label{eq:charac_j2_prev}
j''(\mathbf{u})\mathbf{h}^2
&=
\alpha\|\mathbf{h}\|_{\mathbb{R}^{\ell}}^2
+
\|\bz\|_{\Omega}^{2} + \|\cu \bz\|_{\Omega}^{2} \\
& \qquad \qquad + \mathfrak{Re}\{
(\boldsymbol\zeta, \mathcal{S}\mathbf{u} - \mathbf{y}_{\Omega})_{\Omega} +  (\cu(\boldsymbol\zeta), \cu(\mathcal{S}\mathbf{u}) - \mathbf{E}_{\Omega})_{\Omega} \}, \nonumber
\end{align}
where $\bz = \mathcal{S}'(\mathbf{u})\mathbf{h} \in \mathbf{H}_0(\cu,\Omega)$ and $\boldsymbol\zeta = \mathcal{S}''(\mathbf{u})\mathbf{h}^2  \in \mathbf{H}_0(\cu,\Omega)$ solve \eqref{eq:first_der_S} and \eqref{eq:second_der_S}, respectively. We now set $\bw = \overline{\boldsymbol\zeta}$ in \eqref{eq:adj_eq} and $\bw=\overline{\bp}$ in \eqref{eq:second_der_S} to obtain
\[
\mathfrak{Re}\{
(\boldsymbol\zeta, \mathcal{S}\mathbf{u} - \mathbf{y}_{\Omega})_{\Omega} +  (\cu(\boldsymbol\zeta), \cu(\mathcal{S}\mathbf{u}) - \mathbf{E}_{\Omega})_{\Omega} \}
=
\mathfrak{Re}\{2\omega^{2}((\bsiep\cdot\mathbf{h}) \bz,\overline{\bp})_{\Omega}\}.
\]
Replacing the previous identity in \eqref{eq:charac_j2_prev} results in
\begin{equation}\label{eq:charac_j''}
j''(\mathbf{u})\mathbf{h}^2
=
\alpha\|\mathbf{h}\|_{\mathbb{R}^{\ell}}^2
+
\mathfrak{Re}\{2\omega^{2}((\bsiep\cdot\mathbf{h}) \bz,\overline{\bp})_{\Omega}\}
+
\|\bz\|_{\Omega}^{2}
+ 
 \|\cu \bz\|_{\Omega}^{2}.
\end{equation}

\underline{Step 2.} (estimate \eqref{eq:Lipschitz_property}) Let $\mathbf{u}_1,\mathbf{u}_2 \in U_{ad}$ and $\mathbf{h}\in \mathbb{R}^{\ell}$. Define $\bz_{1}= \mathcal{S}'(\mathbf{u}_{1})\mathbf{h}$ and $\bz_{2} = \mathcal{S}'(\mathbf{u}_{2})\mathbf{h}$. In view of the characterization \eqref{eq:charac_j''}, we obtain
\begin{multline*}
[j''(\mathbf{u}_{1}) - j''(\mathbf{u}_{2})]\mathbf{h}^2
=
\mathfrak{Re}\{2\omega^{2} ((\bsiep \cdot \mathbf{h}) (\bz_{1} - \bz_{2}),\overline{\bp}_{1})_{\Omega}\}
+
\mathfrak{Re}\{2\omega^{2} ((\bsiep \cdot \mathbf{h})\bz_{2},\overline{\bp}_{1} - \overline{\bp}_{2})_{\Omega}\}\\
+ 
[\|\bz_{1}\|_{\Omega}^{2} - \|\bz_{2}\|_{\Omega}^2]
+
[\|\cu \bz_{1}\|_{\Omega}^{2} - \|\cu \bz_{2}\|_{\Omega}^2]
=: \mathbf{I} + \mathbf{II} + \mathbf{III} + \mathbf{IV},
\end{multline*}
where $\bp_{i}$ ($i\in\{1,2\}$) denotes the solution to \eqref{eq:adj_eq} with $\by$ and $\mathbf{u}$ replaced by $\by_{i} = \mathcal{S}\mathbf{u}_i$ and  $\mathbf{u}_{i}$, respectively. We bound each term on the right-hand side of the latter identity. 

The use of an elemental inequality in combination with the stability estimate \eqref{eq:stab_adj_eq} for $\bp_{1}$ yields the estimation
\begin{align*}
|\mathbf{I}| 
\lesssim 
\|\mathbf{h}\|_{\mathbb{R}^{\ell}} \|\bsiep\|_{\mathrm{L}^{\infty}(\Omega;\mathbb{C})} \|\bz_{1} - \bz_{2}\|_{\Omega}
\|\bp_{1}\|_{\Omega}
\lesssim
\|\mathbf{h}\|_{\mathbb{R}^{\ell}}\|\bz_{1} - \bz_{2}\|_{\mathbf{H}_0(\cu,\Omega)}.
\end{align*}
Hence, it suffices to bound $\|\bz_{1} - \bz_{2}\|_{\mathbf{H}_0(\cu,\Omega)}$. Note that $\bz_{1} - \bz_{2} \in \mathbf{H}_0(\cu,\Omega)$ corresponds to the solution of
\begin{align*}
&(\mu^{-1}\cu (\bz_{1} - \bz_{2}), \cu \bw)_{\Omega} - \omega^{2}((\bsiep\cdot\mathbf{u}_{1})(\bz_{1} - \bz_{2}),\bw)_{\Omega} \\ 
& \qquad \qquad = \omega^{2}((\bsiep\cdot\mathbf{h}) (\by_{1} - \by_{2}),\bw)_{\Omega} +\omega^{2}((\bsiep\cdot(\mathbf{u}_{1} - \mathbf{u}_{2}))\bz_{2},\bw)_{\Omega}
\end{align*}
for all $\bw\in \mathbf{H}_{0}(\cu,\Omega)$. A stability estimate allows us to obtain 
\begin{align*}
\|\bz_{1} - \bz_{2}\|_{\mathbf{H}_0(\cu,\Omega)}
\lesssim
\|\mathbf{h}\|_{\mathbb{R}^{\ell}}\|\by_{1} - \by_{2}\|_{\Omega} + \|\bz_{2}\|_{\Omega}\|\mathbf{u}_{1} - \mathbf{u}_{2}\|_{\mathbb{R}^{\ell}}.
\end{align*}
We control $\|\bz_{2}\|_{\Omega}$ in view of the stability estimate $\|\bz_{2}\|_{\Omega} \leq \|\bz_{2}\|_{\mathbf{H}_0(\cu,\Omega)} \lesssim \|\mathbf{h}\|_{\mathbb{R}^{\ell}}$. The term $\|\by_{1} - \by_{2}\|_{\Omega}$ is bounded as follows:
\begin{equation}\label{eq:estimate_y_1_y_2}
\|\by_{1} - \by_{2}\|_{\Omega}
\leq 
\|\by_{1} - \by_{2}\|_{\mathbf{H}_0(\cu,\Omega)}
\lesssim
\|\by_{2}\|_{\Omega}
\|\mathbf{u}_{1} - \mathbf{u}_{2}\|_{\mathbb{R}^{\ell}}
\lesssim
\|\bbf\|_{\Omega}
\|\mathbf{u}_{1} - \mathbf{u}_{2}\|_{\mathbb{R}^{\ell}}.
\end{equation}
We thus conclude that 
\begin{align}\label{eq:diff_z1_z2}
\|\bz_{1} - \bz_{2}\|_{\mathbf{H}_0(\cu,\Omega)}\lesssim \|\mathbf{u}_1-\mathbf{u}_2\|_{\mathbb{R}^{\ell}}\|\mathbf{h}\|_{\mathbb{R}^{\ell}},
\end{align}
and, consequently $|\mathbf{I}| \lesssim \|\mathbf{u}_1-\mathbf{u}_2\|_{\mathbb{R}^{\ell}}\|\mathbf{h}\|_{\mathbb{R}^{\ell}}^2$. The control of $\mathbf{II}$ follows similar arguments. In fact, in view of the estimate $\|\bz_{2}\|_{\Omega} \lesssim \|\mathbf{h}\|_{\mathbb{R}^{\ell}}$, we obtain
\begin{align*}
|\mathbf{II}| 
\lesssim
\|\mathbf{h}\|_{\mathbb{R}^{\ell}} \|\bsiep\|_{\mathbf{L}^{\infty}(\Omega;\mathbb{C})} \|\bz_{2}\|_{\Omega} \|\bp_{1} - \bp_{2}\|_{\Omega} 
\lesssim
\|\mathbf{h}\|_{\mathbb{R}^{\ell}}^2 \|\bp_{1} - \bp_{2}\|_{\mathbf{H}_0(\cu,\Omega)}.
\end{align*}
The term $\|\bp_{1} - \bp_{2}\|_{\mathbf{H}_0(\cu,\Omega)}$ is controlled as follows:
\begin{align*}
\|\bp_{1} - \bp_{2}\|_{\mathbf{H}_0(\cu,\Omega)} 
\lesssim
\|\by_{1} - \by_{2}\|_{\mathbf{H}_0(\cu,\Omega)} + \|\bp_{2}\|_{\Omega}\|\mathbf{u}_{1} - \mathbf{u}_{2}\|_{\mathbb{R}^{\ell}}
\lesssim
\|\mathbf{u}_{1} - \mathbf{u}_{2}\|_{\mathbb{R}^{\ell}},
\end{align*}
upon using estimate \eqref{eq:estimate_y_1_y_2} and the stability estimate \eqref{eq:stab_adj_eq} for $\bp_{2}$. To control $\mathbf{III}$, we use the bounds $\|\bz_{1}\|_{\Omega} \lesssim \|\mathbf{h}\|_{\mathbb{R}^{\ell}}$, $\|\bz_{2}\|_{\Omega} \lesssim \|\mathbf{h}\|_{\mathbb{R}^{\ell}}$, and \eqref{eq:diff_z1_z2}, to arrive at 
\begin{equation*}\label{eq:estimation_of_III}
|\mathbf{III}| 
\lesssim 
\|\bz_{1} - \bz_{2}\|_{\Omega}\|\bz_{1} + \bz_{2}\|_{\Omega}
\lesssim
\|\mathbf{u}_{1} - \mathbf{u}_{2}\|_{\mathbb{R}^{\ell}}\|\mathbf{h}\|_{\mathbb{R}^{\ell}}^2.
\end{equation*}
Finally, to estimate the term $\mathbf{IV}$, we use the bound \eqref{eq:diff_z1_z2}, $\|\bz_{1}\|_{\mathbf{H}_0(\cu,\Omega)} \lesssim \|\mathbf{h}\|_{\mathbb{R}^{\ell}}$, and $\|\bz_{2}\|_{\mathbf{H}_0(\cu,\Omega)} \lesssim \|\mathbf{h}\|_{\mathbb{R}^{\ell}}$. These arguments yield
\begin{equation*}\label{eq:estimation_of_IV}
|\mathbf{IV}| 
\lesssim 
\|\mathbf{curl}(\bz_{1} - \bz_{2})\|_{\Omega}\|\mathbf{curl}(\bz_{1} + \bz_{2})\|_{\Omega}
\lesssim
\|\mathbf{u}_{1} - \mathbf{u}_{2}\|_{\mathbb{R}^{\ell}}\|\mathbf{h}\|_{\mathbb{R}^{\ell}}^2.
\end{equation*}

The desired bound \eqref{eq:Lipschitz_property} follows from the identity $[j''(\mathbf{u}_{1}) - j''(\mathbf{u}_{2})]\mathbf{h}^2 = \mathbf{I} + \mathbf{II} + \mathbf{III} + \mathbf{IV}$ and a collection of the estimates obtained for $\mathbf{I}$, $\mathbf{II}$, $\mathbf{III}$, and $\mathbf{IV}$.
\end{proof}

%%%%%%%%%%%%%%%%%%%%%%%%%%%%%%%%%%%%%%%%%%%%%%%%%%%
%%%%%%%%%%%%%%%%%%%%%%%%%%%%%%%%%%%%%%%%%%%%%%%%%%%
%%%%%%%%%%%%%%%%%%%%%%%%%%%%%%%%%%%%%%%%%%%%%%%%%%%
%%%%%%%%%%%%%%%%%%%%%%%%%%%%%%%%%%%%%%%%%%%%%%%%%%%
%%%%%%%%%%%%%%%%%%%%%%%%%%%%%%%%%%%%%%%%%%%%%%%%%%%
%%%%%%%%%%%%%%%%%%%%%%%%%%%%%%%%%%%%%%%%%%%%%%%%%%%
%%%%%%%%%%%%%%%%%%%%%%%%%%%%%%%%%%%%%%%%%%%%%%%%%%%
%%%%%%%%%%%%%%%%%%%%%%%%%%%%%%%%%%%%%%%%%%%%%%%%%%%

\section{Finite element approximation}
\label{sec:fem_for_ocp}
To approximate the optimal control problem \eqref{eq:weak_min_problem}--\eqref{eq:weak_st_eq}, we propose the following discrete problem: Find $\min \mathcal{J}(\by_{h},\mathbf{u}_{h})$, with $(\by_{h},\mathbf{u}_{h})\in  \mathbf{V}(\mathscr{T}_{h})\times U_{ad}$, subject to 
\begin{equation}
\label{eq:discrete_state_equation}
(\mu^{-1}\cu \by_{h}, \cu \bw_{h})_{\Omega} -\omega^{2}((\bsiep\cdot \mathbf{u}_{h}) \by_{h},\bw_{h})_{\Omega} = (\bbf,\bw_{h})_{\Omega} \quad \forall \bw_{h}\in \mathbf{V}(\mathscr{T}_{h}).
\end{equation}
We recall that $\mathbf{V}(\mathscr{T}_{h})$ is defined as in \eqref{def:discrete_space}. 

Let us introduce the discrete control to state mapping $\mathcal{S}_{h}:  \mathbf{U} \ni \mathbf{u}_h \mapsto \by_h \in \mathbf{V}(\mathscr{T}_{h})$, where $\by_h$ solves \eqref{eq:discrete_state_equation}. In view of Lax-Milgram lemma, we have that $\mathcal{S}_{h}$ is continuous. We also introduce the discrete reduced cost function $j_{h}(\mathbf{u}_{h}):=\mathcal{J}(\mathcal{S}_{h}\mathbf{u}_{h},\mathbf{u}_{h})$. 

The existence of optimal solutions follows from the compactness of $U_{ad}$ and the continuity of $j_{h}$.  As in the continuous case, we characterize local optimal solutions through a discrete first-order optimality condition:  If $\mathbf{u}^{*}_{h}$ denotes a discrete local solution, then $j_{h}^{\prime}(\mathbf{u}^{*}_{h})(\mathbf{u} - \mathbf{u}^{*}_{h}) \geq  0$ for all $\mathbf{u} \in U_{ad}$. Following the arguments developed in the proof of Theorem \ref{thm:first_ord_opt}, we can rewrite the latter inequality as follows:
\begin{align}\label{eq:discrete_var_ineq}
\sum_{k=1}^{\ell}\left(\alpha(\mathbf{u}^{*}_{h})_{k} + \omega^{2} \mathfrak{Re}\left\{\int_{\Omega_{k}}\bsiep\by^{*}_{h}\cdot\bp^{*}_{h}\right\}\right)(\mathbf{u}_{k} - (\mathbf{u}^{*}_{h})_{k}) \geq 0 \qquad \forall \mathbf{u}\in U_{ad},
\end{align}
where $\by^{*}_{h}=\mathcal{S}_{h}\mathbf{u}^{*}_{h}$, and $\bp^{*}_{h} \in\mathbf{V}(\mathscr{T}_{h})$ solves the discrete adjoint problem 
\begin{align}\label{eq:discrete_adjoint_equation}
&(\mu^{-1}\cu \bp^{*}_{h}, \cu \bw_{h})_{\Omega}  -\omega^2((\bsiep\cdot \mathbf{u}^{*}_{h})\bp^{*}_{h},\bw_{h})_{\Omega} \\ 
& \qquad \qquad = (\overline{\by^{*}_{h} - \by_{\Omega}},\bw_{h})_{\Omega} + (\overline{\cu \by^{*}_{h} - \bE_{\Omega}}, \cu \bw_{h})_{\Omega} \quad \forall \bw_{h} \in \mathbf{V}(\mathscr{T}_{h}), \nonumber
\end{align}
whose well-posedness follows from the Lax-Milgram lemma. 

%%%%%%%%%%%%%%%%%%%%%%%%%%%%%%%%%%%%%%%%%%%%%%%%%%%
%%%%%%%%%%%%%%%%%%%%%%%%%%%%%%%%%%%%%%%%%%%%%%%%%%%
%%%%%%%%%%%%%%%%%%%%%%%%%%%%%%%%%%%%%%%%%%%%%%%%%%%
%%%%%%%%%%%%%%%%%%%%%%%%%%%%%%%%%%%%%%%%%%%%%%%%%%%

\subsection{Convergence of the discretization}

In order to prove convergence properties of our discrete solutions, we shall consider the following assumption:
\begin{align}\label{eq:assumption_f_and_PW}
\bbf\in \mathbf{H}(\textnormal{div},\Omega) \quad \text{and} \quad \mu,\bsiep\in P\rW^{1,\infty}(\Omega).
\end{align}

\begin{lemma}[error estimate]
\label{lemma:error_estim_st}
Let $\mathbf{u},\mathbf{u}_{h}\in U_{ad}$ and let $\by\in\bH_0(\cu,\Omega)$ and $\by_{h}\in \mathbf{V}(\mathscr{T}_{h})$ be the unique solutions to \eqref{eq:weak_st_eq} and \eqref{eq:discrete_state_equation}, respectively. 
If assumption \eqref{eq:assumption_f_and_PW} holds, then we have
\begin{align}\label{eq:Lipschitz_discrete}
\|\by - \by_{h} \|_{\mathbf{H}_0(\cu,\Omega)} \lesssim h^{s} + \| \mathbf{u} - \mathbf{u}_{h}\|_{\mathbb{R}^{\ell}},
\end{align}
where $s\in [0,\mathfrak{t})$ is given as in Theorem \ref{thm:extra_reg_Maxwell}. Moreover, if $\mathbf{u}_{h} \to \mathbf{u}$ in $\mathbb{R}^{\ell}$ as $h\downarrow 0$, then $j(\mathbf{u}) = \lim_{h\to 0}j_{h}(\mathbf{u}_{h})$.
\end{lemma}
\begin{proof}
We introduce the auxiliary variable $\mathsf{y}_{h}\in\mathbf{V}(\mathscr{T}_{h})$ as the solution to 
\begin{align*}
(\mu^{-1}\cu \mathsf{y}_{h}, \cu \bw_{h})_{\Omega} - \omega^2((\bsiep\cdot \mathbf{u}) \mathsf{y}_{h},\bw_{h})_{\Omega} = (\bbf,\bw_{h})_{\Omega}  \quad \forall\bw_{h}\in \mathbf{V}(\mathscr{T}_{h}).
\end{align*}
The use of the triangle inequality yields
\begin{align}\label{eq:estimate_triangle_st}
\|\by - \by_{h} \|_{\mathbf{H}_0(\cu,\Omega)} \leq \|\by - \mathsf{y}_{h} \|_{\mathbf{H}_0(\cu,\Omega)}  + \|\mathsf{y}_{h} - \by_{h} \|_{\mathbf{H}_0(\cu,\Omega)}.
\end{align}
To estimate $\|\by - \mathsf{y}_{h} \|_{\mathbf{H}_0(\cu,\Omega)}$ in \eqref{eq:estimate_triangle_st}, we note that $\mathsf{y}_{h}$ corresponds to the finite element approximation of $\by$ in $\mathbf{V}(\mathscr{T}_{h})$. 
Hence, in light of the assumptions made on $\bbf$, $\mu$, and $\bsiep$, we use Theorem \ref{thm:error_estimate} to obtain $\|\by - \mathsf{y}_{h} \|_{\mathbf{H}_0(\cu,\Omega)}\lesssim h^{s}$ with $s\in[0,\mathfrak{t})$. 
On the other hand, we note that $\mathsf{y}_{h} - \by_{h} \in \mathbf{V}(\mathscr{T}_{h})$ solves the discrete problem
\begin{align*}
&(\mu^{-1}\cu (\mathsf{y}_{h} - \by_{h}), \cu \bw_{h})_{\Omega} - \omega^{2}((\bsiep\cdot \mathbf{u}) (\mathsf{y}_{h} - \by_{h}),\bw_{h})_{\Omega} \\
& \qquad \qquad \qquad \qquad
= \omega^{2}((\bsiep\cdot (\mathbf{u} - \mathbf{u}_{h}))\by_{h},\bw_{h})_{\Omega} \quad \forall\bw_{h}\in \mathbf{V}(\mathscr{T}_{h}).
\end{align*}
The well-posedness of the latter discrete problem in combination with the estimate $\|\by_{h} \|_{\Omega} \lesssim \|\bbf\|_{\Omega}$ implies that $\|\mathsf{y}_{h} - \by_{h} \|_{\mathbf{H}_0(\cu,\Omega)} \lesssim \| \mathbf{u} - \mathbf{u}_{h}\|_{\mathbb{R}^{\ell}}$. 
Therefore, \eqref{eq:Lipschitz_discrete} follows from the estimates provided for $\|\by - \mathsf{y}_{h} \|_{\mathbf{H}_0(\cu,\Omega)}$ and $\|\mathsf{y}_{h} - \by_{h} \|_{\mathbf{H}_0(\cu,\Omega)}$ and \eqref{eq:estimate_triangle_st}.

The second result of the theorem stems from the convergence $\mathbf{u}_{h} \to \mathbf{u}$ in $\mathbb{R}^{\ell}$ as $h\downarrow 0$, and the convergence $\by_{h} \to \by$ in $\mathbf{H}_0(\cu,\Omega)$, which follows from \eqref{eq:Lipschitz_discrete}. 
\end{proof}

We now prove that the sequence of discrete global solutions $\{ \mathbf{u}^{*}_h \}_{h>0}$ contains subsequences that converge, as $h \downarrow 0$, to global solutions of problem \eqref{eq:weak_min_problem}--\eqref{eq:weak_st_eq}.

\begin{theorem}[convergence of global solutions]
\label{thm:convergence_discrete_sol}
Let $\mathbf{u}^{*}_h\in U_{ad}$ be a global solution of the discrete optimal control problem. If assumption \eqref{eq:assumption_f_and_PW} holds, then there exist subsequences of $\{\mathbf{u}^{*}_{h}\}_{h>0}$ (still indexed by $h$) such that $\mathbf{u}^{*}_h \to \mathbf{u}^{*}$ in $\mathbb{R}^{\ell}$, as $h \downarrow 0$. Here, $\mathbf{u}^{*}\in U_{ad}$ corresponds to a global solution of the optimal control problem \eqref{eq:weak_min_problem}--\eqref{eq:weak_st_eq}. 
\end{theorem}
\begin{proof}
Since, for every $h>0$, $\mathbf{u}^{*}_{h}\in U_{ad}$, we have that the sequence $\{\mathbf{u}^{*}_{h}\}_{h>0}$ is uniformly bounded. Hence, there exists a subsequence (still indexed by $h$) such that $\mathbf{u}^{*}_{h}\to \mathbf{u}^{*}$ in $\mathbb{R}^{\ell}$ as $h\downarrow 0$. We now prove that $\mathbf{u}^{*}\in U_{ad}$ solves \eqref{eq:weak_min_problem}--\eqref{eq:weak_st_eq}.

Let $\tilde{\mathbf{u}}\in U_{ad}$ be a global solution to \eqref{eq:weak_min_problem}--\eqref{eq:weak_st_eq}. We denote by $\{\tilde{\mathbf{u}}_{h}\}_{h>0}\subset U_{ad}$ a sequence such that $\tilde{\mathbf{u}}_{h} \to \tilde{\mathbf{u}}$ as $h\downarrow 0$. Hence, the global optimality of $\tilde{\mathbf{u}}$, Lemma \ref{lemma:error_estim_st}, the global optimality of $\mathbf{u}^{*}_{h}$, and the convergence $\tilde{\mathbf{u}}_{h} \to \tilde{\mathbf{u}}$ in $\mathbb{R}^{\ell}$ imply the bound
\begin{align*}
j(\tilde{\mathbf{u}})\leq j(\mathbf{u}^{*}) = \lim_{h\downarrow 0}j_{h}(\mathbf{u}^{*}_{h}) \leq \lim_{h\downarrow 0}j_{h}(\tilde{\mathbf{u}}_{h}) = j(\tilde{\mathbf{u}}).
\end{align*}
This proves that $\mathbf{u}^{*}$ is a global solution to \eqref{eq:weak_min_problem}--\eqref{eq:weak_st_eq}.
\end{proof}

In what follows, we prove that strict local solutions of problem \eqref{eq:weak_min_problem}--\eqref{eq:weak_st_eq} can be approximated by local solutions of the discrete optimal control problem.

\begin{theorem}[convergence of local solutions]
\label{thm:convergence_discrete_sol_local}
Let $\mathbf{u}^{*}\in U_{ad}$ be a strict local minimum of \eqref{eq:weak_min_problem}--\eqref{eq:weak_st_eq}. If assumption \eqref{eq:assumption_f_and_PW} holds, then there exists a sequence of local minima $\{\mathbf{u}^{*}_h\}_{h>0}$ of the discrete problem satisfying  $\mathbf{u}^{*}_h \to \mathbf{u}^{*}$ in $\mathbb{R}^{\ell}$ and $j_{h}(\mathbf{u}^{*}_{h}) \to j(\mathbf{u}^{*})$ in $\mathbb{R}$ as $h\downarrow 0$.
\end{theorem}
\begin{proof}
Since $\mathbf{u}^{*}$ is a \emph{strict} local minimum of \eqref{eq:weak_min_problem}--\eqref{eq:weak_st_eq}, there exists $\delta > 0$ such that the problem 
\begin{align}\label{eq:local_problem}
\min\{ j(\mathbf{u}): \mathbf{u} \in U_{ad}\cap B_{\delta}(\mathbf{u}^{*})\} \quad \text{with} \quad B_{\delta}(\mathbf{u}^{*}):=\{ \mathbf{u} \in \mathbb{R}^{\ell} : \|\mathbf{u}^{*}-\mathbf{u}\|_{\mathbb{R}^{\ell}}\leq \delta\},
\end{align}
admits $\mathbf{u}^{*}$ as the unique solution. On the other hand, let us consider, for $h>0$, the discrete problem: Find $\min\{j_{h}(\mathbf{u}_h): \mathbf{u}_h\in U_{ad}\cap B_{\delta}(\mathbf{u}^{*})\}$. We notice that this problem admits a solution. In fact, the set $U_{ad}\cap B_{\delta}(\mathbf{u}^{*})$ is closed, bounded, and nonempty.

Let $\mathbf{u}^{*}_h$ be a global solution of $\min\{j_{h}(\mathbf{u}_h): \mathbf{u}_h\in U_{ad,h}\cap B_{\delta}(\mathbf{u}^{*})\}$. We proceed as in the proof of Theorem \ref{thm:convergence_discrete_sol} to conclude the existence of a subsequence of $\{\mathbf{u}^{*}_h\}_{h>0}$ such that it converges to a solution of problem \eqref{eq:local_problem}. Since the latter problem admits a unique solution $\mathbf{u}^{*}$, we must have $\mathbf{u}^{*}_h \rightarrow \mathbf{u}^{*}$ in $\mathbb{R}^{\ell}$ as $h \downarrow 0$. This convergence also implies, for $h$ small enough, that the constraint $\mathbf{u}^{*}_{h} \in B_{\delta}(\mathbf{u}^{*})$ is not active. As a result, $\mathbf{u}^{*}_h$ is a local solution of the discrete optimal control problem. Finally, Lemma \ref{lemma:error_estim_st} yields that $\lim_{h\to 0}j_{h}(\mathbf{u}^{*}_{h}) = j(\mathbf{u}^{*})$, in view of the convergence $\mathbf{u}^{*}_h \rightarrow \mathbf{u}^{*}$ in $\mathbb{R}^{\ell}$.
\end{proof}

%%%%%%%%%%%%%%%%%%%%%%%%%%%%%%%%%%%%%%%%%%%%%%%%%%%
%%%%%%%%%%%%%%%%%%%%%%%%%%%%%%%%%%%%%%%%%%%%%%%%%%%
%%%%%%%%%%%%%%%%%%%%%%%%%%%%%%%%%%%%%%%%%%%%%%%%%%%
%%%%%%%%%%%%%%%%%%%%%%%%%%%%%%%%%%%%%%%%%%%%%%%%%%%

\subsection{A priori error estimates}

Let $\{\mathbf{u}^{*}_{h}\}_{h>0} \subset U_{ad}$ be a sequence of local minima of the discrete control problems such that $\mathbf{u}^{*}_{h} \to \mathbf{u}^{*}$ in $\mathbb{R}^{\ell}$ as $h\downarrow 0$, where $\mathbf{u}^{*}\in U_{ad}$ is a strict local solution of \eqref{eq:weak_min_problem}--\eqref{eq:weak_st_eq}; see Theorem \ref{thm:convergence_discrete_sol_local}. 
In this section we obtain an order of convergence for the approximation error $\mathbf{u}^{*} - \mathbf{u}^{*}_{h}$ in $\mathbb{R}^{\ell}$. 

Let $\mathbf{u}\in U_{ad}$ be arbitrary and let $\by\in \mathbf{H}_{0}(\cu,\Omega)$ be the unique solution to \eqref{eq:weak_st_eq} associated to $\mathbf{u}$. Let $\bp\in \mathbf{H}_{0}(\cu,\Omega)$ be the unique solution to problem \eqref{eq:adj_eq}. We introduce $\bp_{h}\in\mathbf{V}(\mathscr{T}_{h})$ as the finite element approximation of $\bp$. In order to prove the remaining results of this section, we assume that there exists $\mathfrak{s}\in(0,1]$, such that
\begin{equation}\label{eq:assumption_p-ph}
\|\bp - \bp_{h} \|_{\Omega} \lesssim h^{\mathfrak{s}}.
\end{equation}
With this assumption at hand, we prove the following auxiliary result.

\begin{proposition}[error estimate]\label{prop:error_estimate_adj}
Let $\bp^{*}\in \mathbf{H}_0(\cu,\Omega)$ and $\bp^{*}_{h}\in \mathbf{V}(\mathscr{T}_{h})$ be the unique solutions to \eqref{eq:adj_eq} and \eqref{eq:discrete_adjoint_equation}, respectively. Let us assume that assumptions \eqref{eq:assumption_f_and_PW} and \eqref{eq:assumption_p-ph} hold. Then, we have the error estimate
\begin{align*}\label{eq:Lipschitz_discrete_adj}
\|\bp^{*} - \bp^{*}_{h} \|_{\Omega} \lesssim h^{\min\{s,\mathfrak{s}\}} + \| \mathbf{u}^{*} - \mathbf{u}^{*}_{h}\|_{\mathbb{R}^{\ell}},
\end{align*}
where $\mathfrak{s}\in(0,1]$ and $s\in[0,\mathfrak{t})$ with $\mathfrak{t}$ given as in Theorem \ref{thm:error_estimate}.
\end{proposition}
\begin{proof}
The use of the triangle inequality yields
\begin{align}\label{eq:estimate_triangle_adj}
\|\bp^{*} - \bp^{*}_{h} \|_{\Omega} 
\lesssim
\|\bp^{*} - \mathsf{p}_{h} \|_{\Omega}  + \|\mathsf{p}_{h} - \bp^{*}_{h} \|_{\Omega},
\end{align}
where $\mathsf{p}_{h}\in\mathbf{V}(\mathscr{T}_{h})$ is the unique solution to 
\begin{align}\label{eq:aux_ph}
&(\mu^{-1}\cu \mathsf{p}_{h}, \cu \bw_{h})_{\Omega} - \omega^{2}((\bsiep\cdot \mathbf{u}^{*})\mathsf{p}_{h},\bw_{h})_{\Omega} \\ 
& \qquad \qquad = (\overline{\by^{*} - \by_{\Omega}},\bw_{h})_{\Omega} + (\overline{\cu \by^{*} - \bE_{\Omega}}, \cu \bw_{h})_{\Omega} \quad \forall \bw_{h} \in \mathbf{V}(\mathscr{T}_{h}). \nonumber
\end{align}
We notice that $\mathsf{p}_{h}$ corresponds to the finite element approximation of $\bp^{*}$ in $\mathbf{V}(\mathscr{T}_{h})$. Assumption \eqref{eq:assumption_p-ph} thus yields $\|\bp^{*} - \mathsf{p}_{h} \|_{\Omega}\lesssim h^{\mathfrak{s}}$. On the other hand, we note that $\mathsf{p}_{h} - \bp^{*}_{h} \in \mathbf{V}(\mathscr{T}_{h})$ solves 
\begin{align*}
&(\mu^{-1}\cu (\mathsf{p}_{h} - \bp^{*}_{h}), \cu \bw_{h})_{\Omega} - \omega^{2}((\bsiep\cdot \mathbf{u}^{*})(\mathsf{p}_{h} - \bp^{*}_{h}),\bw_{h})_{\Omega} = (\overline{\by^{*} - \by^{*}_{h}},\bw_{h})_{\Omega} \\  
& \quad \quad + (\overline{\cu (\by^{*} - \by^{*}_{h})}, \cu \bw_{h})_{\Omega} + \omega^{2}((\bsiep\cdot (\mathbf{u}^{*} - \mathbf{u}_{h}^{*}))\bp^{*}_{h},\bw_{h})_{\Omega} \quad \forall \bw_{h} \in \mathbf{V}(\mathscr{T}_{h}).
\end{align*}
The well-posedness of the previous discrete problem, the estimate $\|\bp^{*}_{h} \|_{\mathbf{H}_0(\cu,\Omega)} \lesssim \|\bbf\|_{\Omega} + \|\by_{\Omega}\|_{\Omega} + \|\bE_{\Omega}\|_{\Omega}$, and Lemma \ref{lemma:error_estim_st} imply that 
\begin{align*}
\|\mathsf{p}_{h} - \bp^{*}_{h} \|_{\Omega} 
\lesssim \|\by^{*} - \by^{*}_{h} \|_{\mathbf{H}_0(\cu,\Omega)} + \|\mathbf{u}^{*} - \mathbf{u}^{*}_{h}\|_{\mathbb{R}^{\ell}} \lesssim h^{s} + \| \mathbf{u}^{*} - \mathbf{u}^{*}_{h}\|_{\mathbb{R}^{\ell}}.
\end{align*}
Using in \eqref{eq:estimate_triangle_adj} the estimates obtained for $\|\bp^{*} - \mathsf{p}_{h} \|_{\Omega}$ and $\|\mathsf{p}_{h} - \bp^{*}_{h} \|_{\Omega}$ ends the proof.
\end{proof}

We now provide a first estimate for $\|\mathbf{u}^{*} - \mathbf{u}^{*}_{h}\|_{\mathbb{R}^{\ell}}$.

\begin{lemma}[auxiliary estimate]\label{lemma:aux_estimate}
Let $\mathbf{u}^{*}\in U_{ad}$ such that it satisfies the second-order optimality condition \eqref{eq:second_order_equivalent}. If assumptions \eqref{eq:assumption_f_and_PW} and \eqref{eq:assumption_p-ph} hold, then there exists $h_{\dagger} > 0$ such that 
\begin{align}\label{eq:aux_estimate}
\frac{\nu}{2}\|\mathbf{u}^{*}-\mathbf{u}^{*}_h\|_{\mathbb{R}^{\ell}}^2 \leq [j'(\mathbf{u}^{*}_h)-j'(\mathbf{u}^{*})](\mathbf{u}^{*}_h-\mathbf{u}^{*}) \quad \forall h < h_{\dagger}.
\end{align}
\end{lemma}
\begin{proof}
We divide the proof into two steps.

\underline{Step 1.} Let us prove that $\mathbf{u}^{*}_{h} - \mathbf{u}^{*} \in \mathbf{C}_{\mathbf{u}^{*}}^{\tau}$ when $h$ is small enough; we recall that $\mathbf{C}_{\mathbf{u}^{*}}^{\tau}$ is defined in \eqref{def:cone_kappa}. Since $\mathbf{u}^{*}_{h}\in U_{ad}$ the sign condition \eqref{eq:sign_cond} holds. To prove the remaining condition \eqref{eq:v_i_tau}, we introduce the term $\bar{\mathfrak{d}}_{h}\in\mathbb{R}^{\ell}$ as follows:
\begin{align*}
(\bar{\mathfrak{d}_{h}})_{k}:=\alpha(\mathbf{u}^{*}_{h})_{k} 
+ \omega^{2} \mathfrak{Re}\left\{\int_{\Omega_{k}}\bsiep\by_{h}^{*}\cdot\bp_{h}^{*}\right\}, \qquad k\in\{1,\ldots,\ell\}.
\end{align*}
Invoke the term $\bar{\mathfrak{d}}\in\mathbb{R}^{\ell}$ defined by $\bar{\mathfrak{d}}_{k}:=\alpha\mathbf{u}^{*}_{k} 
+ \omega^{2} \mathfrak{Re}\{\int_{\Omega_{k}}\bsiep\by^{*}\cdot\bp^{*}\}$. A simple computation thus reveals that
\begin{align*}
\|\bar{\mathfrak{d}} - \bar{\mathfrak{d}}_{h}\|_{\mathbb{R}^{\ell}}
\leq & \,
\alpha\|\mathbf{u}^{*} - \mathbf{u}^{*}_{h}\|_{\mathbb{R}^{\ell}} + \omega^{2} \left(\sum_{k=1}^{\ell}\mathfrak{Re}\left\{\int_{\Omega_{k}}\bsiep(\by^{*}\cdot\bp^{*} - \by^{*}_{h}\cdot\bp^{*}_{h})\right\}^{2}\right)^{\frac{1}{2}} \\
\leq & \,
\alpha\|\mathbf{u}^{*} - \mathbf{u}^{*}_{h}\|_{\mathbb{R}^{\ell}} + \omega^{2} \left(\sum_{k=1}^{\ell}\left|\int_{\Omega_{k}}\bsiep(\by^{*}\cdot\bp^{*} - \by^{*}_{h}\cdot\bp^{*}_{h})\right|^{2}\right)^{\frac{1}{2}} \\
\lesssim & \,
\|\mathbf{u}^{*} - \mathbf{u}^{*}_{h}\|_{\mathbb{R}^{\ell}} + \|\bsiep\|_{\rL^{\infty}(\Omega;\mathbb{C})}\int_{\Omega}|\by^{*}\cdot\bp^{*} - \by^{*}_{h}\cdot\bp^{*}_{h}| \\
\lesssim 
& \, \|\mathbf{u}^{*} - \mathbf{u}^{*}_{h}\|_{\mathbb{R}^{\ell}} + (\|\by^{*} - \by^{*}_{h}\|_{\Omega}\|\bp^{*}\|_{\Omega} + \|\by^{*}_{h}\|_{\Omega}\|\bp^{*}  - \bp^{*}_{h}\|_{\Omega}).
%& \,\alpha\|\mathbf{u}^{*}_{h} - \mathbf{u}^{*}\|_{\mathbb{R}^{\ell}} + \omega^{2}\sqrt{\ell} \|\bsiep\|_{\mathbf{L}^{\infty}(\Omega)} (\|\by^{*} - \by^{*}_{h}\|_{\Omega}\|\bp^{*}\|_{\Omega} + \|\by^{*}_{h}\|_{\Omega}\|\bp^{*}  - \bp^{*}_{h}\|_{\Omega}).
\end{align*}
Hence, in view of Lemma \ref{lemma:error_estim_st}, Proposition \ref{prop:error_estimate_adj}, and the convergence $\mathbf{u}^{*}_{h} \to \mathbf{u}^{*}$ in $\mathbb{R}^{\ell}$, as $h\downarrow 0$, we conclude that there exists $h_{\circ}>0$ such that $\|\bar{\mathfrak{d}} - \bar{\mathfrak{d}}_{h}\|_{\mathbb{R}^{\ell}} < \tau$ for all $h<h_{\circ}$. 

Now, let $k\in\{1,\ldots,\ell\}$ be fixed but arbitrary. 
If, on one hand, $\bar{\mathfrak{d}}_{k} > \tau$, then $(\bar{\mathfrak{d}}_{h})_{k} > 0$ and, in view of inequalities \eqref{eq:var_ineq_with_adj_state} and \eqref{eq:discrete_var_ineq}, we also have that $\mathbf{u}^{*}_{k} = (\mathbf{u}^{*}_{h})_{k} = \mathbf{a}_{k}$. 
Consequently, $(\mathbf{u}^{*}_{h})_{k}  - \mathbf{u}^{*}_{k} = 0$. If, on the other hand, $\bar{\mathfrak{d}}_{k} < - \tau$, then $(\bar{\mathfrak{d}}_{h})_{k} < 0$ and $\mathbf{u}^{*}_{k} = (\mathbf{u}^{*}_{h})_{k} = \mathbf{b}_{k}$, and thus $(\mathbf{u}^{*}_{h})_{k}  - \mathbf{u}^{*}_{k} = 0$. 
Therefore, $\mathbf{u}^{*}_{h} - \mathbf{u}^{*}$ satisfies condition  \eqref{eq:v_i_tau} and thus it belongs to $\mathbf{C}_{\mathbf{u}^{*}}^{\tau}$.

\underline{Step 2.} Let us prove estimate \eqref{eq:aux_estimate}. Since $\mathbf{u}^{*}_{h} - \mathbf{u}^{*} \in \mathbf{C}_{\mathbf{u}^{*}}^{\tau}$ for all $h<h_{\circ}$, we are allowed to use $\mathbf{v}=\mathbf{u}^{*}_{h} - \mathbf{u}^{*}$ in the second-order optimality condition \eqref{eq:second_order_equivalent} to obtain
\begin{equation}\label{eq:j''_u_uh}
j''(\mathbf{u}^{*})(\mathbf{u}^{*}_{h} - \mathbf{u}^{*})^2  \geq \nu \|\mathbf{u}^{*}_{h} - \mathbf{u}^{*}\|_{\mathbb{R}^{\ell}}^2.
\end{equation}
On the other hand, the use of the mean value theorem yields $(j'(\mathbf{u}^{*}_{h})-j'(\mathbf{u}^{*}))(\mathbf{u}^{*}_{h} - \mathbf{u}^{*})=j''(\mathbf{u}^{*}_{\theta})(\mathbf{u}^{*}_{h} - \mathbf{u}^{*})^2$, where $\mathbf{u}^{*}_{\theta}=\mathbf{u}^{*}+\theta_{h}(\mathbf{u}^{*}_{h} - \mathbf{u}^{*})$ with $\theta_{h} \in (0,1)$. This identity in combination with inequality \eqref{eq:j''_u_uh} results in
\begin{equation}\label{eq:ineq_j''_u_uh}
\nu \|\mathbf{u}^{*}_{h} - \mathbf{u}^{*}\|_{\mathbb{R}^{\ell}}^2
\leq (j'(\mathbf{u}^{*}_{h})-j'(\mathbf{u}^{*}))(\mathbf{u}^{*}_{h} - \mathbf{u}^{*}) + (j''(\mathbf{u}^{*}) - j''(\mathbf{u}^{*}_{\theta}))(\mathbf{u}^{*}_{h} - \mathbf{u}^{*})^2.
\end{equation}
The convergence $\mathbf{u}^{*}_{\theta} \to \mathbf{u}^{*}$ in $\mathbb{R}^{\ell}$ as $h\downarrow 0$ and estimate \eqref{eq:Lipschitz_property} allow us to conclude the existence of $0 < h_{\dagger} \leq h_{\circ}$  such that 
\begin{align*}
(j''(\mathbf{u}^{*}) - j''(\mathbf{u}^{*}_{\theta}))(\mathbf{u}^{*}_{h} - \mathbf{u}^{*})^2
\leq \frac{\nu}{2}\|\mathbf{u}^{*}_{h} - \mathbf{u}^{*}\|_{\mathbb{R}^{\ell}}^2 \quad \forall h < h_{\dagger}.
\end{align*}
The use of the latter inequality in \eqref{eq:ineq_j''_u_uh} concludes the proof.
\end{proof}

We are now in position to present the main result of this section.

\begin{theorem}[a priori error estimate]\label{thm:a_priori_estimate_ocp}
Let $\mathbf{u}^{*}\in U_{ad}$ be such that it satisfies the second-order optimality condition \eqref{eq:second_order_equivalent}. Then, if assumptions \eqref{eq:assumption_f_and_PW} and \eqref{eq:assumption_p-ph} hold, there exists $h_{\dagger} > 0$ such that 
\begin{align*}
\|\mathbf{u}^{*}-\mathbf{u}^{*}_h\|_{\mathbb{R}^{\ell}} \lesssim  h^{\min\{s,\mathfrak{s}\}} \quad \forall h < h_{\dagger},
\end{align*}
where $\mathfrak{s}\in(0,1]$ and $s\in[0,\mathfrak{t})$ with $\mathfrak{t}$ given as in Theorem \ref{thm:error_estimate}.
\end{theorem}
\begin{proof}
Invoke estimate \eqref{eq:aux_estimate},  the variational inequality \eqref{eq:variational_inequality} with $\mathbf{u} = \mathbf{u}^{*}_{h}$, and inequality  $-j_{h}^{\prime}(\mathbf{u}^{*}_{h})(\mathbf{u}^{*}_{h} - \mathbf{u}^{*} ) \geq  0$ to obtain
\begin{equation*}
\frac{\nu}{2}
\|\mathbf{u}^{*} - \mathbf{u}^{*}_{h}\|_{\mathbb{R}^{\ell}}^{2} \leq 
[j'(\mathbf{u}^{*}_{h}) - j'(\mathbf{u}^{*})](\mathbf{u}^{*}_{h} - \mathbf{u}^{*}) 
\leq  
[j'(\mathbf{u}^{*}_{h}) - j_{h}'(\mathbf{u}^{*}_{h})](\mathbf{u}^{*}_{h} - \mathbf{u}^{*}).
\end{equation*}
A direct computation reveals that 
\begin{align*}
[j'(\mathbf{u}^{*}_{h}) - j_{h}'(\mathbf{u}^{*}_{h})](\mathbf{u}^{*}_{h} - \mathbf{u}^{*}) = \omega^{2}\sum_{k=1}^{\ell}\mathfrak{Re}\left\{\int_{\Omega_{k}}\bsiep(\by_{\mathbf{u}^{*}_{h}}\cdot\bp_{\mathbf{u}^{*}_{h}} - \by^{*}_{h}\cdot\bp^{*}_{h})\right\}(\mathbf{u}^{*}_{h} - \mathbf{u}^{*}) _{k},
\end{align*}
where $\by_{\mathbf{u}^{*}_{h}}\in \mathbf{H}_0(\cu,\Omega)$ corresponds to the unique solution to problem \eqref{eq:weak_st_eq} with $\mathbf{u}=\mathbf{u}^{*}_{h}$, and $\bp_{\mathbf{u}^{*}_{h}}\in \mathbf{H}_0(\cu,\Omega)$ is the unique solution to problem \eqref{eq:adj_eq} with $\mathbf{u}=\mathbf{u}^{*}_{h}$ and $\by=\by_{\mathbf{u}^{*}_{h}}$. 
Hence, by proceeding as in Step 1 of the proof of Lemma \ref{lemma:aux_estimate} we obtain
\begin{align}\label{eq:estimate_uh-u_mu}
\frac{\nu}{2}\|\mathbf{u}^{*}_{h} - \mathbf{u}^{*}\|_{\mathbb{R}^{\ell}}
\lesssim
\|\by^{*}_{h} - \by_{\mathbf{u}^{*}_{h}}\|_{\Omega}\|\bp_{\mathbf{u}^{*}_{h}}\|_{\Omega} + \|\by^{*}_{h}\|_{\Omega}\|\bp^{*}_{h}  - \bp_{\mathbf{u}^{*}_{h}}\|_{\Omega}.
\end{align}
Using, in \eqref{eq:estimate_uh-u_mu}, the stability bounds $\|\by^{*}_{h}\|_{\Omega}\lesssim \|\bbf\|_{\Omega}$ and $\|\bp_{\mathbf{u}^{*}_{h}}\|_{\Omega} \lesssim \|\by_{\Omega}\|_{\Omega} + \|\mathbf{E}_{\Omega}\|_{\Omega} +  \|\bbf\|_{\Omega}$ in combination with the a priori error estimate from Theorem \ref{thm:error_estimate} we arrive at
\begin{equation}\label{eq:estimate_u-uh_w_p}
\|\mathbf{u}^{*} - \mathbf{u}^{*}_{h}\|_{\mathbb{R}^{\ell}} 
\lesssim
h^{s} + \|\bp^{*}_{h}  - \bp_{\mathbf{u}^{*}_{h}}\|_{\Omega}.
\end{equation}
We now bound $\|\bp^{*}_{h} - \bp_{\mathbf{u}^{*}_{h}}\|_{\Omega}$. 
We introduce $\hat{\mathsf{p}}_{h} \in \mathbf{V}(\mathscr{T}_{h})$, defined as the finite element approximation of $\bp_{\mathbf{u}^{*}_{h}}$. The use of the triangle inequality and assumption \eqref{eq:assumption_p-ph} yield
\begin{align*}
\|\bp^{*}_{h} - \bp_{\mathbf{u}^{*}_{h}}\|_{\Omega}
\leq 
\|\bp^{*}_{h} - \hat{\mathsf{p}}_{h}\|_{\Omega} + \|\hat{\mathsf{p}}_{h} - \bp_{\mathbf{u}^{*}_{h}}\|_{\Omega} 
\lesssim
\|\bp^{*}_{h} - \hat{\mathsf{p}}_{h}\|_{\Omega} + h^{\mathfrak{s}}.
\end{align*}
We notice that $\bp^{*}_{h} - \hat{\mathsf{p}}_{h}\in \mathbf{V}(\mathscr{T}_{h})$ solves the discrete problem
\begin{align*}
&(\mu^{-1}\cu (\bp^{*}_{h} - \hat{\mathsf{p}}_{h}), \cu \bw_{h})_{\Omega} - \omega^2((\bsiep\cdot \mathbf{u}^{*}_{h})(\bp^{*}_{h} - \hat{\mathsf{p}}_{h}),\bw_{h})_{\Omega} \\ 
&\qquad \qquad = (\overline{\by^{*}_{h} - \by_{\mathbf{u}^{*}_{h}}
},\bw_{h})_{\Omega} + (\overline{\cu (\by^{*}_{h} - \by_{\mathbf{u}^{*}_{h}})}, \cu \bw_{h})_{\Omega} \quad \forall \bw_{h}\in \mathbf{V}(\mathscr{T}_{h}).
\end{align*}
The stability of this problem provides the bound $\|\bp^{*}_{h} - \hat{\mathsf{p}}_{h}\|_{\Omega} \lesssim \|\by^{*}_{h} - \by_{\mathbf{u}^{*}_{h}}\|_{\mathbf{H}_0(\cu,\Omega)}\lesssim h^{s}$, upon using the error estimate from Theorem \ref{thm:error_estimate}. We have thus concluded that $\|\bp^{*}_{h} - \bp_{\mathbf{u}^{*}_{h}}\|_{\Omega} \lesssim h^{\min\{s,\mathfrak{s}\}}$ which, in light of \eqref{eq:estimate_u-uh_w_p}, concludes the proof.
\end{proof}

For the last result of this section, we assume that there exist $\tilde{\mathfrak{s}}\in(0,1]$, such that
\begin{align}\label{eq:assumption_p-ph_curl}
\|\cu(\bp - \bp_{h}) \|_{\Omega} \lesssim h^{\tilde{\mathfrak{s}}},
\end{align}
where $\bp\in \mathbf{H}_{0}(\cu,\Omega)$ is the solution of problem \eqref{eq:adj_eq} and $\bp_{h}\in\mathbf{V}(\mathscr{T}_{h})$ corresponds to its finite element approximation.

\begin{corollary}[error estimate]\label{coro:error_estimate}
Let $\mathbf{u}^{*}\in U_{ad}$ such that it satisfies the second-order optimality condition \eqref{eq:second_order_equivalent}. If assumptions \eqref{eq:assumption_f_and_PW}, \eqref{eq:assumption_p-ph}, and \eqref{eq:assumption_p-ph_curl} hold, then there exists $h_{\dagger} > 0$ such that 
\begin{equation}\label{eq:total_error_apriori}
\|\mathbf{u}^{*}-\mathbf{u}^{*}_h\|_{\mathbb{R}^{\ell}} + \|\by^{*} - \by^{*}_{h}\|_{\mathbf{H}(\cu,\Omega)} + \|\bp^{*} - \bp^{*}_{h}\|_{\mathbf{H}(\cu,\Omega)} \lesssim h^{\min\{s,\mathfrak{s},\tilde{\mathfrak{s}}\}} \quad \forall h < h_{\dagger}.
\end{equation}
\end{corollary}
\begin{proof}
Since the bound for $\|\mathbf{u}^{*}-\mathbf{u}^{*}_h\|_{\mathbb{R}^{\ell}}$ follows from Theorem \ref{thm:a_priori_estimate_ocp}, we concentrate on the remaining terms on the left-hand side of \eqref{eq:total_error_apriori}. To estimate $\|\by^{*} - \by^{*}_{h}\|_{\mathbf{H}(\cu,\Omega)}$ we invoke the auxiliary variable $\by_{\mathbf{u}^{*}_{h}}\in \mathbf{H}_0(\cu,\Omega)$, defined as the unique solution to problem \eqref{eq:weak_st_eq} with $\mathbf{u}=\mathbf{u}^{*}_{h}$, and the triangle inequality to obtain
\begin{align*}
\|\by^{*} - \by^{*}_{h}\|_{\mathbf{H}(\cu,\Omega)}
\leq 
\|\by^{*} - \by^{*}_{\mathbf{u}^{*}_{h}}\|_{\mathbf{H}(\cu,\Omega)} + \|\by^{*}_{\mathbf{u}^{*}_{h}} - \by^{*}_{h}\|_{\mathbf{H}(\cu,\Omega)}.
\end{align*}
The error estimate from Theorem \ref{thm:error_estimate} in conjunction with the stability estimate $\|\by^{*} - \by^{*}_{\mathbf{u}^{*}_{h}}\|_{\mathbf{H}(\cu,\Omega)} \lesssim \|\mathbf{u}^{*}-\mathbf{u}^{*}_h\|_{\mathbb{R}^{\ell}}$ immediately yield $\|\by^{*} - \by^{*}_{h}\|_{\mathbf{H}(\cu,\Omega)} \lesssim h^{\min\{s,\mathfrak{s}\}}$ for all $h < h_{\dagger}$. To bound $\|\bp^{*} - \bp^{*}_{h}\|_{\mathbf{H}(\cu,\Omega)}$, we introduce $\mathsf{p}\in \mathbf{H}_0(\cu,\Omega)$ as the unique solution to problem \eqref{eq:adj_eq} with $\mathbf{u}=\mathbf{u}^{*}_{h}$ and $\by=\by^{*}_{h}$. We thus can write
\begin{align*}
\|\bp^{*} - \bp^{*}_{h}\|_{\mathbf{H}(\cu,\Omega)}
\leq 
\|\bp^{*} - \mathsf{p}\|_{\mathbf{H}(\cu,\Omega)} + \|\mathsf{p} - \bp^{*}_{h}\|_{\mathbf{H}(\cu,\Omega)},
\end{align*}
and utilize assumptions \eqref{eq:assumption_p-ph} and \eqref{eq:assumption_p-ph_curl}, the bound $\|\bp^{*} - \mathsf{p}\|_{\mathbf{H}(\cu,\Omega)} \lesssim \|\mathbf{u}^{*}-\mathbf{u}^{*}_h\|_{\mathbb{R}^{\ell}} + \|\by^{*} - \by^{*}_{h}\|_{\mathbf{H}(\cu,\Omega)}$, and the estimates proved for $\|\mathbf{u}^{*}-\mathbf{u}^{*}_h\|_{\mathbb{R}^{\ell}}$ and $\|\by^{*} - \by^{*}_{h}\|_{\mathbf{H}(\cu,\Omega)}$. These arguments yield that $\|\bp^{*} - \bp^{*}_{h}\|_{\mathbf{H}(\cu,\Omega)} \lesssim h^{\min\{s,\mathfrak{s},\tilde{\mathfrak{s}}\}}$ for all $h < h_{\dagger}$.
\end{proof}

%%%%%%%%%%%%%%%%%%%%%%%%%%%%%%%%%%%%%%%%%%%%%%%%%%%
%%%%%%%%%%%%%%%%%%%%%%%%%%%%%%%%%%%%%%%%%%%%%%%%%%%
%%%%%%%%%%%%%%%%%%%%%%%%%%%%%%%%%%%%%%%%%%%%%%%%%%%
%%%%%%%%%%%%%%%%%%%%%%%%%%%%%%%%%%%%%%%%%%%%%%%%%%%

\subsection{A posteriori error estimates}

In this section, we devise an a posteriori error estimator for the optimal control problem \eqref{eq:weak_min_problem}--\eqref{eq:weak_st_eq} and study its reliability and efficiency properties. We recall that, in this context, the parameter $h$ should be interpreted as $h = 1/n$, where $n\in\mathbb{N}$ is the index set in a sequence of refinements  of an initial mesh $\T_{\text{in}}$; see section \ref{sec:apost_st_eq}.
 
We start with an instrumental result for our a posteriori error analysis. 

\begin{lemma}[auxiliary estimate]
Let $\mathbf{u}^{*}\in U_{ad}$ be such that it satisfies the second-order optimality condition \eqref{eq:second_order_equivalent}. Let $C_{L}>0$ and $\nu>0$ be the constants appearing in \eqref{eq:Lipschitz_property} and \eqref{eq:second_order_equivalent}, respectively. Assume that
\begin{align}\label{eq:assumption_mesh}
\mathbf{u}^{*}_{h} - \mathbf{u}^{*}\in \mathbf{C}_{\mathbf{u}^{*}}^{\tau} \qquad \text{ and } \qquad \|\mathbf{u}^{*}_{h} - \mathbf{u}^{*}\|_{\mathbb{R}^{\ell}} \leq \nu/(2C_{L}).
\end{align}
Then, we have 
\begin{align}\label{eq:aux_estimate_apost}
\frac{\nu}{2}\|\mathbf{u}^{*} - \mathbf{u}^{*}_h\|_{\mathbb{R}^{\ell}}^2 \leq [j'(\mathbf{u}^{*}_h)-j'(\mathbf{u}^{*})](\mathbf{u}^{*}_h-\mathbf{u}^{*}).
\end{align}
\end{lemma}
\begin{proof}
Since $\mathbf{u}^{*}_{h} - \mathbf{u}^{*}\in \mathbf{C}_{\mathbf{u}^{*}}^{\tau}$, we can use $\mathbf{v}=\mathbf{u}^{*}_{h} - \mathbf{u}^{*}$ in the second-order sufficient optimality condition \eqref{eq:second_order_equivalent} to obtain
\begin{align}\label{eq:diff_1}
\nu\|\mathbf{u}^{*}_{h} - \mathbf{u}^{*}\|_{\mathbb{R}^{\ell}}^2 \leq j''(\mathbf{u}^{*})(\mathbf{u}^{*}_{h} - \mathbf{u}^{*})^2.
\end{align}
On the other hand, the use of the mean value theorem yields $(j'(\mathbf{u}^{*}_{h})-j'(\mathbf{u}^{*}))(\mathbf{u}^{*}_{h}-\mathbf{u}^{*})=j''(\mathbf{u}_{\theta}^{*})(\mathbf{u}^{*}_{h}-\mathbf{u}^{*})^2$ with $\mathbf{u}_{\theta}^{*}=\mathbf{u}^{*}+\theta_{h}(\mathbf{u}^{*}_{h} - \mathbf{u}^{*})$ and $\theta_{h} \in (0,1)$. Consequently, from inequality \eqref{eq:diff_1} we arrive at
\begin{align}\label{eq:ineq_u_h_baru}
\nu\|\mathbf{u}^{*}_h - \mathbf{u}^{*}\|_{\mathbb{R}^{\ell}}^2 
\leq (j'(\mathbf{u}^{*}_h)-j'(\mathbf{u}^{*}))(\mathbf{u}^{*}_h - \mathbf{u}^{*}) + (j''(\mathbf{u}^{*})-j''(\mathbf{u}_{\theta}^{*}))(\mathbf{u}^{*}_h - \mathbf{u}^{*}_h)^2.
\end{align}
To control the term $(j''(\mathbf{u}^{*})-j''(\mathbf{u}_{\theta}^{*}))(\mathbf{u}^{*}_h - \mathbf{u}^{*}_h)^2$ in \eqref{eq:ineq_u_h_baru}, we use estimate \eqref{eq:Lipschitz_property}, the fact that $\theta_{h} \in (0,1)$, and assumption \eqref{eq:assumption_mesh}. These arguments lead to
\begin{align*}
(j''(\mathbf{u}^{*})-j''(\mathbf{u}_{\theta}^{*}))(\mathbf{u}^{*}_{h} - \mathbf{u}^{*})^2  
\leq
C_{L}\|\mathbf{u}^{*}_{h} - \mathbf{u}^{*}\|_{\mathbb{R}^{\ell}}\|\mathbf{u}^{*}_{h} - \mathbf{u}^{*}\|_{\mathbb{R}^{\ell}}^2
\leq 
\frac{\nu}{2}\|\mathbf{u}^{*}_{h} - \mathbf{u}^{*}\|_{\mathbb{R}^{\ell}}^2.
\end{align*}
Using the latter estimation in \eqref{eq:diff_1} yields the desired inequality \eqref{eq:aux_estimate_apost}. 
\end{proof}

\subsubsection{Global reliability analysis}\label{sec:glob_rel_ana}

In the present section we prove an upper bound for the total error approximation in terms of a proposed a posteriori error estimator. The analysis relies on estimates on the error between a solution to the discrete optimal control problem and auxiliary variables that we define in what follows. 

We first define the variable $\by_{\mathbf{u}^{*}_{h}}\in \mathbf{H}_{0}(\cu,\Omega)$ as the unique solution to problem \eqref{eq:weak_st_eq} with $\mathbf{u}=\mathbf{u}^{*}_{h}$. We thus introduce, for $T\in\T_{h}$, the local error indicator associated to the discrete state equation: $\mathcal{E}_{st,T}^2:=\mathcal{E}_{T,1}^2 +\mathcal{E}_{T,2}^2$, where $\mathcal{E}_{T,1}$ and $\mathcal{E}_{T,2}$ are given by
\begin{align*}
\mathcal{E}_{T,1}^2
:= \, &
h_{T}^2\|\di(\boldsymbol{f} +\omega^2(\bsiep\cdot\mathbf{u}^{*}_{h})\by^{*}_{h})\|_{T}^2 + \dfrac{h_{T}}{2}\sum_{S\in\mathscr{S}_T^I}\left\|\jump{(\bbf + \omega^2(\bsiep\cdot\mathbf{u}^{*}_{h})\by^{*}_{h})\cdot\boldsymbol{n}}\right\|_{S}^{2},\\
\mathcal{E}_{T,2}^2
:= \, &
h_{T}^2\left\|\bbf-\cu(\mu^{-1}\cu\by^{*}_h)+\omega^2(\bsiep\cdot\mathbf{u}^{*}_{h})\by^{*}_{h}\right\|_{T}^2 \\
& ~ +\dfrac{h_{T}}{2}\sum_{S\in\mathscr{S}_T^I}\left\|\jump{\mu^{-1}\cu\by^{*}_h\times\boldsymbol{n}}\right\|_{S}^{2},
\end{align*}
respectively. The error estimator associated to the finite element discretization of the state equation is defined by $\mathcal{E}_{st,\mathscr{T}_h}^{2} := \sum_{T\in\mathscr{T}_h}\mathcal{E}_{st,T}^2$.
An application of Theorem \ref{thm:global_reli_weak} with $\mathbf{f}=\boldsymbol{f}$ and $\mathfrak{u}=\mathbf{u}^{*}_{h}$ immediately yields the a posteriori error estimate
\begin{equation}\label{eq:estimate_state_hat_discrete_st}
\|\by_{\mathbf{u}^{*}_{h}} - \by^{*}_{h}\|_{\mathbf{H}(\cu,\Omega)}
\lesssim 
\mathcal{E}_{st,\T_{h}}.
\end{equation}

Let us introduce the term $\mathsf{p}\in\mathbf{H}_{0}(\cu,\Omega)$ as the unique solution to
\begin{align}\label{eq:aux_p}
&(\mu^{-1}\cu \mathsf{p}, \cu \bw)_{\Omega} - \omega^2((\bsiep\cdot \mathbf{u}^{*}_{h})\mathsf{p},\bw)_{\Omega} \\ 
& \qquad \qquad = (\overline{\by^{*}_{h} - \by_{\Omega}},\bw)_{\Omega} + (\overline{\cu \by^{*}_{h} - \bE_{\Omega}}, \cu \bw)_{\Omega} \quad \forall \bw \in \mathbf{H}_{0}(\cu,\Omega).\nonumber
\end{align}
Define now, for $T\in\T_{h}$, the local error indicator associated to the discrete adjoint equation: $\mathcal{E}_{adj,T}^2:=\mathsf{E}_{T,1}^2+\mathsf{E}_{T,2}^2$, where $\mathsf{E}_{T,1}$ and $\mathsf{E}_{T,2}$ are defined by
\begin{align*}
\mathsf{E}_{T,1}^2 
:= \, & h_{T}^2\|\di(\overline{\by^{*}_{h} - \by_{\Omega}} + \omega^2(\bsiep\cdot\mathbf{u}^{*}_{h})\bp^{*}_{h})\|_{T}^2 \\ 
& ~ + \dfrac{h_{T}}{2} \sum_{S\in\mathscr{S}_T^I}\!\!\left\|\jump{(\overline{\by^{*}_{h} - \by_{\Omega}}+\omega^{2}(\bsiep\cdot\mathbf{u}^{*}_{h})\bp^{*}_{h})\cdot\boldsymbol{n}}\right\|_{S}^{2}, \\
\mathsf{E}_{T,2}^2 
 := \, &
  h_{T}^2\left\|\overline{\by^{*}_{h}\! - \!\by_{\Omega}} \! + \mathbf{curl}(\overline{\cu \by^{*}_{h}-\bE_{\Omega}})-\cu(\mu^{-1}\!\cu\bp^{*}_h)+\omega^2(\bsiep\cdot\mathbf{u}^{*}_{h})\bp^{*}_{h}\right\|_{T}^2 \\
& ~+\dfrac{h_{T}}{2}\sum_{S\in\mathscr{S}_T^I}\left\|\jump{(\overline{\cu\by^{*}_{h}-\bE_{\Omega}}-\mu^{-1}\cu\bp^{*}_h)\times\boldsymbol{n}}\right\|_{L^2(S)}^{2},
\end{align*}
respectively. The global error estimator associated to the finite element discretization of the state equation is thus defined by $\mathcal{E}_{adj,\mathscr{T}_h}^{2} := \sum_{T\in\mathscr{T}_h}\mathcal{E}_{adj,T}^2$. 

The next result establishes reliability properties for the discrete adjoint equation. 
\begin{lemma}[upper bound]\label{lemma:estimate_state_hat_discrete_adj}
Let $\mathsf{p}\in \mathbf{H}_{0}(\cu,\Omega)$ and $\bp^{*}_{h}\in\mathbf{V}(\T_{h})$ be the unique solutions to \eqref{eq:aux_p} and \eqref{eq:discrete_adjoint_equation}, respectively.\! If, for all $T\in\T_{h}$, $\by_{\Omega}|_{T},\mathbf{E}_{\Omega}|_{T}\in \mathbf{H}^{1}(T;\mathbb{C})$, then
\begin{equation}\label{eq:estimate_state_hat_discrete_adj}
\|\mathsf{p} - \bp^{*}_{h}\|_{\mathbf{H}(\cu,\Omega)}
\lesssim \mathcal{E}_{adj,\T_{h}}. 
\end{equation}
The hidden constant is independent of $\mathsf{p} $, $\bp^{*}_{h}$, the size of the elements in $\T_{h}$, and $\#\T_{h}$.
\end{lemma}
\begin{proof} The proof closely follows the arguments developed in the proof of Theorem~\ref{thm:global_reli_weak} (see also \cite[Lemma 3.2]{hoppe2015}). 

Define $\mathbf{e}_{\mathsf{p}}:=\mathsf{p} - \bp^{*}_{h}$. Galerkin orthogonality, the decomposition $\bw - \Pi_{h}\bw = \nabla \varphi + \boldsymbol\Psi$, with $\varphi\in \textrm{H}_0^1(\Omega)$ and $\boldsymbol\Psi\in \mathbf{H}_{0}^{1}(\Omega)$, and an elementwise integration by parts formula allow us to obtain
\begin{multline*}
(\mu^{-1} \cu \mathbf{e}_{\mathsf{p}}, \cu \bw)_{\Omega}  -\omega^{2}((\bsiep\cdot\mathbf{u}^{*}_{h})\mathbf{e}_{\mathsf{p}},\bw)_{\Omega} = 
\sum_{T\in\T_{h}}\!\!(\overline{\by^{*}_{h}\! - \!\by_{\Omega}} \! + \mathbf{curl}(\overline{\cu \by^{*}_{h}  - \bE_{\Omega}}) \\
 - \cu(\mu^{-1}\!\cu\bp^{*}_h) + \omega^{2}(\bsiep\cdot\mathbf{u}^{*}_{h})\bp^{*}_{h},\boldsymbol\Psi)_{T} +  \sum_{S\in\mathcal{S}}(\llbracket (\overline{\cu\by^{*}_{h}-\bE_{\Omega}}
-\mu^{-1}\cu\bp^{*}_h)\times\boldsymbol{n}\rrbracket,\boldsymbol\Psi)_{S}\\ 
-\sum_{T\in \T_{h}} (\di(\overline{\by^{*}_{h} - \by_{\Omega}} +\omega^{2}(\bsiep\cdot\mathbf{u}^{*}_{h})\bp^{*}_{h}),\varphi)_{T}
+ \sum_{S\in\mathcal{S}}(\jump{(\overline{\by^{*}_{h} - \by_{\Omega}} +\omega^{2}(\bsiep\cdot\mathbf{u}^{*}_{h})\bp^{*}_{h})\cdot\boldsymbol{n}}, \varphi)_{S}.
\end{multline*}
Hence, using $\bw = \mathbf{e}_{\mathsf{p}}$, an analogous estimate of \eqref{eq:coercivity_a} for $\mathbf{e}_{\mathsf{p}}$, basic inequalities, the estimates in \eqref{eq:property_op_decom}, and the finite number of overlapping patches, we arrive at $\|\mathbf{e}_{\mathsf{p}}\|_{\mathbf{H}(\cu,\Omega)}^{2}
\lesssim \mathcal{E}_{adj,\T_{h}}\|\mathbf{e}_{\mathsf{p}}\|_{\mathbf{H}(\cu,\Omega)}$, which concludes the proof.
\end{proof} 

After having defined error estimators associated to the discretization of the state and adjoint equations, we define an a posteriori error estimator for the discrete optimal control problem which can be decomposed as the sum of two contributions:
\begin{align}\label{def:error_estimator_ocp}
\mathcal{E}_{ocp,\T_{h}}^2:=\mathcal{E}_{st,\T_{h}}^2 + \mathcal{E}_{adj,\T_{h}}^2.
\end{align}

We now state and prove the main result of this section. 

\begin{theorem}[global reliability]\label{thm:global_rel}
Let $\mathbf{u}^{*}\in U_{ad}$ be such that it satisfies the second-order optimality condition \eqref{eq:second_order_equivalent}. Let $\mathbf{u}^{*}_{h}$ be a local minimum of the discrete optimal control problem with $\by^{*}_{h}$ and $\bp^{*}_{h}$ being the corresponding state and adjoint state, respectively. If, for all $T\in\T_{h}$, $\bbf|_{T},\by_{\Omega}|_{T},\mathbf{E}_{\Omega}|_{T}\in \mathbf{H}^{1}(T;\mathbb{C})$ and assumption \eqref{eq:assumption_mesh} holds, then
\begin{align*}\label{eq:global_rel}
\|\bp^{*} - \bp^{*}_{h}\|_{\mathbf{H}(\cu,\Omega)} + \|\by^{*} - \by^{*}_{h}\|_{\mathbf{H}(\cu,\Omega)} + \|\mathbf{u}^{*} - \mathbf{u}^{*}_{h}\|_{\mathbb{R}^{\ell}}
\lesssim 
\mathcal{E}_{ocp,\T_{h}},
\end{align*}
with a hidden constant that is independent of continuous and discrete optimal variables, the size of the elements in $\T_{h}$, and $\#\T_{h}$.
\end{theorem}
\begin{proof}
We proceed in three steps.

\underline{Step 1.} ($\|\mathbf{u}^{*} - \mathbf{u}^{*}_{h}\|_{\mathbb{R}^{\ell}} \lesssim \mathcal{E}_{ocp,\T_{h}}$) 
Since we have assumed \eqref{eq:assumption_mesh}, we are in position to use estimate \eqref{eq:aux_estimate_apost}. The latter, the variational inequality \eqref{eq:variational_inequality} with $\mathbf{u} = \mathbf{u}^{*}_{h}$, and inequality  $-j_{h}^{\prime}(\mathbf{u}^{*}_{h})(\mathbf{u}^{*}_{h} - \mathbf{u}^{*} ) \geq  0$ yield the bound
\begin{align*}
\|\mathbf{u}^{*} - \mathbf{u}^{*}_{h}\|_{\mathbb{R}^{\ell}}^{2} \lesssim 
[j'(\mathbf{u}^{*}_{h}) - j'(\mathbf{u}^{*})](\mathbf{u}^{*}_{h} - \mathbf{u}^{*}) 
\leq  
[j'(\mathbf{u}^{*}_{h}) - j_{h}'(\mathbf{u}^{*}_{h})](\mathbf{u}^{*}_{h} - \mathbf{u}^{*}).
\end{align*}
Using the arguments that lead to \eqref{eq:estimate_uh-u_mu} in the proof of Theorem \ref{thm:a_priori_estimate_ocp}, we obtain
\begin{align*}
\|\mathbf{u}^{*} - \mathbf{u}^{*}_{h}\|_{\mathbb{R}^{\ell}}
\lesssim 
\|\by^{*}_{h} - \by_{\mathbf{u}^{*}_{h}}\|_{\Omega} + \|\bp^{*}_{h}  - \bp_{\mathbf{u}^{*}_{h}}\|_{\Omega},
\end{align*}
where $\by_{\mathbf{u}^{*}_{h}}\in \mathbf{H}_0(\cu,\Omega)$ corresponds to the unique solution to problem \eqref{eq:weak_st_eq} with $\mathbf{u}=\mathbf{u}^{*}_{h}$, and $\bp_{\mathbf{u}^{*}_{h}}\in \mathbf{H}_0(\cu,\Omega)$ is the unique solution to problem \eqref{eq:adj_eq} with $\mathbf{u}=\mathbf{u}^{*}_{h}$ and $\by=\by_{\mathbf{u}^{*}_{h}}$. Invoke the a posteriori error estimate \eqref{eq:estimate_state_hat_discrete_st} to conclude that 
\begin{equation}\label{eq:estimate_u_uh_apost_I}
\|\mathbf{u}^{*} - \mathbf{u}^{*}_{h}\|_{\mathbb{R}^{\ell}} \lesssim \mathcal{E}_{st,\T_{h}} + \|\bp^{*}_{h}  - \bp_{\mathbf{u}^{*}_{h}}\|_{\Omega}.
\end{equation}
To estimate $\|\bp^{*}_{h}  - \bp_{\mathbf{u}^{*}_{h}}\|_{\Omega}$ we invoke the term $\mathsf{p}\in\mathbf{H}_{0}(\cu,\Omega)$, solution to \eqref{eq:aux_p}, and the a posteriori error estimate \eqref{eq:estimate_state_hat_discrete_adj} to arrive at
\begin{align}\label{eq:estimate_u_uh_apost_II}
\|\bp^{*}_{h}  - \bp_{\mathbf{u}^{*}_{h}}\|_{\Omega}
\leq
\|\bp^{*}_{h}  - \mathsf{p}\|_{\Omega} + \|\mathsf{p}  - \bp_{\mathbf{u}^{*}_{h}}\|_{\Omega}
\lesssim
\mathcal{E}_{adj,\T_{h}} + \|\mathsf{p}  - \bp_{\mathbf{u}^{*}_{h}}\|_{\mathbf{H}(\cu,\Omega)}.
\end{align}
Finally, we note that the term $\mathsf{p}  - \bp_{\mathbf{u}^{*}_{h}}\in \mathbf{H}_0(\cu,\Omega)$ solves 
\begin{align*}
&(\mu^{-1}\cu (\mathsf{p}  - \bp_{\mathbf{u}^{*}_{h}}), \cu \bw)_{\Omega} - \omega^{2}((\bsiep\cdot \mathbf{u}^{*}_{h})(\mathsf{p}  - \bp_{\mathbf{u}^{*}_{h}}),\bw)_{\Omega} \\ 
& \qquad \qquad = (\overline{\by^{*}_{h} - \by_{\mathbf{u}^{*}_{h}}},\bw)_{\Omega} + (\overline{\cu (\by^{*}_{h} - \by_{\mathbf{u}^{*}_{h}}}), \cu \bw)_{\Omega} \quad \forall \bw \in \mathbf{H}_{0}(\cu,\Omega).
\end{align*}
The stability of this problem gives us $\|\mathsf{p}  - \bp_{\mathbf{u}^{*}_{h}}\|_{\mathbf{H}(\cu,\Omega)} \lesssim \|\by^{*}_{h} - \by_{\mathbf{u}^{*}_{h}}\|_{\mathbf{H}(\cu,\Omega)}\lesssim \mathcal{E}_{st,\T_{h}}$, where, to obtain the last inequality, we have used the error estimate \eqref{eq:estimate_state_hat_discrete_st}. Therefore, using  $\|\mathsf{p}  - \bp_{\mathbf{u}^{*}_{h}}\|_{\mathbf{H}(\cu,\Omega)} \lesssim \mathcal{E}_{st,\T_{h}}$ in  \eqref{eq:estimate_u_uh_apost_II} and the obtained estimate in \eqref{eq:estimate_u_uh_apost_I}, we conclude that: 
\begin{align}\label{eq:estimate_u-uh_apost_final}
\|\mathbf{u}^{*} - \mathbf{u}^{*}_{h}\|_{\mathbb{R}^{\ell}} \lesssim \mathcal{E}_{ocp,\T_{h}}.
\end{align}

\underline{Step 2.} ($\|\by^{*} - \by^{*}_{h}\|_{\mathbf{H}(\cu,\Omega)} \lesssim \mathcal{E}_{ocp,\T_{h}}$) Invoke the variable $\by_{\mathbf{u}^{*}_{h}}\in \mathbf{H}_0(\cu,\Omega)$ and the triangle inequality to obtain
\begin{align}\label{eq:estimate_y_yh_apost}
\|\by^{*} - \by^{*}_{h}\|_{\mathbf{H}(\cu,\Omega)} 
\leq
\|\by_{\mathbf{u}^{*}_{h}} - \by^{*}_{h}\|_{\mathbf{H}(\cu,\Omega)} + \|\by^{*} - \by_{\mathbf{u}^{*}_{h}}\|_{\mathbf{H}(\cu,\Omega)}. 
\end{align}
The first term in the right-hand side of \eqref{eq:estimate_y_yh_apost} can be bounded in view of \eqref{eq:estimate_state_hat_discrete_st}, whereas the second term can be bounded in view of the stability estimate $\|\by^{*} - \by_{\mathbf{u}^{*}_{h}}\|_{\mathbf{H}(\cu,\Omega)} \lesssim \|\mathbf{u}^{*} - \mathbf{u}^{*}_{h}\|_{\mathbb{R}^{\ell}}$. These bounds, in combination with \eqref{eq:estimate_u-uh_apost_final}, yield 
\begin{align}\label{eq:estimate_y-yh_apost_final}
\|\by^{*} - \by^{*}_{h}\|_{\mathbf{H}(\cu,\Omega)} \lesssim \mathcal{E}_{ocp,\T_{h}}.
\end{align}

\underline{Step 3.} ($\|\bp^{*} - \bp^{*}_{h}\|_{\mathbf{H}(\cu,\Omega)} \lesssim \mathcal{E}_{ocp,\T_{h}}$) Similarly to the previous step, we use the variable $\mathsf{p}\in \mathbf{H}_0(\cu,\Omega)$, solution to \eqref{eq:aux_p}, and the triangle inequality to arrive at
\begin{equation}\label{eq:estimate_p_ph_apost}
\|\bp^{*} - \bp^{*}_{h}\|_{\mathbf{H}(\cu,\Omega)} 
\leq
\|\bp^{*} - \mathsf{p}\|_{\mathbf{H}(\cu,\Omega)} + \|\mathsf{p} - \bp^{*}_{h}\|_{\mathbf{H}(\cu,\Omega)}. 
\end{equation}
The term $\|\bp^{*} - \mathsf{p}\|_{\mathbf{H}(\cu,\Omega)}$ is controlled in view of \eqref{eq:estimate_state_hat_discrete_adj}. To bound the remaining term in \eqref{eq:estimate_p_ph_apost}, we use the stability estimate $\|\bp^{*} - \mathsf{p}\|_{\mathbf{H}(\cu,\Omega)} \lesssim \|\by^{*} - \by^{*}_{h}\|_{\mathbf{H}(\cu,\Omega)} + \|\mathbf{u}^{*} - \mathbf{u}^{*}_{h}\|_{\mathbb{R}^{\ell}}$. Hence, we have $\|\bp^{*} - \bp^{*}_{h}\|_{\mathbf{H}(\cu,\Omega)} \lesssim  \|\by^{*} - \by^{*}_{h}\|_{\mathbf{H}(\cu,\Omega)} + \|\mathbf{u}^{*} - \mathbf{u}^{*}_{h}\|_{\mathbb{R}^{\ell}} + \mathcal{E}_{adj,\T_{h}}$. We conclude the proof in view of estimates \eqref{eq:estimate_u-uh_apost_final} and \eqref{eq:estimate_y-yh_apost_final}.
\end{proof}

%%%%%%%%%%%%%%%%%%%%%%%%%%%%%%%%%%%%%%%%%%%%%%%%%%%%%%%%%%%%
%%%%%%%%%%%%%%%%%%%%%%%%%%%%%%%%%%%%%%%%%%%%%%%%%%%%%%%%%%%%
%%%%%%%%%%%%%%%%%%%%%%%%%%%%%%%%%%%%%%%%%%%%%%%%%%%%%%%%%%%%
%%%%%%%%%%%%%%%%%%%%%%%%%%%%%%%%%%%%%%%%%%%%%%%%%%%%%%%%%%%%

\subsubsection{Efficiency analysis}\label{eq:sec:eff}

In the forthcoming analysis we derive an upper bound for the a posteriori error estimator $\mathcal{E}_{ocp,\mathscr{T}_{h}}$. To simplify the exposition, in this section we assume that $\mu^{-1}$ and $\bsiep$ are piecewise polynomial on the partition $\mathcal{P}$; see section \ref{sec:partition_fields}. The analysis will be based on standard bubble function arguments. In particular, it requires the introduction of bubble functions for tetrahedra and their corresponding faces (see \cite{MR1885308,MR3059294}). 

\begin{lemma}[bubble function properties]\label{lemma:burbuja}
Let $j\geq 0$. For any $T\in\T_{h}$  and $ S\in \mathscr{S}_T^I$, let $b_{T}$ and $b_{ S}$  be the corresponding interior quadratic and cubic edge bubble function, respectively. Then, for all $q\in \mathbb{P}_j(T)$ and $p\in \mathbb{P}_j(S)$, there hold 
\begin{equation*}
\|q\|_{T}^2 \lesssim \|b_{T}^{1/2} q\|_{T}^2\leq \|q\|_{T}^2, \qquad
\|b_Sp\|_{S}^2\leq \|p\|^2_{S} \lesssim \|b_S^{1/2}p\|_{S}^2.
\end{equation*}
Moreover, for all $p \in \mathbb{P}_j( S)$, there exists an extension of $p \in  \mathbb{P}_j(T)$,   which we denote simply as $p$, such
that the following estimates hold
\[
 h_T\|p\|_{S}^2\lesssim \|b^{1/2}_{S} p\|_{T}^2 \lesssim h_T\|p\|_{S}^2 \qquad
\forall p\in \mathbb{P}_j(S).
\]
\end{lemma}

As a final ingredient, given $T\in \T_{h}$ and $\bv\in \mathbf{L}^{2}(\Omega;\mathbb{C})$ such that $\bv|_{T} \in \mathbf{H}^{1}(T;\mathbb{C})$, we introduce the term
\begin{align*}
\mathrm{osc}(\bv;T) :=& \,\sum_{T'\in \mathcal{N}_T}\left(h_{T'}\|\bv- \boldsymbol\pi_{T}\bv\|_{T'} + h_{T'}\|\textnormal{div}\,\bv- \pi_{T}\textnormal{div}\,\bv\|_{T'}\right) \\
& ~ +  \sum_{S'\in\mathscr{S}_{T}^{I}}h_{T}^{\frac{1}{2}}\|\jump{(\bv - \boldsymbol\pi_{T}\bv)\cdot\boldsymbol{n}}\|_{S'},
\end{align*}
where $\pi_{T}$ denotes the $\rL^{2}(T)$--orthogonal projection operator onto $\mathbb{P}_{0}(T)$, $\boldsymbol{\pi}_{T}$ denotes the $\mathbf{L}^{2}(T)$--orthogonal projection operator onto $[\mathbb{P}_{0}(T)]^{3}$, and $\mathcal{N}_{T}$ is defined in \eqref{def:patch}.

\begin{theorem}[local efficiency of $\mathcal{E}_{st,T}$]\label{thm:eff_st}
Let $\mathbf{u}^{*} \in U_{ad}$ be a local solution to \eqref{eq:weak_min_problem}--\eqref{eq:weak_st_eq}. Let $\mathbf{u}^{*}_{h}$ be a local minimum of the discrete optimal control problem with $\by^{*}_{h}$ and $\bp^{*}_{h}$ being the corresponding state and adjoint state, respectively. Then, for $T\in\T_{h}$, the local error indicator $\mathcal{E}_{st,T}$ satisfies the  bound
\begin{align*}
\mathcal{E}_{st,T}
\lesssim
\|\mathbf{u}^{*} - \mathbf{u}^{*}_{h}\|_{\mathbb{R}^{l}} + \|\by^{*} - \by^{*}_{h}\|_{\mathbf{H}(\cu,\mathcal{N}_{T})} + \mathrm{osc}(\bbf;T),
\end{align*}
where $\mathcal{N}_{T}$ is defined in \eqref{def:patch}. The hidden constant is independent of  continuous and discrete optimal variables, the size of the elements in $\T_{h}$, and $\#\T_{h}$.
\end{theorem}
\begin{proof} Let $T\in\T_{h}$ and $S\in \mathscr{S}_{T}^{I}$. We define the element and interelement residuals
\begin{align*}
\mathcal{R}_{T,1} &:= \di(\boldsymbol{f} +\omega^{2}(\bsiep\cdot\mathbf{u}^{*}_{h})\by^{*}_{h})|^{}_{T}, \quad \mathcal{J}_{S,1}:= \jump{(\bbf + \omega^{2}(\bsiep\cdot\mathbf{u}^{*}_{h})\by^{*}_{h})\cdot\boldsymbol{n}},\\
\mathcal{R}_{T,2} &:= (\bbf-\cu(\mu^{-1}\cu\by^{*}_h)+\omega^2(\bsiep\cdot\mathbf{u}^{*}_{h})\by^{*}_{h})|^{}_{T},\quad \mathcal{J}_{S,2}:= \jump{\mu^{-1}\cu\by^{*}_h\times\boldsymbol{n}}.
\end{align*}
We immediately note that $\mathcal{E}_{T,k}^{2}:=h_{T}^{2}\|\mathcal{R}_{T,k}\|_{T}^{2} + \tfrac{h_{T}}{2}\sum_{S\in\mathscr{S}_{T}^{I}}\|\mathcal{J}_{S,k}\|_{S}^{2}$ with $k\in\{1,2\}$, and $\mathcal{E}_{st,T}^{2}:=\mathcal{E}_{T,1}^{2} + \mathcal{E}_{T,2}^{2}$; cf. section \ref{sec:glob_rel_ana}. We now proceed on the basis of four steps and estimate each term in the definition of the local estimator $\mathcal{E}_{st,T}$ separately.

\underline{Step 1.} (estimation of $h_{T}\|\mathcal{R}_{T,2}\|_{T}$) Let $T\in \T_{h}$. We define the term $\tilde{\mathcal{R}}_{T,2}:= (\pi_{T}\bbf-\cu(\mu^{-1}\cu\by^{*}_h)+\omega^{2}(\bsiep\cdot\mathbf{u}^{*}_{h})\by^{*}_{h})|^{}_T$. The triangle inequality yields
\begin{align}\label{eq:triangle_RT1}
h_{T}\|\mathcal{R}_{T,2}\|_{T} \leq h_{T}\|\bbf - \boldsymbol\pi_{T}\bbf\|_{T} + h_{T}\|\tilde{\mathcal{R}}_{T,2}\|_{T}.
\end{align}
Now, a simple computation reveals, in view of \eqref{eq:weak_st_eq}, that
\begin{align}\label{eq:identity_R1_J1}
&(\mu^{-1}\cu(\by^{*} - \by^{*}_{h}),\cu \bw)_{\Omega} - \omega^{2}((\bsiep\cdot\mathbf{u}^{*})(\by^{*} - \by^{*}_{h}),\bw)_{\Omega} \\
= \sum_{T \in \T}&(\tilde{\mathcal{R}}_{T,2},\bw)_{T} - \sum_{S\in \mathcal{S}}(\mathcal{J}_{S,2},\bw)_{S} + (\bbf - \boldsymbol\pi_{T}\bbf,\bw)_{\Omega}  - \omega^{2}((\bsiep\cdot[\mathbf{u}_{h}^{*} - \mathbf{u}^{*}]\by^{*}_{h},\bw)_{\Omega} \nonumber
\end{align}
for all $\bw\in \mathbf{H}_{0}(\cu,\Omega)$. We now invoke the bubble function $b_{T}$, introduced in Lemma \ref{lemma:burbuja}, set $\bw = b_{T}\tilde{\mathcal{R}}_{T,2}\in \mathbf{H}_0^{1}(T)$ in \eqref{eq:identity_R1_J1}, and use basic inequalities to obtain
\begin{align*}
\|\tilde{\mathcal{R}}_{T,2}\|_{T}^{2}
\lesssim &
\|\bbf - \boldsymbol\pi_{T}\bbf\|_{T}\|\tilde{\mathcal{R}}_{T,2}\|_{T} + \|\mathbf{u}^{*} - \mathbf{u}^{*}_{h}\|_{\mathbb{R}^{\ell}}\|\by^{*}_{h}\|_{T}\|\tilde{\mathcal{R}}_{T,2}\|_{T}\\
& ~ + \|\mathbf{u}^{*}\|_{\mathbb{R}^{\ell}}\|\by^{*} - \by^{*}_{h}\|_{T}\|\tilde{\mathcal{R}}_{T,2}\|_{T} + \|\cu (\by^{*} - \by^{*}_{h})\|_{T}\|\cu (b_{T}\tilde{\mathcal{R}}_{T,2})\|_{T},
\end{align*}
upon using the properties of $b_{T}$ provided in Lemma \ref{lemma:burbuja}. Hence, a standard inverse estimate and the bounds $\|\by^{*}_{h}\|_{T}\leq \|\by^{*}_{h}\|_{\Omega} \lesssim \|\bbf\|_{\Omega}$ and $ \|\mathbf{u}^{*}\|_{\mathbb{R}^{\ell}} \leq \|\mathbf{b}\|_{\mathbb{R}^{\ell}}$ yield 
\begin{equation*}
h_{T}\|\tilde{\mathcal{R}}_{T,2}\|_{T}
\lesssim
h_{T}\|\bbf - \boldsymbol\pi_{T}\bbf\|_{T} + h_{T}\|\mathbf{u}^{*} - \mathbf{u}^{*}_{h}\|_{\mathbb{R}^{\ell}} + h_{T}\|\by^{*} - \by^{*}_{h}\|_{T} + \|\cu (\by^{*} - \by^{*}_{h})\|_{T},
\end{equation*}
which, in view of \eqref{eq:triangle_RT1}, allows us to conclude that 
\begin{align}\label{eq:final_RT1}
h_{T}\|\mathcal{R}_{T,2}\|_{T}
\lesssim & ~
h_{T}\|\bbf\! -\! \boldsymbol\pi_{T}\bbf\|_{T} + h_{T}\|\mathbf{u}^{*}\! -\! \mathbf{u}^{*}_{h}\|_{\mathbb{R}^{\ell}} 
  + h_{T}\|\by^{*}\! -\! \by^{*}_{h}\|_{T} + \|\cu (\by^{*} - \by^{*}_{h})\|_{T}. \nonumber
\end{align}

\underline{Step 2.} (estimation of $h_{T}^{\frac{1}{2}}\|\mathcal{J}_{S,2}\|_{S}$) Let $T\in \T_{h}$ and $S\in \mathscr{S}_{T}^{I}$. Invoke the bubble function $b_{S}$ from Lemma \ref{lemma:burbuja}, use $\bw = b_{S}\mathcal{J}_{S,2}$ in \eqref{eq:identity_R1_J1}, and a standard inverse estimate in combination with the properties of $b_{S}$ to arrive at 
\begin{align*}
\|\mathcal{J}_{S,2}\|_{S}^{2}
\lesssim &
\sum_{T'\in \mathcal{N}_{S}}\left(\|\mathcal{R}_{T,2}\|_{T'} + \|\mathbf{u}^{*} - \mathbf{u}^{*}_{h}\|_{\mathbb{R}^{\ell}}\|\by^{*}_{h}\|_{T'} \right.\\
&  \left.  +\, h_{T'}^{-1}\|\cu(\by^{*} - \by^{*}_{h})\|_{T'}
+  \|\mathbf{u}^{*}\|_{\mathbb{R}^{\ell}}\|\by^{*} - \by^{*}_{h}\|_{T'}\right)h_{T}^{\frac{1}{2}}\|\mathcal{J}_{S,1}\|_{S}.
\end{align*}
We thus conclude, in light of $\|\by^{*}_{h}\|_{T'}\lesssim \|\bbf\|_{\Omega}$ and estimate \eqref{eq:final_RT1}, the estimation
\begin{align*}
\|\mathcal{J}_{S,2}\|_{S}
\lesssim & ~
h_{T}\|\mathbf{u}^{*} - \mathbf{u}^{*}_{h}\|_{\mathbb{R}^{\ell}} \\
&~ + \sum_{T'\in \mathcal{N}_{S}}\!\left(h_{T}\|\bbf \!-\! \boldsymbol\pi_{T}\bbf\|_{T'} + h_{T}\|\by^{*} \!-\! \by^{*}_{h}\|_{T'} + \|\cu (\by^{*} \!-\! \by^{*}_{h})\|_{T'}\right).
\end{align*}

\underline{Step 3.} (estimation of $h_{T}\|\mathcal{R}_{T,1}\|_{T}$) Let $T\in \T_{h}$. We define the term $\tilde{\mathcal{R}}_{T,1}:= (\pi_{T}\text{div}\,\bbf-\text{div}(\omega^{2}(\bsiep\cdot\mathbf{u}^{*}_{h})\by^{*}_{h}))|^{}_T$. The triangle inequality thus yields
\begin{equation}\label{eq:triangle_RT2}
h_{T}\|\mathcal{R}_{T,1}\|_{T} \leq h_{T}\|\text{div}\,\bbf - \pi_{T}\text{div}\,\bbf\|_{T} + h_{T}\|\tilde{\mathcal{R}}_{T,1}\|_{T}.
\end{equation}
On the other hand, in light of \eqref{eq:weak_st_eq}, we have
\begin{align}\label{eq:identity_R2}
&(\mu^{-1}\cu(\by^{*} - \by^{*}_{h}),\cu \bw)_{\Omega} - \omega^{2}((\bsiep\cdot\mathbf{u}^{*})(\by^{*} - \by^{*}_{h}),\bw)_{\Omega} \\
= \!\! \sum_{T \in \T} \!\!  \big( (\bbf + & \, \omega^{2}(\bsiep\cdot \mathbf{u}^{*}_{h})\by^{*}_{h},\bw)_{T}   - \! (\mu^{-1}\cu \by^{*}_{h},\cu \bw)_{T} - \omega^{2}((\bsiep\cdot [\mathbf{u}_{h}^{*}\! - \! \mathbf{u}^{*}])\by^{*}_{h},\bw)_{T}\big) \nonumber
\end{align}
for all $\bw\in \mathbf{H}_{0}(\cu,\Omega)$. 
We then set $\bw = \nabla(b_{T}\tilde{\mathcal{R}}_{T,1})$ in the latter identity, and apply an integration by parts formula to obtain
\begin{align*}
& \omega^{2} ((\bsiep\cdot\mathbf{u}^{*})(\by^{*} - \by^{*}_{h}),\nabla(b_{T}\tilde{\mathcal{R}}_{T,1}))_{T} - \omega^{2}((\bsiep\cdot [\mathbf{u}_{h}^{*} - \mathbf{u}^{*}])\by^{*}_{h}, \nabla(b_{T}\tilde{\mathcal{R}}_{T,1}))_{T} \\
& =  \|b_{T}^{1/2}\tilde{\mathcal{R}}_{T,1}\|_{T}^{2} + (\text{div}\,\bbf - \pi_{T}\text{div}\,\bbf,b_{T}\tilde{\mathcal{R}}_{T,1})_{T}.
\end{align*}
Therefore, utilizing standard inverse estimates in combination with the properties of $b_{T}$ we obtain $
h_{T}\|\tilde{\mathcal{R}}_{T,1}\|_{T}
\lesssim
\|\by^{*} - \by^{*}_{h}\|_{T} + \|\mathbf{u}^{*} - \mathbf{u}^{*}_{h}\|_{\mathbb{R}^{\ell}} + h_{T}\|\text{div}\,\bbf - \pi_{T}\text{div}\,\bbf\|_{T},$
which, in view of \eqref{eq:triangle_RT2}, implies that
\begin{align}\label{eq:final_RT2}
h_{T}\|\mathcal{R}_{T,1}\|_{T}
\lesssim
\|\by^{*} - \by^{*}_{h}\|_{T} + \|\mathbf{u}^{*} - \mathbf{u}^{*}_{h}\|_{\mathbb{R}^{\ell}} + h_{T}\|\text{div}\,\bbf - \pi_{T}\text{div}\,\bbf\|_{T}.
\end{align}

\underline{Step 4.} (estimation of $h_{T}^{\frac{1}{2}}\|\mathcal{J}_{S,1}\|_{S}$) Let $T\in \T_{h}$ and $S\in \mathscr{S}_{T}^{I}$. Define $\tilde{\mathcal{J}}_{S,1} :=  \jump{(\boldsymbol\pi_{T}\bbf +\omega^{2} (\bsiep\cdot\mathbf{u}^{*}_{h})\by^{*}_{h})\cdot\boldsymbol{n}}$. An application of the triangle inequality results in
\begin{equation}\label{eq:triangle_JS2}
h_{T}^{\frac{1}{2}}\|\mathcal{J}_{S,1}\|_{S} \leq h_{T}^{\frac{1}{2}}\|\jump{(\bbf - \boldsymbol\pi_{T}\bbf)\cdot\boldsymbol{n}}\|_{S} + h_{T}^{\frac{1}{2}}\|\tilde{\mathcal{J}}_{S,1}\|_{S}.
\end{equation}
Invoke the bubble function $b_{S}$ from Lemma \ref{lemma:burbuja}, use $\bw = \nabla(b_{S}\tilde{\mathcal{J}}_{S,1})$ in \eqref{eq:identity_R2}, and apply an integration by parts formula. These arguments yield the identity
\begin{align*}
&\sum_{T'\in\mathcal{N}_{S}}\left(-\omega^{2} ((\bsiep\cdot\mathbf{u}^{*})(\by^{*} - \by^{*}_{h}),\nabla(b_{T}\mathcal{J}_{S,1}))_{T'} + \omega^{2}((\bsiep\cdot [\mathbf{u}_{h}^{*} - \mathbf{u}^{*}])\by^{*}_{h}, \nabla(b_{S}\mathcal{J}_{T,1})_{T'}\right) \\
& \quad = \|b_{S}^{1/2}\tilde{\mathcal{J}}_{S,1}\|_{S}^{2} + (\jump{(\bbf - \boldsymbol\pi_{T}\bbf)\cdot\boldsymbol{n}},b_{S}\tilde{\mathcal{J}}_{S,1})_{S} - \sum_{T'\in\mathcal{N}_{S}}(\mathcal{R}_{T,1},b_{S}\tilde{\mathcal{J}}_{S,1})_{T'}.
\end{align*}
We thus utilize inverse estimates in combination with the properties of $b_{S}$ to obtain
\begin{equation*}
h_{T}^{\frac{1}{2}}\|\tilde{\mathcal{J}}_{S,1}\|_{S}
\lesssim
\|\mathbf{u}^{*} - \mathbf{u}^{*}_{h}\|_{\mathbb{R}^{\ell}} + \sum_{T'\in\mathcal{N}_{S}}(\|\by^{*} - \by^{*}_{h}\|_{T'} + h_{T}\|\mathcal{R}_{T,1}\|_{T'} + h_{T}^{\frac{1}{2}}\|\jump{(\bbf - \boldsymbol\pi_{T}\bbf)\cdot\boldsymbol{n}}\|_{S}).
\end{equation*}
The combination of the latter estimate and estimates \eqref{eq:triangle_JS2} and \eqref{eq:final_RT2} results in
\begin{align*}
h_{T}^{\frac{1}{2}}\|\mathcal{J}_{S,1}\|_{S}
\lesssim & \,
\|\mathbf{u}^{*} - \mathbf{u}^{*}_{h}\|_{\mathbb{R}^{\ell}} + \sum_{T'\in\mathcal{N}_{S}}(\|\by^{*} - \by^{*}_{h}\|_{T'} \\
& ~+ h_{T}\|\text{div}\,\bbf - \pi_{T'}\text{div}\,\bbf\|_{T'}+ h_{T}^{\frac{1}{2}}\|\jump{(\bbf - \boldsymbol\pi_{T}\bbf)\cdot\boldsymbol{n}}\|_{S}).
\end{align*}

We end the proof in view of the estimates obtained in the four previous steps.
\end{proof}

\begin{theorem}[local efficiency of $\mathcal{E}_{adj,T}$]\label{thm:eff_adj}
Let $\mathbf{u}^{*} \in U_{ad}$ be a local solution to \eqref{eq:weak_min_problem}--\eqref{eq:weak_st_eq}. Let $\mathbf{u}^{*}_{h}$ be a local minimum of the discrete optimal control problem with $\by^{*}_{h}$ and $\bp^{*}_{h}$ being the corresponding state and adjoint state, respectively. Then, for $T\in\T_{h}$, the local error indicator $\mathcal{E}_{adj,T}$ satisfies the  bound
\begin{align*}
\mathcal{E}_{adj,T}
&\lesssim 
\, \|\mathbf{u}^{*} - \mathbf{u}^{*}_{h}\|_{\mathbb{R}^{l}} + \|\by^{*} - \by^{*}_{h}\|_{\mathbf{H}(\cu,\mathcal{N}_{T})} + \|\bp^{*} - \bp^{*}_{h}\|_{\mathbf{H}(\cu,\mathcal{N}_{T})} + \mathrm{osc}(\by_{\Omega};T) \\
&  + \sum_{T'\in \mathcal{N}_T}h_{T'}\|\cu \bE_{\Omega}- \boldsymbol\pi_{T}\cu\bE_{\Omega}\|_{T'} +  \sum_{S'\in\mathscr{S}_{T}^{I}}h_{T}^{\frac{1}{2}}\|\jump{(\bE_{\Omega} - \boldsymbol\pi_{T}\bE_{\Omega})\times\boldsymbol{n}}\|_{S'},  \nonumber
 \end{align*}
where $\mathcal{N}_{T}$ is defined in \eqref{def:patch}. The hidden constant is independent of  continuous and discrete optimal variables, the size of the elements in $\T_{h}$, and $\#\T_{h}$.
\end{theorem}
\begin{proof}
The proof follows analogous arguments to the ones provided in the proof of Theorem \ref{thm:eff_st}. For brevity, we skip details.
\end{proof}

We conclude this section with the following result, which is a direct consequence of Theorems \ref{thm:eff_st} and \ref{thm:eff_adj}.

\begin{corollary}[efficiency of $\mathcal{E}_{ocp,T}$]
In the framework of Theorems \ref{thm:eff_st} and \ref{thm:eff_adj} we have, for $T\in\T_{h}$, that the local error indicator $\mathcal{E}_{ocp,T}$ satisfies the bound
\begin{multline*}
\mathcal{E}_{ocp,T}
\lesssim
\|\mathbf{u}^{*} - \mathbf{u}^{*}_{h}\|_{\mathbb{R}^{l}} + \|\by^{*} - \by^{*}_{h}\|_{\mathbf{H}(\cu,\mathcal{N}_{T})} + \|\bp^{*} - \bp^{*}_{h}\|_{\mathbf{H}(\cu,\mathcal{N}_{T})} + \mathrm{osc}(\bbf;T) \\
 \! + \mathrm{osc}(\by_{\Omega};T) + \!\!\!\!  \sum_{T'\in \mathcal{N}_T} \!\! h_{T'}\|\cu \bE_{\Omega} - \boldsymbol\pi_{T}\cu\bE_{\Omega}\|_{T'} + \!\!\!\! \sum_{S'\in\mathscr{S}_{T}^{I}} \!\!  h_{T}^{\frac{1}{2}}\|\jump{(\bE_{\Omega} - \boldsymbol\pi_{T}\bE_{\Omega}) \! \times \! \boldsymbol{n}}\|_{S'},
\end{multline*}
where $\mathcal{N}_{T}$ is defined in \eqref{def:patch}. The hidden constant is independent of  continuous and discrete optimal variables, the size of the elements in $\T_{h}$, and $\#\T_{h}$.
\end{corollary}

%
%%%%%%%%%%%%%%%%%%%%%%%%%%%%%%%%%%%%%%%%%%%%%%%%%%%
%%%%%%%%%%%%%%%%%%%%%%%%%%%%%%%%%%%%%%%%%%%%%%%%%%%
%%%%%%%%%%%%%%%%%%%%%%%%%%%%%%%%%%%%%%%%%%%%%%%%%%%
%%%%%%%%%%%%%%%%%%%%%%%%%%%%%%%%%%%%%%%%%%%%%%%%%%%
%%%%%%%%%%%%%%%%%%%%%%%%%%%%%%%%%%%%%%%%%%%%%%%%%%%
%%%%%%%%%%%%%%%%%%%%%%%%%%%%%%%%%%%%%%%%%%%%%%%%%%%
%%%%%%%%%%%%%%%%%%%%%%%%%%%%%%%%%%%%%%%%%%%%%%%%%%%
%%%%%%%%%%%%%%%%%%%%%%%%%%%%%%%%%%%%%%%%%%%%%%%%%%%

\section{Numerical experiments}
\label{sec:num_ex}

In this section, we present three numerical tests in order to validate our theoretical findings and assess the performance of the proposed a posteriori error estimator $\mathcal{E}_{ocp,\T_{h}}$, defined in \eqref{def:error_estimator_ocp}. 
These experiments have been carried out with the help of a code that we implemented in a FEniCS script \cite{fenics_book} by using lowest-order N\'ed\'elec elements.  

In the following numerical examples, we shall restrict to the case where all the functions and variables present in the optimal control problem are real-valued. 
This, with the aim of simplifying numerical computations, acknowledging that the inclusion of complex variables would significantly increase computational costs. 
In particular, and following Remark~\ref{rmk:real2}, we consider the following problem: $\min \calJ(\by,\mathbf{u})$
subject to 
\[
\cu \chi\cu \by  + (\kappa\cdot \mathbf{u})\by= \bbf \quad \mbox{in } \Omega, \qquad
\by\times \bn= \mathbf{0} \quad \mbox{on } \Gamma,
\]
and the control constraints $\mathbf{u}=(\mathbf{u}_{1},\ldots,\mathbf{u}_{\ell})\in U_{ad}$ and $U_{ad}:=\left\{\mathbf{v} \in \mathbb{R}^{\ell}: \mathbf{a}\leq \mathbf{v}\leq \mathbf{b}\right\}$. We recall that real-valued coefficients $\kappa,\chi \in P\rW^{1,\infty}(\Omega)$ satisfy  $\kappa \geq \kappa_0 >0 $ 
and $\chi \geq \chi_0 >0 $ with $\kappa_0, \mu_0\in \mathbb{R}^{+}$ and that $\kappa\cdot \mathbf{u} = \sum_{k=1}^{\ell}\kappa|_{\Omega_{k}}^{}\mathbf{u}_{k}$.

\subsection{Implementation issues}
In this section we briefly discuss implementation details of the discretization strategy proposed in section \ref{sec:fem_for_ocp}.

For a given mesh $\mathscr{T}_{h}$, we seek $(\by^{*}_{h},\bp^{*}_{h},\mathbf{u}^{*}_{h}) \in \mathbf{V}(\T_{h})\times \mathbf{V}(\T_{h}) \times U_{ad}$ that solves 
\begin{align*}
\begin{cases}
(\mu^{-1}\cu \by_{h}^{*}, \cu \bv_{h})_{\Omega} +((\kappa\cdot \mathbf{u}^{*}_{h}) \by^{*}_{h},\bv_{h})_{\Omega} = & \hspace{-0.4cm} (\bbf,\bv_{h})_{\Omega}, \\
(\mu^{-1}\cu \bp^{*}_{h}, \cu \bw_{h})_{\Omega} +((\kappa\cdot \mathbf{u}^{*}_{h})\bp^{*}_{h},\bw_{h})_{\Omega} = & \hspace{-0.3cm} (\by^{*}_{h} - \by_{\Omega},\bw_{h})_{\Omega} \\
& \hspace{-0.2cm} + \, (\cu \by^{*}_{h} - \bE_{\Omega}, \cu \bw_{h})_{\Omega},\\ 
\sum_{k=1}^{\ell}\left(\alpha(\mathbf{u}^{*}_{h})_{k} - \int_{\Omega_{k}}\kappa\by^{*}_{h}\cdot\bp^{*}_{h}\right)(\mathbf{u}_{k} - (\mathbf{u}^{*}_{h})_{k}) \geq & \hspace{-0.3cm} 0,
\end{cases}
\end{align*}
for all $(\bv_{h},\bw_{h},\mathbf{u}_{h}) \in \mathbf{V}(\T_{h})\times \mathbf{V}(\T_{h}) \times U_{ad}$. 
This \emph{discrete optimality system} is solved by using a semi-smooth Newton method. 
To present the latter, we define $\mathbf{X}(\T_{h}) := \mathbf{V}(\T_{h})\times \mathbf{V}(\T_{h}) \times \mathbb{R}^{\ell}$ and introduce, for $\boldsymbol\eta = (\by_{h},\bp_{h},\mathbf{u}_{h})$ and $\Theta = (\bv_{h},\bw_{h},\mathbf{u}_{h}) $ in $\mathbf{X}(\T_{h})$, the operator $F_{\T_{h}}: \mathbf{X}(\T_{h}) \rightarrow \mathbf{X}(\T_{h})'$, whose dual action on $\Theta$, i.e. $\langle F_{\T_{h}}(\Psi),\Theta \rangle^{}_{ \mathbf{X}(\T_{h})', \mathbf{X}(\T_{h})}$, is defined by
\begin{align*}
\begin{pmatrix} 
(\mu^{-1}\cu \by_{h}, \cu \bv_{h})_{\Omega} +((\kappa\cdot \mathbf{u}_{h}) \by_{h} - \bbf,\bv_{h})_{\Omega}  \\
(\mu^{-1}\cu \bp_{h} - \cu \by_{h} + \bE_{\Omega}, \cu \bw_{h})_{\Omega} +((\kappa\cdot \mathbf{u}_{h})\bp^{*}_{h} - \by_{h} + \by_{\Omega},\bw_{h})_{\Omega}  \\
(\mathbf{u}_{h})_{1} - \mathbf{c}_{1} - \max\{\mathbf{a}_{1} - \mathbf{c}_{1}, 0\} + \max\{\mathbf{c}_{1} - \mathbf{b}_{1}, 0\}\\
\vdots \\
(\mathbf{u}_{h})_{\ell} - \mathbf{c}_{\ell} - \max\{\mathbf{a}_{\ell} - \mathbf{c}_{\ell}, 0\} + \max\{\mathbf{c}_{\ell} - \mathbf{b}_{\ell}, 0\}
\end{pmatrix},
\end{align*}
where $\mathbf{c}_{k}:= -\alpha^{-1}\int_{\Omega_{k}}\kappa\by_{h}\cdot\bp_{h}$ with $k\in\{1,\ldots,\ell\}$. 
Given an initial guess $\boldsymbol\eta_{0} = (\by_{h}^{0},\bp_{h}^{0},\mathbf{u}_{h}^{0}) \in \mathbf{X}(\T_{h})$ and $j \in \mathbb{N}_{0}$, we consider the following Newton iteration $
 \boldsymbol\eta_{j+1} = \boldsymbol\eta_{j} + \delta\boldsymbol\eta$, where the incremental term $ \delta\boldsymbol\eta \! = \! (\delta \by_{h},\delta \bp_{h},\delta \mathbf{u}_{h})\in\mathbf{X}(\T_{h})$ solves 
\begin{align}\label{eq:Newton_method_iter}
\langle F'_{\T_{h}}(\boldsymbol\eta_{j})(\delta\boldsymbol\eta), \Theta \rangle^{}_{ \mathbf{X}(\T_{h})', \mathbf{X}(\T_{h})} = -\langle F_{\T_{h}}(\boldsymbol\eta_{j}),\Theta \rangle^{}_{ \mathbf{X}(\T_{h})', \mathbf{X}(\T_{h})}
\end{align}
for all $\Theta = (\bv_{h},\bw_{h},\mathbf{u}_{h})\in \mathbf{X}(\T_{h})$.
Here, $F'_{\T_{h}}(\boldsymbol\eta_{j})(\delta\boldsymbol\eta)$ denotes the G\^ateaux derivate of $F_{\T_{h}}$ at $\boldsymbol\eta_{j} = (\by_{h}^{j},\bp_{h}^{j},\mathbf{u}_{h}^{j})$ in the direction $\delta\boldsymbol\eta$. 
We immediately notice that, in the semi-smooth Newton method, we apply the following derivative to $\max\{\cdot,0\}$:
\begin{align*}
\max\{c,0\}'
=
1 ~ \text{ if } c \geq 0, 
\qquad
\max\{c,0\}'
=
0 ~ \text{ if } c < 0.
\end{align*}

To apply the adaptive finite element method, we generate a sequence of nested conforming triangulations using the adaptive procedure described in \textbf{Algorithm} \ref{Algorithm}.
\begin{algorithm}[!h]
\caption{\textbf{ Adaptive Algorithm.}}
\label{Algorithm}
\small
\textbf{Input:} Initial mesh $\mathscr{T}_{0}$, data $\bbf$, desired states $\by_{\Omega}$ and $\bE_{\Omega}$, functions $\boldsymbol\chi$ and $\kappa$, vector constraints $\mathbf{a}$ and $\mathbf{b}$, and control cost $\alpha$.
\\
\textbf{Set:} $n=0$.
\\
\textbf{Active set strategy:}
\\
$\boldsymbol{1}:$ Choose initial discrete guess $\boldsymbol\eta_{0} = (\by_{n}^{0},\bp_{n}^{0},\mathbf{u}_{n}^{0}) \in \mathbf{X}(\T_{n})$. 
\\
$\boldsymbol{2}:$ Compute $[\by^{*}_{n},\bp^{*}_{n},\mathbf{u}^{*}_{n}]=\mathbf{SSNM}[\mathscr{T}_n,\boldsymbol\eta_{0}, \bbf, \by_{\Omega}, \bE_{\Omega}, \boldsymbol\chi, \kappa, \mathbf{a}, \mathbf{b}, \alpha]$, where \textbf{SSNM} implements Newton iteration \eqref{eq:Newton_method_iter}.
\\
\textbf{Adaptive loop:}
\\
$\boldsymbol{3}:$ For each $T\in\mathscr{T}_n$ compute the local indicators $\mathcal{E}_{st,T}$ and $\mathcal{E}_{adj,T}$ defined in section \ref{sec:glob_rel_ana}.
\\
$\boldsymbol{4}:$ Mark an element $T$ for refinement if $\zeta_{T}\geq 0.5\max_{T'\in\T_h}\zeta_{T'}$, with $\zeta_{T}\in\{\mathcal{E}_{st,T}, \mathcal{E}_{adj,T}\}$.
\\
$\boldsymbol{5}:$ From step $\boldsymbol{4}$, construct a new mesh, using a longest edge bisection algorithm. Set $n \leftarrow n + 1$ and go to step $\boldsymbol{1}$.
\end{algorithm}
 
%%%%%%%%%%%%%%%%%%%%%%%%%%%%%%%%%%%%%%%%%%%%%%%%%%%
%%%%%%%%%%%%%%%%%%%%%%%%%%%%%%%%%%%%%%%%%%%%%%%%%%%
%%%%%%%%%%%%%%%%%%%%%%%%%%%%%%%%%%%%%%%%%%%%%%%%%%%
%%%%%%%%%%%%%%%%%%%%%%%%%%%%%%%%%%%%%%%%%%%%%%%%%%%

\subsection{Test 1. Smooth solutions}
We consider this example to verify that the expected order of convergence 
is obtained when solutions of the control problem are smooth. 
In this context, we assume $\Omega:=(0,1)^3$, $\mathbf{a}=0.01$, $\mathbf{b}=5$, $\alpha=0.1$, $\chi=1$, and $\kappa=0.1$;
the source term $\bbf$, the desired states $\by_\Omega$ and $\bE_\Omega$, and the boundary conditions are chosen such that the exact optimal state and adjoint state are given by
\begin{align*}
\by^{*}(\boldsymbol{x}) &= (\cos(\pi x)\sin(\pi y)\sin(\pi z), 
          \sin(\pi x)\cos(\pi y)\sin(\pi z),
          \sin(\pi x)\sin(\pi y)\cos(\pi z)),\\
\bp^{*}(\boldsymbol{x}) &= -(x^2\sin(\pi y)\sin(\pi z),
             \sin(\pi x)\sin(\pi z),
             \sin(\pi x)\sin(\pi y)),
\end{align*}
%\begin{align*}
%\by^{*} =\begin{pmatrix}
%    \cos(\pi x)\sin(\pi y)\sin(\pi z) \\
%    \sin(\pi x)\cos(\pi y)\sin(\pi z) \\
%    \sin(\pi x)\sin(\pi y)\cos(\pi z)
%\end{pmatrix},
%\qquad
%\bp^{*} = -\begin{pmatrix}
%        x^2\sin(\pi y)\sin(\pi z)\\
%         \sin(\pi x)\sin(\pi z)\\
%         \sin(\pi x)\sin(\pi y)
%\end{pmatrix}.
%\end{align*}
where $\boldsymbol{x} = (x,y,z)$. 
Given the smoothness of the solution, we present the obtained errors and their experimental rates of convergence only with uniform refinement.
In particular, Table~\ref{tabla:error_square} 
shows the convergence history for  $ \|\by^{*} - \by^{*}_{h}\|_{\mathbf{H}(\cu,\Omega)}$ and  $ \|\bp^{*} - \bp^{*}_{h}\|_{\mathbf{H}(\cu,\Omega)}$. 
In the same table, the corresponding experimental convergence rates are shown in terms of the mesh size $h$.
We observe that the optimal rate of convergence is attained for 
both variables (cf.~ Theorem \ref{thm:extra_reg_Maxwell}\textrm{(ii)} and Corollary \ref{coro:error_estimate}).

\begin{table}[!ht]
	{\footnotesize%\setlength{\tabcolsep}{4.7pt}
		\caption{Test 1: $\mathbf{H}(\cu,\Omega)$-error and experimental order of convergence for the approximations of $\by^{*}$ and $\bp^{*}$ with uniform refinement.}
		\label{tabla:error_square}
		\begin{center}
			\begin{tabular}{ l c l c l }
 \toprule
%				 \multicolumn{5}{|c|}{Uniform}                                                                                                                  
 			 \multicolumn{1}{l}{$h$} & $ \|\by^{*} - \by^{*}_{h}\|_{\mathbf{H}(\cu,\Omega)}$& 
Order & $ \|\bp^{*} - \bp^{*}_{h}\|_{\mathbf{H}(\cu,\Omega)}$    & Order \\
\midrule
				0.8660& 0.98925 & --    & 1.70729 & --         \\
				0.4330& 0.38458 & 0.825 & 0.96359 & 1.363 \\
				0.2165& 0.16768 & 0.961 & 0.49503 & 1.197 \\
				0.1082& 0.08271 & 0.986 & 0.24997 & 1.019  \\
				0.0541& 0.04609 & 0.972 & 0.12747 & 0.843 \\			
\bottomrule
			\end{tabular}
	\end{center}}
\end{table}

\subsection{Test 2. A 3D L-shaped domain} This test aims to assess the performance of the numerical scheme when solving the optimal control problem for a solution with a line singularity, with uniform and adaptive refinement. To this end, we consider the classical three-dimensional L-shape domain given by
\begin{align*}
\Omega:=(-1,1)\times(-1,1)\times(0,1)\backslash\bigg((0,1)\times(-1,0)\times(0,1) \bigg).
\end{align*}
An example of the initial mesh used for this example is depicted in Figure~\ref{fig:Lshape-mesh-estimador-theta} (left).
Let $\bbf$, $\by_\Omega$, and $\bE_\Omega$ be such that the exact solution of the optimal control problem with $\mathbf{a}=0.01$, $\mathbf{b}=1$, $\alpha=1$, $\chi=1$, $\kappa=0.01$ is $
\by^{*}=\bp^{*}=(\tfrac{\partial S }{\partial x},\frac{\partial S }{\partial y},0)$, where function $S$ is given, in terms of the polar coordinates $(r,\theta)$, by $S(r,\theta)=r^{2/3}\sin(2\theta/3)$.
Notice that  $(\by^{*},\bp^{*})$ have a line singularity located at $z-$axis, and the solution belongs only to $\mathbf{H}^{2/3-\epsilon}(\cu,\Omega)$ for any $\epsilon > 0$ (see, for instance, \cite{MR2051073}). According to \eqref{eq:total_error_apriori} the expected convergence rate should be $\mathcal{O}(h^{2/3-\epsilon})$ for any $\epsilon > 0$.

In Figure~\ref{fig:initial_curves} (right) we present experimental rates of convergence for $\|\by^{*}-\by^{*}_h\|_{\mathbf{H}(\cu,\Omega)}$, with uniform and adaptive refinement, in terms of the number of elements $N$ of the meshes. 
We observe that $\by_{h}^{*}$ converges to $\by^{*}$ 
with order $\mathcal{O}(N^{-0.2}) \approx \mathcal{O}(h^{0.6})$ for the uniform case, which is close to the 
expected order of convergence. 
On the other hand, the convergence for the adaptive scheme is  
$\mathcal{O}(N^{-0.3}) \approx \mathcal{O}(h^{0.9})$. 
We note that the adaptive scheme is able to recover the optimal order $\mathcal{O}(N^{-1/3}) \approx \mathcal{O}(h)$. 
In the same figure, we also present $\mathcal{E}_{ocp,\T_{h}}$ for each adaptive iteration. It notes that the estimator decays asymptotically as $\mathcal{O}(N^{-0.29})$.
We observe that the convergences of the a posteriori error estimator and the energy error are almost optimal. 
Due to the similarity in observed behavior between the approximation of $\bp^{*}$ and the previous results, 
both in terms of error and estimator performance, we have omitted its analysis for brevity.
Finally, in Figure~\ref{fig:Lshape-mesh-estimador-theta} (right) we observe a comparison between meshes in different adaptive iterations. It can be seen that the adaptive algorithm refine around the singularity produced by the re-entrant corner.

\begin{figure}[!h]
	\centering
	\begin{minipage}{0.41\linewidth}\centering
		\includegraphics[trim={0 0 0 0},clip,width=3.56cm,height=4.1cm,scale=0.8]{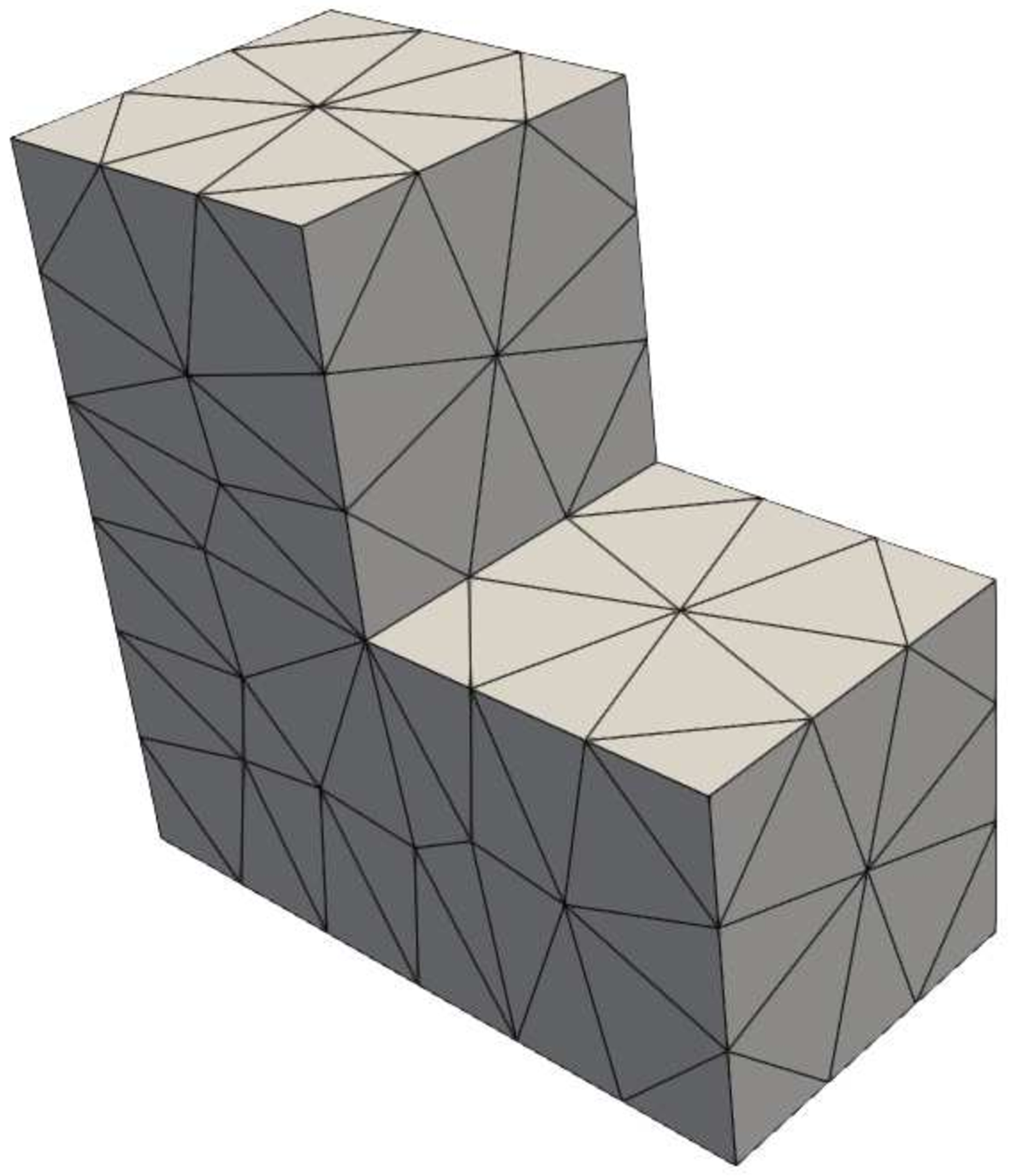}
	\end{minipage}
	\begin{minipage}{0.58\linewidth}\centering
		\includegraphics[trim={0 0 0 0},clip,width=5.5cm,height=4.1cm,scale=0.8]{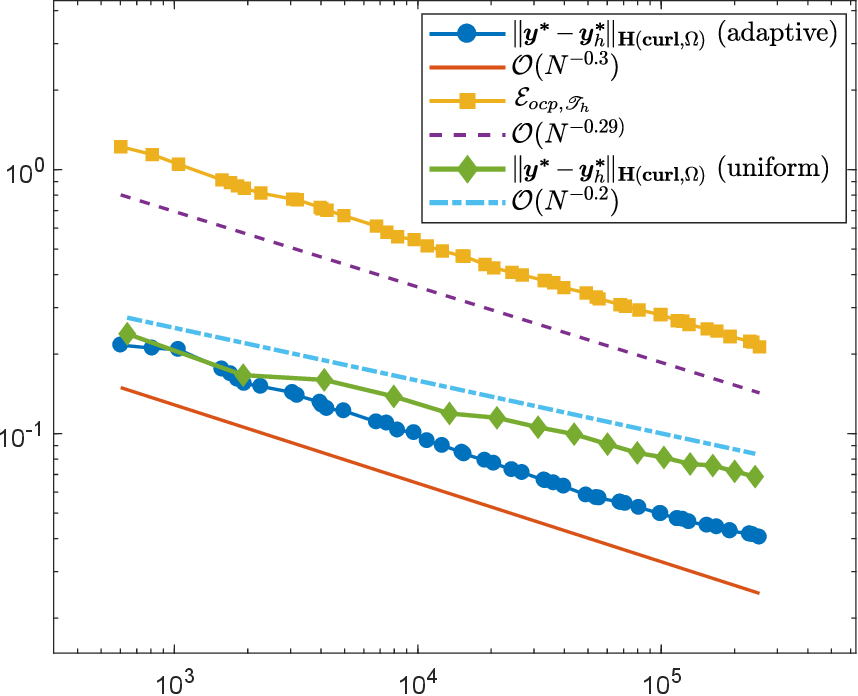}
	\end{minipage}\\
	\caption{Test 2. Left: Initial mesh for the L-shaped domain.  Right: Comparison between error curves for uniform and adaptive refinements,
together with computed values of estimator $\mathcal{E}_{ocp,\T_{h}}$. }
	\label{fig:initial_curves}
\end{figure}

%\begin{figure}[!h]
%	\centering
%	\includegraphics[width=0.5\linewidth]{L_mesh0.eps}
%	\caption{Test 2. Initial mesh for the 3D L-shaped domain.}
%	\label{fig:L-shape-initial-mesh}
%\end{figure}
%
%
%\begin{figure}[!h]
%	\centering
%	\includegraphics[width=0.8\linewidth]{L_convergence.eps}
%	\caption{Test 2. Comparison between error curves for uniform and adaptive refinements,
%together with computed values of the estimator $\mathcal{E}_{ocp}$.}
%	\label{fig:L_errors}
%\end{figure}

\begin{figure}[!h]
	\centering
	\begin{minipage}{0.49\linewidth}\centering
		\includegraphics[trim={0 0 0 0},clip,width=4.23cm,height=4.8cm,scale=0.8]{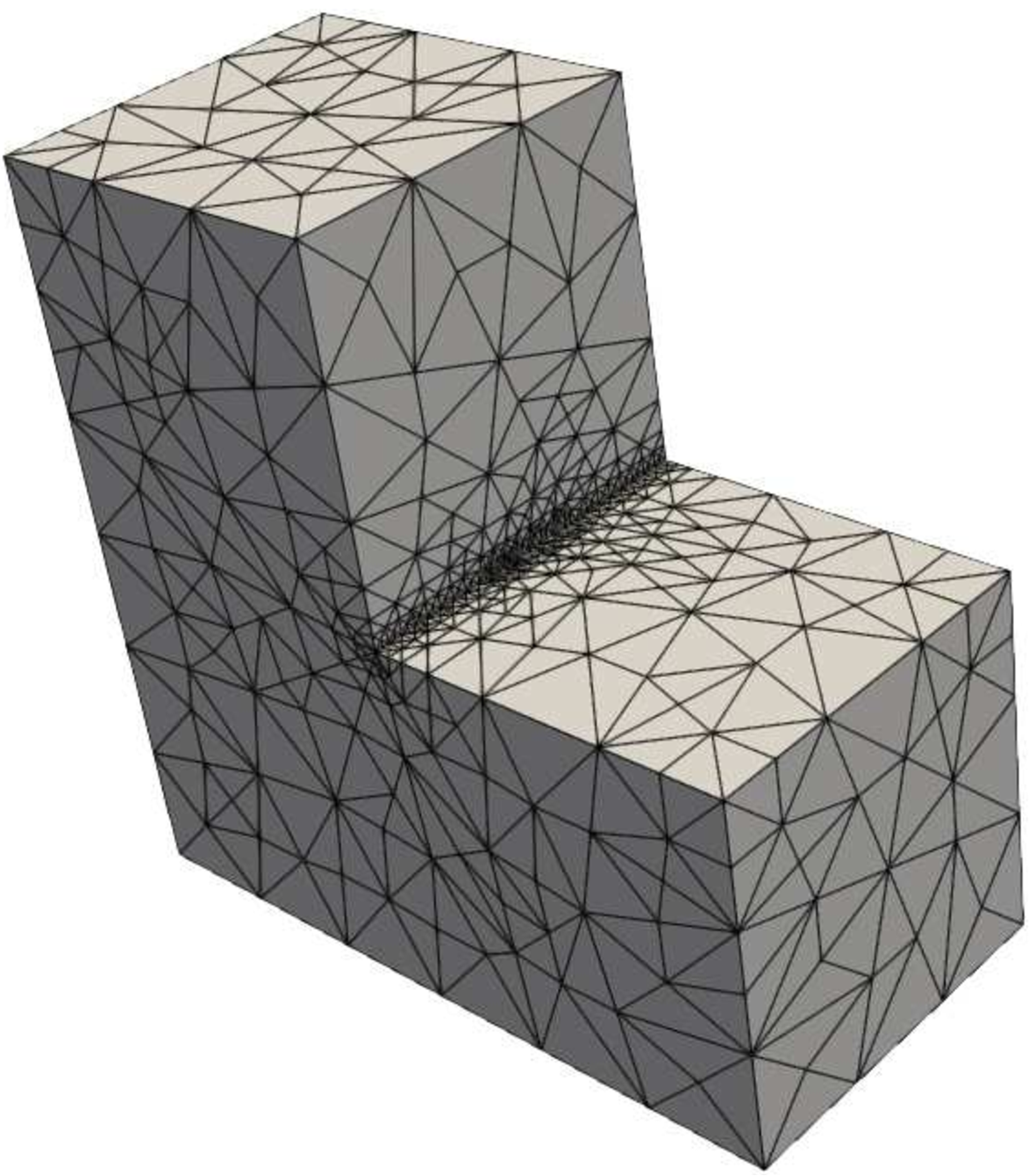}
	\end{minipage}
	\begin{minipage}{0.49\linewidth}\centering
		\includegraphics[trim={0 0 0 0},clip,width=4.23cm,height=4.8cm,scale=0.8]{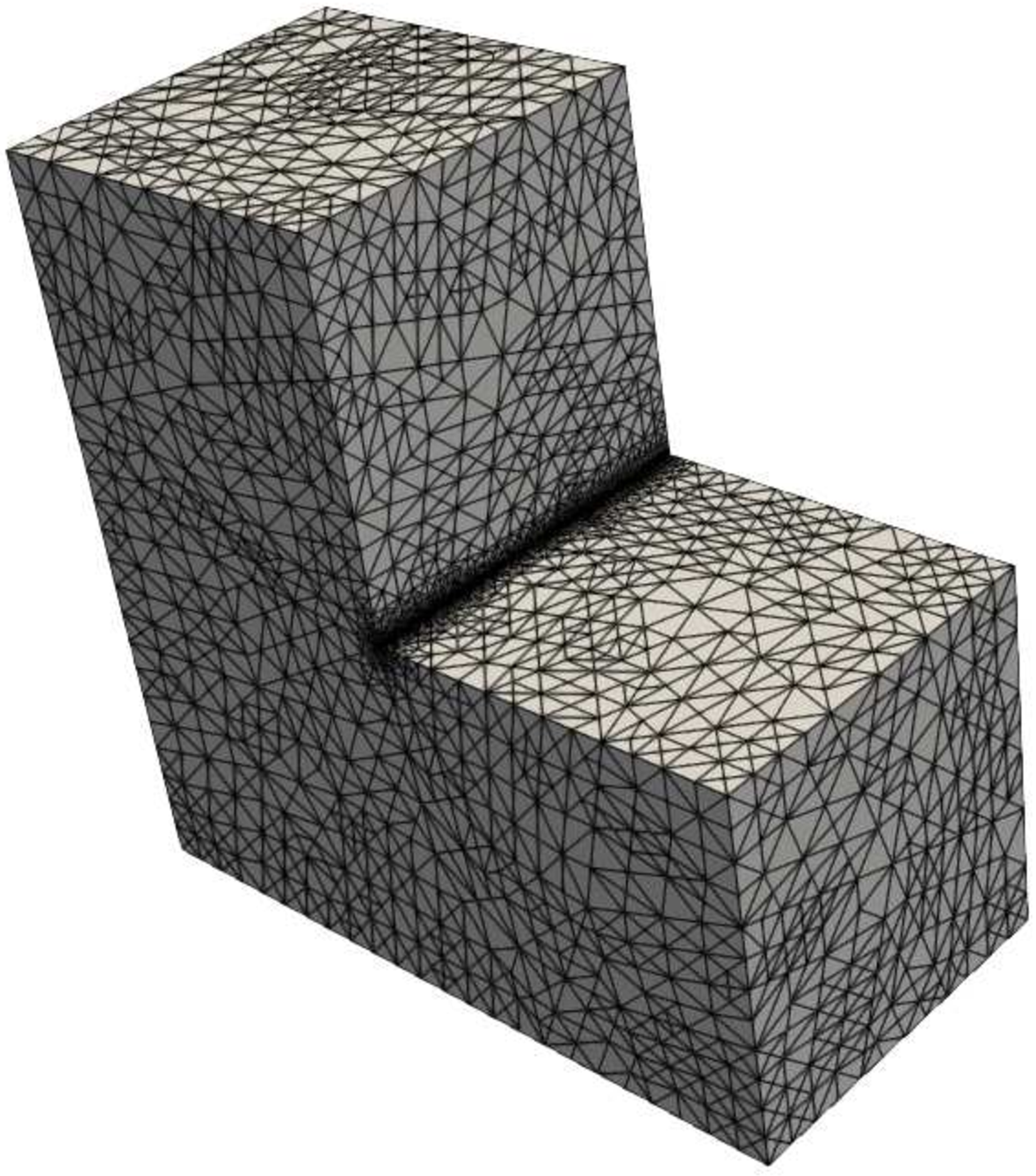}
	\end{minipage}\\
	\caption{Test 2. Intermediate adaptively refined meshes with $15408$ (left) and $263463$ (right) number of elements using the estimator 
$\mathcal{E}_{ocp,\T_{h}}$. }
	\label{fig:Lshape-mesh-estimador-theta}
\end{figure}

\subsection{Test 3. Discontinuous parameters and unknown solution}
This example is to further test the robustness of the adaptive algorithm in the case where discontinuous parameters are considered. 
More precisely,  we consider 
\begin{align*}
\chi(\boldsymbol{x}) = \begin{cases}
0.0001 & \text{if } \boldsymbol{x} \in \Omega_0, \\
1.0 &  \mbox{otherwise}
\end{cases}
\qquad
{\kappa}(\boldsymbol{x})
=\kappa_1(\boldsymbol{x})+\kappa_2(\boldsymbol{x}) 
=\boldsymbol{1_{\Omega_0}}+100\times\boldsymbol{1_{\Omega_1}}.
\end{align*}
Here, $\boldsymbol{1_{\Omega_0}}$, $\boldsymbol{1_{\Omega_1}}$ denote the characteristic functions of $\Omega_0,\Omega_1 \subset \Omega$ defined by
\begin{align*}
\Omega_0 := \left\{\boldsymbol{x}=(x,y,z)\in \Omega: 
\max\{|x-0.5|,|y-0.5|,|z-0.5|\}<	0.25\right\},
\end{align*} 
and $\Omega_1:=\overline{\Omega}_0^c\cap \Omega$, respectively; the computational domain is $\Omega:=(0,1)^3$. 
We choose as  data $\mathbf{a}=(0.1, 0.1)$, $\mathbf{b}=(100,100)$, $\alpha=1$, and 
\begin{align*}
\by_\Omega(\boldsymbol{x}) = (x^2\sin(\pi y)\sin(\pi z),
                 \sin(\pi x)\sin(\pi z),
		         \sin(\pi x)\sin(\pi y)),
\quad
\bbf(\boldsymbol{x}) = (1, 0, 0).
\end{align*}
%\begin{align*}
%\by_\Omega = \begin{pmatrix}
%        x^2\sin(\pi y)\sin(\pi z)\\
%         \sin(\pi x)\sin(\pi z)\\
%         \sin(\pi x)\sin(\pi y)
%\end{pmatrix},
%\qquad
%\bbf = \begin{pmatrix}
%        1.0\\
%         0.0\\
%         0.0
%\end{pmatrix}.
%\end{align*}
In contrast to the previous examples, the solution of this problem cannot be described analytically. 
Moreover, due to the discontinuities of the parameters, a smooth solution cannot be expected and may exhibit pronounced singularities. 
%To address this challenge, the proposed adaptive finite element scheme  can help us to predict the behavior of the unknown solution  and to capture its local singularities.

Figure~\ref{fig:2_domains} illustrates the adaptive meshes generated by \textbf{Algorithm}~\ref{Algorithm}. % for 1626796 number of elements. 
Note that the adaptive refinement is concentrated on the boundary of $\Omega_0$, 
which is where the parameter discontinuity takes place.
In Figure~\ref{fig:2_domains_field} (left), we show the approximate solution on the finest adaptively refined mesh, where we observe that the solution primarily concentrates on $\Omega_0$ and its magnitude decreases outside this region.
In the absence of an exact solution, we employ the error estimators $\mathcal{E}_{st,\T_{h}}$ and $\mathcal{E}_{adj,\T_{h}}$ to evaluate the convergence of the adaptive method. 
Figure~\ref{fig:2_domains_field} (right) shows the convergence history for $\mathcal{E}_{st,\T_{h}}$ and $\mathcal{E}_{ad,\T_{h}}$, computed with uniform and adaptive refinement. 
From this figure we observe a convergence behavior of both estimators towards zero for increasing number of elements of the mesh. 
Notably, the adaptive method achieves significantly superior numerical performance. 
We also observe a lower order of convergence for the estimators compared to the previous example. 
This is expected due to the poor regularity and the non-smoothness detected in the solution. 

\begin{figure}[!h]
	\centering
	\begin{minipage}{0.45\linewidth}\centering
		\includegraphics[trim={0 0 0 0},clip,width=4.2cm,height=4.7cm,scale=0.8]{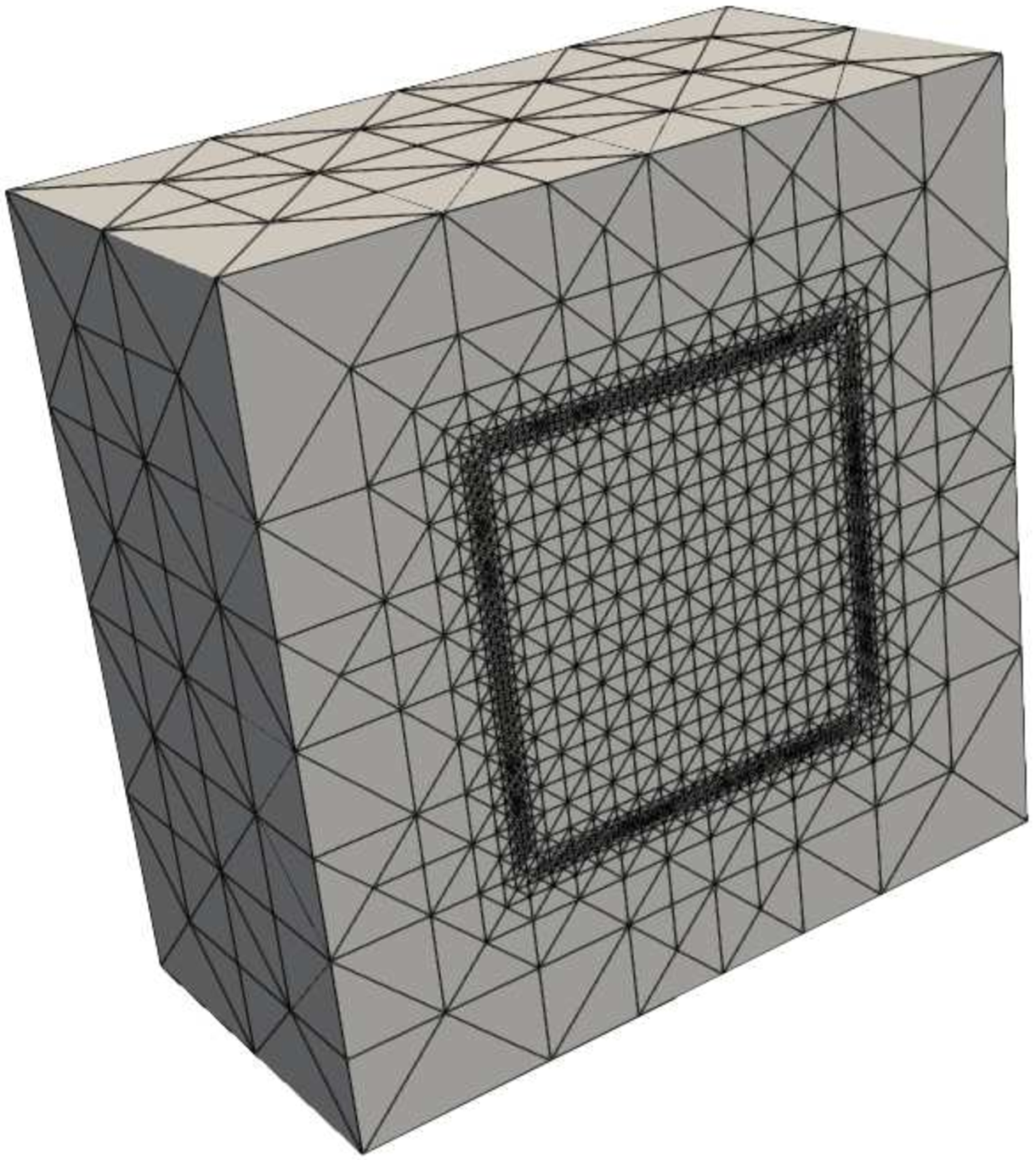}
	\end{minipage}
	\begin{minipage}{0.5\linewidth}\centering
		\includegraphics[trim={0 0 0 0},clip,width=4.95cm,height=4.7cm,scale=0.8]{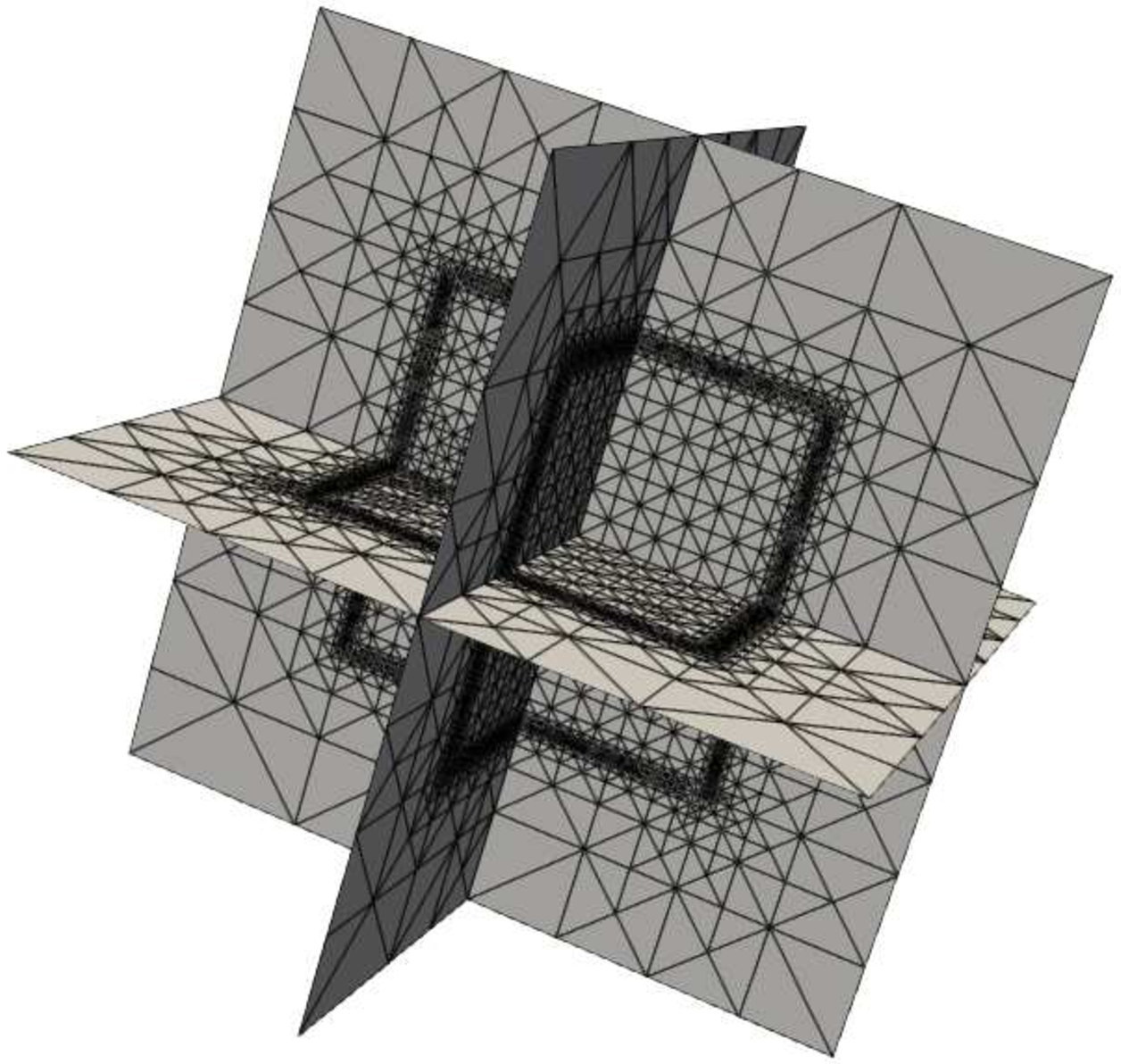}
	\end{minipage}\\
	\caption{Test 3. Adaptively refined mesh  with 1626796 number of elements and the corresponding cross sections of the mesh.}
	\label{fig:2_domains}
\end{figure}
\begin{figure}[!h]
	\centering
	\begin{minipage}{0.45\linewidth}\centering
		\includegraphics[trim={0 0 0 0},clip,width=5.0cm,height=4.7cm,scale=0.8]{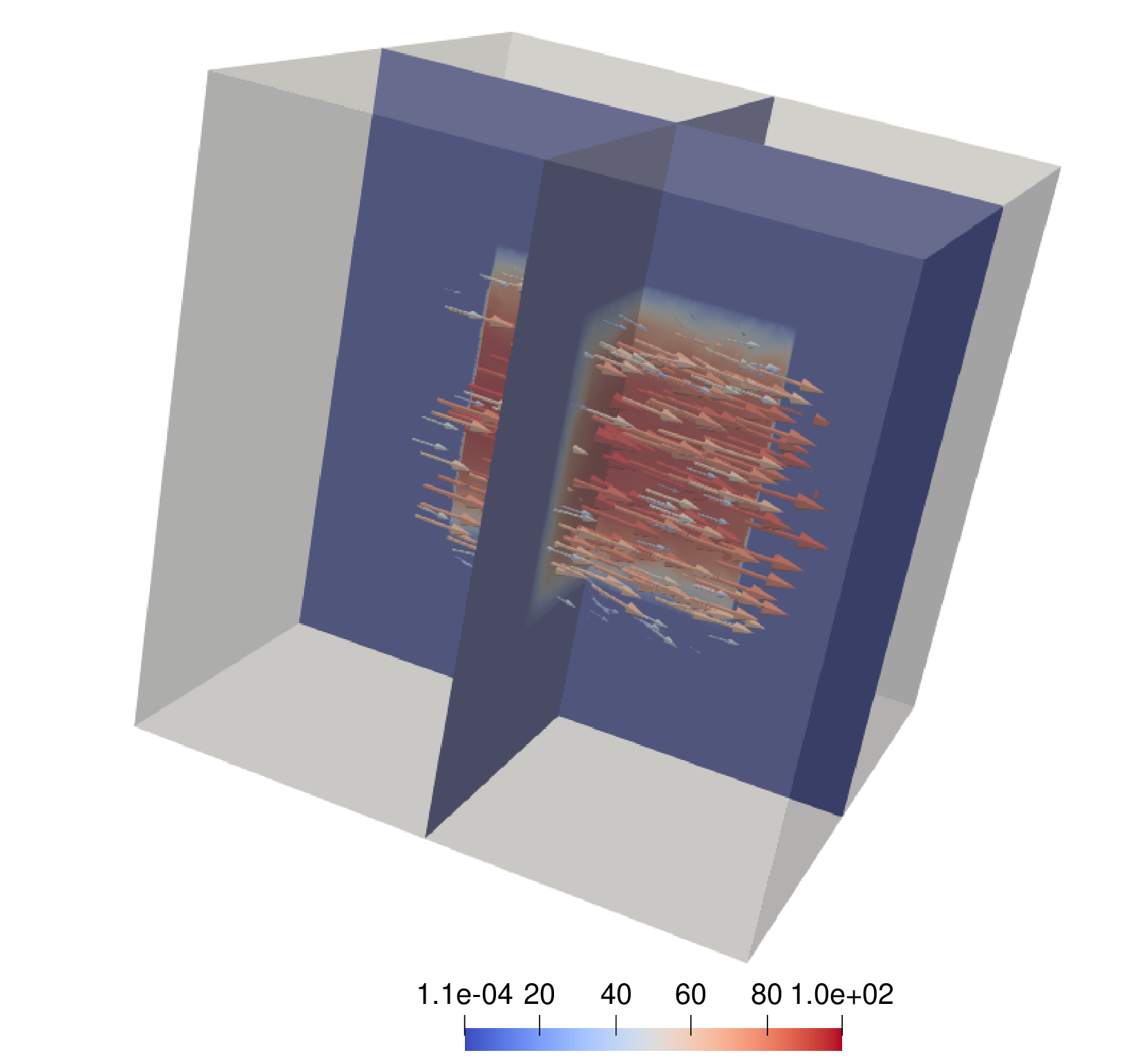}
	\end{minipage}
	\begin{minipage}{0.5\linewidth}\centering
		\includegraphics[trim={0 0 0 0},clip,width=5.5cm,height=4.1cm,scale=0.8]{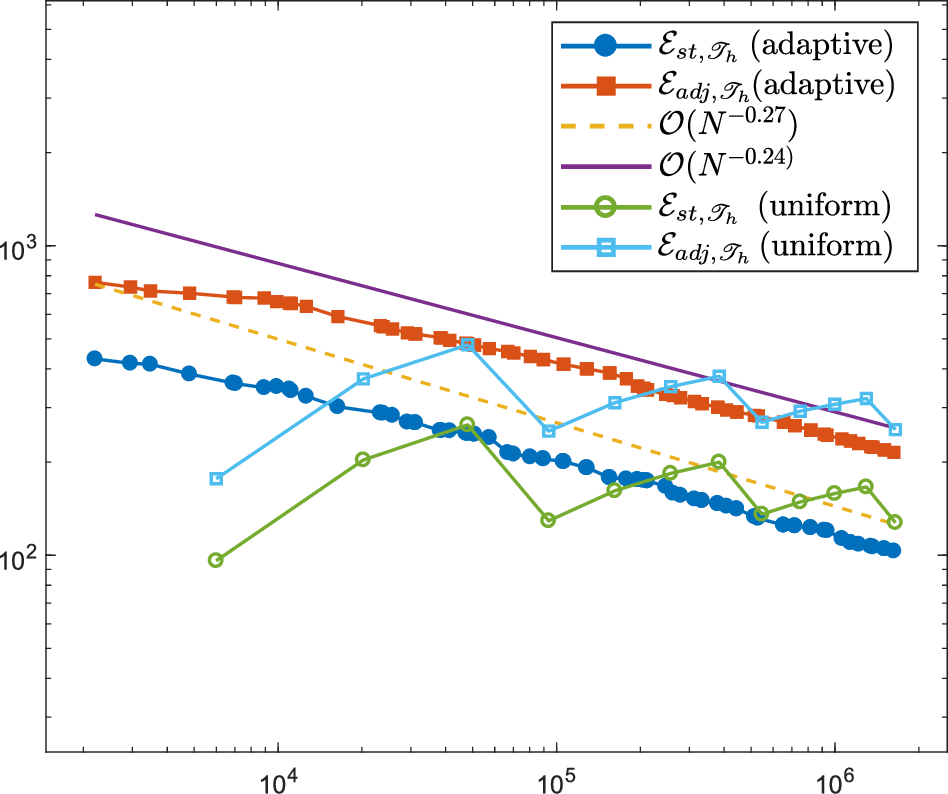}
	\end{minipage}\\
	\caption{Test 3. Left: Numerical solution $\by^{*}_h$ (magnitude and vector field) computed on an adaptively refined mesh with 1626796 number 
of elements.  Right: Comparison between the convergence of the estimators  
$\mathcal{E}_{st,\T_{h}}$ and  $\mathcal{E}_{ad,\T_{h}}$ with uniform and adaptive refinement.}
	\label{fig:2_domains_field}
\end{figure}

%
%-------    References    --------------------------------------
%

\bibliographystyle{siam}
\bibliography{Ref}

\end{document}